\documentclass[11pt,letterpaper]{article}

\usepackage{amssymb,amsmath,amsthm,cite}
\usepackage{fullpage,epstopdf,subcaption}
\usepackage[margin=1in]{geometry}
\usepackage{algorithm}
\usepackage[pdftex]{graphicx}
\usepackage{graphicx}
\usepackage{color}

\newtheorem{theorem}{Theorem}[section]
\newtheorem{lemma}[theorem]{Lemma}
\newtheorem{proposition}[theorem]{Proposition}

\newtheorem{corollary}[theorem]{Corollary}

\theoremstyle{definition}
\newtheorem*{definition}{Definition}

\theoremstyle{remark}
\newtheorem*{remark}{Remark}
\newtheorem*{example}{Example}

\renewcommand{\intercal}{T}

\newif\iftbd
\tbdfalse

\newif\ifgfx
\gfxtrue

\newcommand{\ba}{a}
\newcommand{\bb}{b}
\newcommand{\bc}{c}
\newcommand{\be}{e}
\newcommand{\bg}{g}

\newcommand{\bt}{t}
\newcommand{\bu}{u}
\newcommand{\bv}{v}
\newcommand{\bw}{w}
\newcommand{\bx}{x}
\newcommand{\by}{y}
\newcommand{\bz}{z}
\newcommand{\bone}{1}
\newcommand{\bzero}{0}

\newcommand{\mc}{M_{\K^{\star},\C}}
\newcommand{\pk}{P_{\K^{\star}}}
\newcommand{\pnk}{P_{n,\K^{\star}}}

\newcommand{\pc}{\Phi_{\C}}

\newcommand{\aut}{\mathrm{Aut}}
\newcommand{\orb}{\cdot \mathrm{Aut}(\C)}
\newcommand{\hh}{H}
\newcommand{\lqd}{L(\R^q,\R^d)}

\newcommand{\E}{\mathbb{E}}
\newcommand{\prob}{\mathbb{P}}
\newcommand{\C}{C}
\newcommand{\K}{K}
\newcommand{\KS}{\K^{\star}}
\newcommand{\KH}{\hat{\K}}
\newcommand{\KHN}{\hat{\K}_{n}^{\C}}
\newcommand{\R}{\mathbb{R}}
\newcommand{\Sym}{\mathbb{S}}
\newcommand{\rem}{\lambda_{\C}}
\newcommand{\nspace}{T^{\perp}}

\newcommand{\nc}{N}
\newcommand{\entr}{\eta}

\newcommand{\aff}{\mathrm{aff}}
\newcommand{\spn}{\mathrm{Span}}

\newcommand{\simp}{\Delta}
\newcommand{\fs}{\mathcal{O}}
\newcommand{\sph}{S}
\newcommand{\grv}{G_{r,\bv}}
\newcommand{\grhv}{G_{\hat{r},\bv}}

\newcommand{\lmax}{h_{\C}}
\newcommand{\emax}{\be_{\C}}

\title{Fitting Tractable Convex Sets to Support Function Evaluations}

\author{Yong Sheng Soh\thanks{Institute of High Performance Computing, 1 Fusionopolis Way, \# 16-16 Connexis, Singapore 138632. \texttt{soh\_yong\_sheng@ihpc.a-star.edu.sg}} \hspace{0.25in} Venkat Chandrasekaran\thanks{Department of Computing and Mathematical Sciences and Department of Electrical Engineering, California Institute of Technology, Pasadena, CA 91125, USA. \texttt{venkatc@caltech.edu} \newline The authors were supported in part by NSF grants CCF-1350590 and CCF-1637598, by Air Force Office of Scientific Research Grant FA9550-16-1-0210, by a Sloan research fellowship, and an A*STAR (Agency for Science, Technology and Research, Singapore) fellowship.}}

\date{March 24, 2019; revised October 23, 2019}

\begin{document}

\maketitle

\begin{abstract}
The geometric problem of estimating an unknown compact convex set from evaluations of its support function arises in a range of scientific and engineering applications.  Traditional approaches typically rely on estimators that minimize the error over all possible compact convex sets; in particular, these methods allow for limited incorporation of prior structural information about the underlying set and the resulting estimates become increasingly more complicated to describe as the number of measurements available grows.  We address both of these shortcomings by describing a framework for estimating \emph{tractably specified} convex sets from support function evaluations.  Building on the literature in convex optimization, our approach is based on estimators that minimize the error over structured families of convex sets that are specified as linear images of concisely described sets -- such as the simplex or the spectraplex -- in a higher-dimensional space that is not much larger than the ambient space.  Convex sets parametrized in this manner are significant from a computational perspective as one can optimize linear functionals over such sets efficiently; they serve a different purpose in the inferential context of the present paper, namely, that of incorporating \emph{regularization} in the reconstruction while still offering considerable expressive power.  We provide a geometric characterization of the asymptotic behavior of our estimators, and our analysis relies on the property that certain sets which admit semialgebraic descriptions are Vapnik-Chervonenkis (VC) classes.  Our numerical experiments highlight the utility of our framework over previous approaches in settings in which the measurements available are noisy or small in number as well as those in which the underlying set to be reconstructed is non-polyhedral.

\vspace{0.025in}

\end{abstract}


\section{Introduction}
We consider the problem of estimating a compact convex set given (possibly noisy) evaluations of its support function.  Formally, let $\K^{\star} \subset \R^d$ be a set that is compact and convex.  The support function $h_{\K^{\star}}(\bu)$ of the set $\K^{\star}$ evaluated in the direction $\bu \in \sph^{d-1}$ is defined as:
\begin{equation*}
	h_{\K^{\star}} (\bu) := \sup_{\bx \in \K^{\star} } \langle \bx, \bu \rangle.
\end{equation*}
Here $\sph^{d-1} := \{ \bx ~|~ \| \bx \|_2 = 1 \} \subset \R^d$ denotes the $(d-1)$-dimensional unit sphere.  In words, the quantity $h_{\K^{\star}} (\bu)$ measures the maximum displacement in the direction $\bu$ intersecting $\K^{\star}$.  Given noisy support function evaluations $\left\{ (\bu^{(i)}, y^{(i)} ) ~|~ \allowbreak y^{(i)} \allowbreak = h_{\K^{\star}} (\bu^{(i)}) \allowbreak + \varepsilon^{(i)}, 1 \leq i \leq n \right\}$, where each $\varepsilon^{(i)}$ denotes additive noise, our goal is to reconstruct a convex set $\hat{\K}$ that is close to $\K^{\star}$.

The problem of estimating a convex set from support function evaluations arises in a wide range of problems such as computed tomography \cite{PriWil:90}, target reconstruction from laser-radar measurements \cite{LKW:92}, and projection magnetic resonance imaging \cite{GreRan:02}.  For example, in tomography the extent of the absorption of parallel rays projected onto an object provides support information \cite{PriWil:90,StaPen:88}, while in robotics applications support information can be obtained from an arm clamping onto an object in different orientations \cite{PriWil:90}.  A natural approach to fit a compact convex set to support function data is the following least-squares estimator (LSE):
\begin{equation} \label{eq:sc_intro_lse}
\hat{\K}^{\mathrm{LSE}}_n \in \underset{\K \subset \R^d : \K \text { is compact, convex}}{\mathrm{argmin}} ~~~ \frac{1}{n} \sum_{i=1}^{n} \left( y^{(i)} - h_{\K} (\bu^{(i)}) \right)^{2}.
\end{equation}
An LSE always exists and it is not defined uniquely, although it is always possible to select a polytope that is an LSE; this is the choice that is most commonly employed and analyzed in prior work.  For example, the algorithm proposed by Prince and Willsky \cite{PriWil:90} for planar convex sets reconstructs a polygonal LSE described in terms of its facets, while the algorithm proposed by Gardner and Kiderlen \cite{GarKid:09} for convex sets in any dimension provides a polytopal LSE reconstruction described in terms of extreme points.  The least-squares estimator $\hat{\K}^{\mathrm{LSE}}_n$ is a consistent estimator of $\K^\star$, but it has a number of significant drawbacks.  In particular, as the formulation \eqref{eq:sc_intro_lse} does not incorporate any additional structural information about $\K^\star$ beyond convexity, the estimator $\hat{\K}^{\mathrm{LSE}}_n$ can provide poor
reconstructions when the measurements available are noisy or small in number.  The situation is problematic even when the number of measurements available is large, as the complexity of the resulting estimate grows with the number of measurements in the absence of any regularization due to the regression problem \eqref{eq:sc_intro_lse} being nonparametric (the collection of all compact convex sets in $\R^d$ is not finitely parametrized); consequently, the facial structure of the reconstruction provides little information about the geometry of the underlying set.\footnote{We note that this is the case even though the estimator $\hat{\K}^{\mathrm{LSE}}_n$ is consistent; in particular, consistency simply refers to the convergence as $n \rightarrow \infty$ of $\hat{\K}^{\mathrm{LSE}}_n$ to $\K^\star$ in a topological sense (e.g., in Hausdorff distance) and it does not provide any information about the facial structure of $\hat{\K}^{\mathrm{LSE}}_n$ relative to that of $\K^\star$.} Finally, if the underlying set $\K^\star$ is not polyhedral, a polyhedral choice for the solution $\hat{\K}^{\mathrm{LSE}}_n$ (as is the case with much of the literature on this topic) can provide poor reconstructions.  Indeed, even for intrinsically ``simple'' convex bodies such as the Euclidean ball, one necessarily requires many vertices or facets to obtain accurate polyhedral approximations.  Figure \ref{fig:intro_comparison} provides an illustration of these points.

\subsection{Our Contributions}

To address the drawbacks underlying the least-squares estimator, we seek a framework that \emph{regularizes} for the complexity of the reconstruction in the formulation \eqref{eq:sc_intro_lse}.  A natural approach to developing such a framework is to design an estimator with the same objective as in \eqref{eq:sc_intro_lse} but in which the decision variable $\K$ is constrained to lie in a subclass $\mathcal{F}$ of the collection of all compact, convex sets.  For such a method to be useful, the subclass $\mathcal{F}$ must balance several considerations.  First, $\mathcal{F}$ should be sufficiently expressive in order to faithfully model various attributes of convex sets that arise in applications (for example, sets consisting of both smooth and singular features in their boundary).  Second, the elements of $\mathcal{F}$ should be suitably structured so that incorporating the constraint $\K \in \mathcal{F}$ leads to estimates that are more robust to noise; further, the type of structure underlying the sets in $\mathcal{F}$ also informs the analysis of the statistical properties of the constrained analog of \eqref{eq:sc_intro_lse} as well as the development of efficient algorithms for computing the associated estimate.  Building on the literature on lift-and-project methods in convex optimization \cite{Yan:91,GPT:13}, we consider families $\mathcal{F}$ in which the elements are specified as images under linear maps of a fixed `concisely specified' compact convex set; the choice of this set governs the expressivity of the family $\mathcal{F}$ and we discuss this in greater detail in the sequel.  Due to the availability of computationally tractable procedures for optimization over linear images of concisely described convex sets \cite{NesNem:94}, the study of such descriptions constitutes a significant topic in optimization.  We employ these ideas in a conceptually different context in the setting of the present paper, namely that of incorporating regularization in the reconstruction, which addresses many of the drawbacks with the LSE outlined previously.  Formally, we consider the following regularized convex set regression problem:
\begin{equation} \label{eq:sc_intro_constrainedlse}
\hat{\K}_n^{\C} \in \underset{\K : \K = A (\C), A \in L(\R^{q},\R^d)}{\mathrm{argmin}} \frac{1}{n} \sum_{i=1}^{n} \left( y^{(i)} - h_{\K} (\bu^{(i)}) \right)^2.
\end{equation}
Here $\C \subset \R^q$ is a user-specified compact convex set and $L(\R^{q},\R^d)$ denotes the set of linear maps from $\R^q$ to $\R^d$.  The set $\C$ governs the expressive power of the family $\{A(\C) ~|~ A \in L(\R^{q},\R^d)\}$.  In addition to this consideration, our choices for $\C$ are also driven by statistical and computational aspects of the estimator \eqref{eq:sc_intro_constrainedlse}. Some of our analysis of the statistical properties of the estimator \eqref{eq:sc_intro_constrainedlse} relies on the observation that sets $\mathcal{F}$ that admit certain semialgebraic descriptions form VC classes; this fact serves as the foundation for our characterization based on stochastic equicontinuity of the asymptotic properties of the estimator $\hat{\K}_n^{\C}$ as $n \rightarrow \infty$.  On the computational front, the algorithms we propose for \eqref{eq:sc_intro_constrainedlse} require that the support function associated to $\C$ as well as its derivatives (when they exist) can be computed efficiently.  Motivated by these issues, the choices of $\C$ that we primarily discuss in our examples and numerical illustrations are the simplex and the spectraplex:

\begin{example}
	The \emph{simplex} in $\R^q$ is the set:
	\begin{equation*}
		\simp^{q} := \{ \bx ~|~ \bx \in \R^q, \bx \geq 0, \langle \bx, \bone \rangle = 1 \} \quad \text{where} \quad \bone = (1,\ldots,1)^{\intercal}.
	\end{equation*}
	Convex sets expressed as projections of $\simp^{q}$ are precisely polytopes with at most $q$ extreme points.
\end{example}

\begin{example}
	Let $\mathbb{S}^p \cong \R^{p+1 \choose 2}$ denote the space of $p \times p$ real symmetric matrices.  The \emph{spectraplex} $\fs^{p} \subset \mathbb{S}^p$ (also called the \emph{free spectrahedron}) is the set:
	\begin{equation*}
		\fs^{p} := \{ X ~|~ X \in \mathbb{S}^p, X \succeq 0, \langle X, I \rangle = 1 \}, \quad \text{where} \quad I \in \mathbb{S}^p \text{ is the identity matrix}.
	\end{equation*}
	The spectraplex is a semidefinite analog of the simplex, and it is especially useful if we seek non-polyhedral reconstructions, as can be seen in Figure \ref{fig:intro_comparison} and in Section \ref{sec:numexp}; in particular, linear images of the spectraplex exhibit both smooth and singular features in their boundaries.
\end{example}

The specific selection of $\C$ from the families $\{\simp^{q}\}_{q=1}^\infty$ and $\{\fs^p\}_{p=1}^\infty$ is governed by the complexity of the reconstruction one seeks, which is typically based on prior information about $\K^\star$.  Our analysis in Section \ref{sec:stats} of the statistical properties of the estimator \eqref{eq:sc_intro_constrainedlse} relies on the availability of such additional knowledge about the complexity of $\K^\star$.  In practice in the absence of such information, cross-validation may be employed to obtain a suitable reconstruction; see Section \ref{sec:conc}.

In Section \ref{sec:prelims} we discuss preliminary aspects of our technical setup such as properties of the set of minimizers of the problem \eqref{eq:sc_intro_constrainedlse} as well as a stylized probabilistic model for noisy support function evaluations.  These serve as a basis for the subsequent development in the paper.  In Section \ref{sec:stats} we provide the main theoretical guarantees of our approach.  In our first result, we show that the sequence of estimates $\{\hat{\K}_n^{\C}\}_{n=1}^{\infty}$ converges almost surely (as $n \rightarrow \infty$) in the Hausdorff metric to that linear image of $\C$ which is closest to the underlying set $\K^{\star}$ (see Theorem \ref{thm:stats_setconvergence}).  Under additional conditions, we also characterize certain asymptotic distributional aspects of the sequence $\{\hat{\K}_n^{\C}\}_{n=1}^{\infty}$ (see Theorem \ref{thm:stats_norm}); this result is based on a functional central limit theorem, which requires the computation of appropriate entropy bounds for Vapnik-Chervonenkis (VC) classes of sets that admit semialgebraic descriptions, and it is here that our choice of $\C$ as either a simplex or a spectraplex plays a prominent role.  Our third result describes the facial structure of $\{\hat{\K}_n^{\C}\}_{n=1}^{\infty}$ in relation to the underlying set $\K^\star$.  We prove under appropriate conditions that if $\K^\star$ is a polytope our approach provides a reconstruction that recovers all the simplicial faces (for sufficiently large $n$); if $\K^\star$ is a simplicial polytope, we recover a polytope that is combinatorially equivalent to $\K^\star$.  This result also applies more generally to `rigid' faces for non-polyhedral $\KS$ (see Theorem \ref{thm:preserveface}).

In the sequel, we relate our formulation \eqref{eq:sc_intro_constrainedlse} (when $\C$ is a simplex) to the task of fitting piecewise affine convex functions to data (known as max-affine regression) as well as $K$-means clustering.  Accordingly, the algorithm we propose in Section~\ref{sec:algo} for computing $\hat{\K}_n^{\C}$ bears significant similarities with methods for max-affine regression \cite{MagBoy:09} as well as Lloyd's algorithm for clustering problems.

A restriction in the development in this paper is that the simplex and the spectraplex represent particular affine slices of the nonnegative orthant and the cone of positive semidefinite matrices.  In principle, one can further optimize these slices (both their orientation and their dimension) in \eqref{eq:sc_intro_constrainedlse} to obtain improved reconstructions.  However, this additional degree of flexibility in \eqref{eq:sc_intro_constrainedlse} leads to technical complications in establishing asymptotic normality in Section \ref{sec:stats_asymnormal} as well as to challenges in developing algorithms for solving \eqref{eq:sc_intro_constrainedlse} (even to obtain a local optimum).  The root of these difficulties lies in the fact that it is hard to characterize the variation in the support function with respect to small changes in the slice.  We remark on these challenges in greater detail in Section \ref{sec:conc}, and for the remainder of the paper we proceed with investigating the estimator \eqref{eq:sc_intro_constrainedlse}.

\begin{figure}
	\centering
	\begin{subfigure}{.45\textwidth}
		\centering
		\begin{subfigure}{.5\textwidth}
			\centering
			\includegraphics[width=0.8\textwidth]{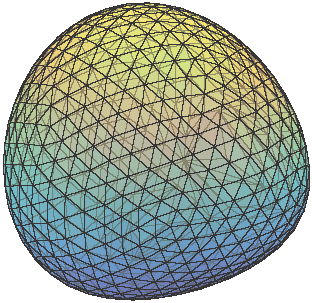}
		\end{subfigure}%
		\begin{subfigure}{.5\textwidth}
			\centering
			\includegraphics[width=0.8\textwidth]{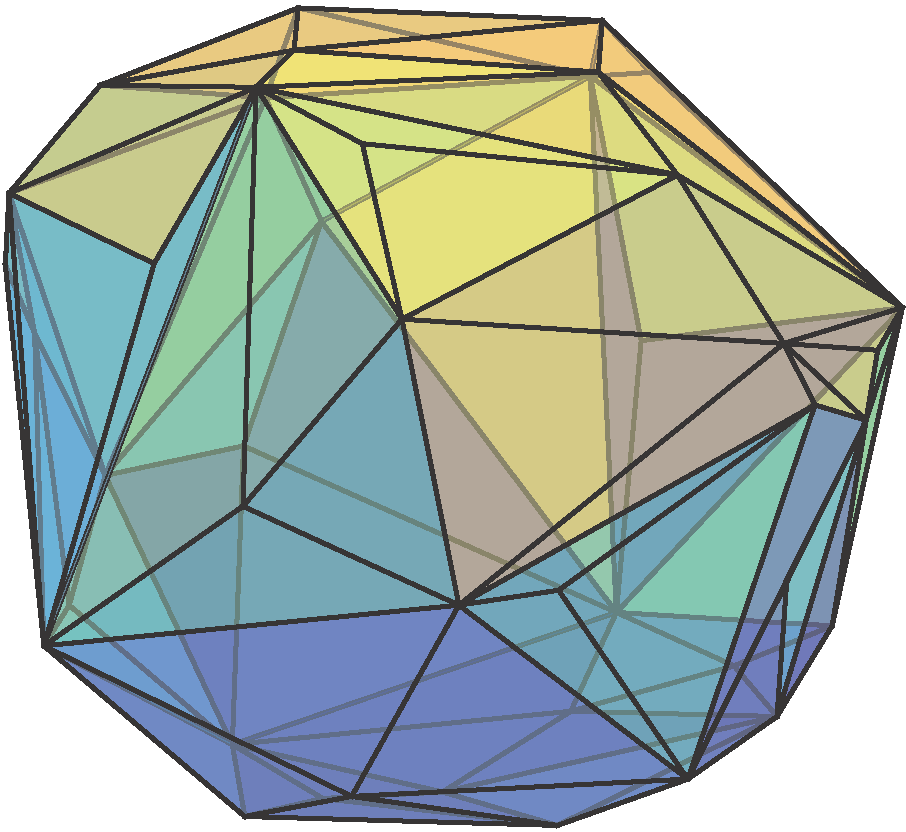}
		\end{subfigure}%
		\caption{Reconstructions of the unit $\ell_2$-ball from $50$ noisy support function measurements as the projection of $\fs^{3}$ (our approach, left), and the LSE (right).}
	\end{subfigure}%
	\hspace{0.5cm}
	\begin{subfigure}{.45\textwidth}
		\centering
		\begin{subfigure}{.5\textwidth}
			\centering
			\includegraphics[width=0.5\textwidth]{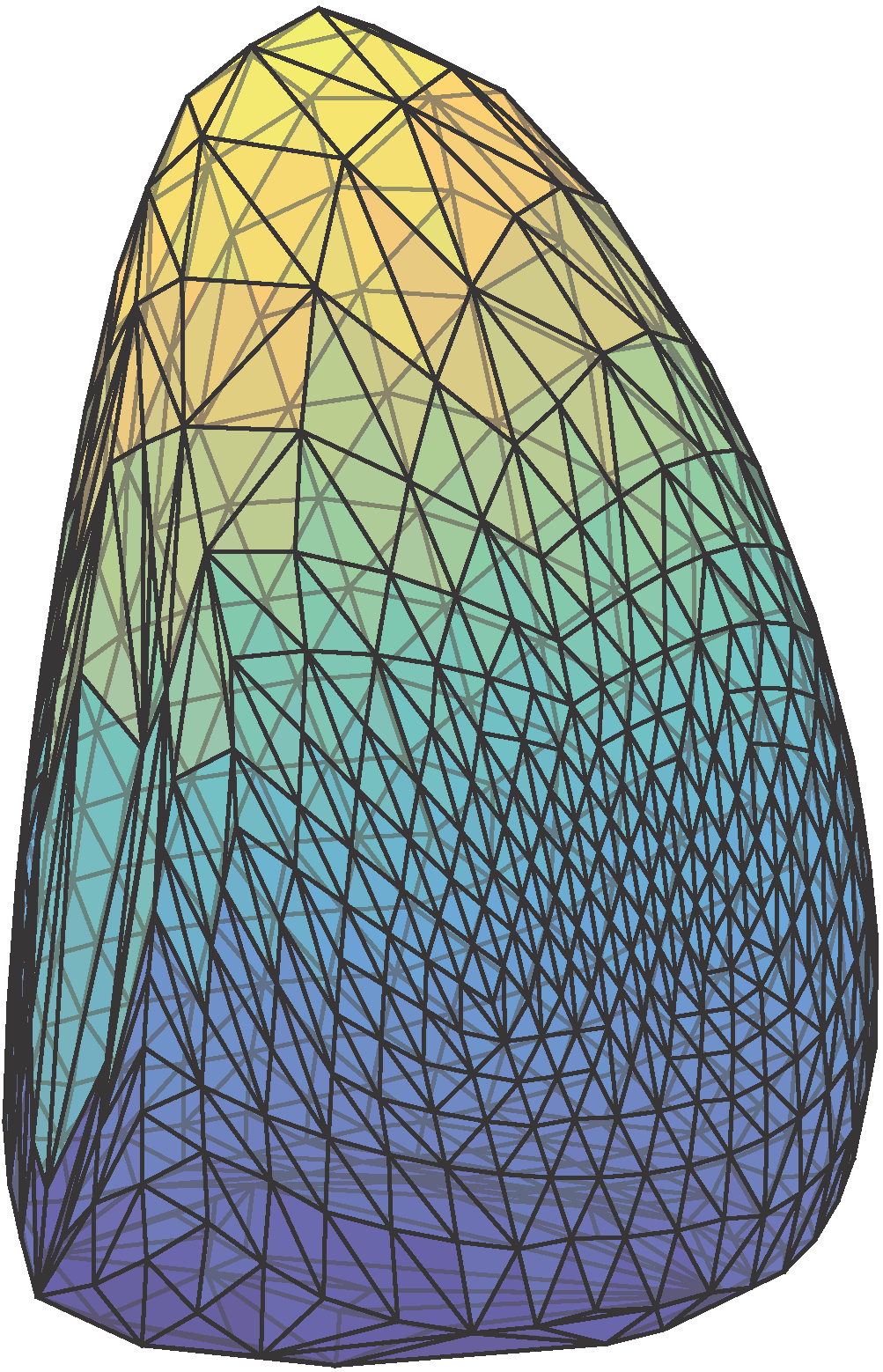}
		\end{subfigure}%
		\begin{subfigure}{.5\textwidth}
			\centering
			\includegraphics[width=0.5\textwidth]{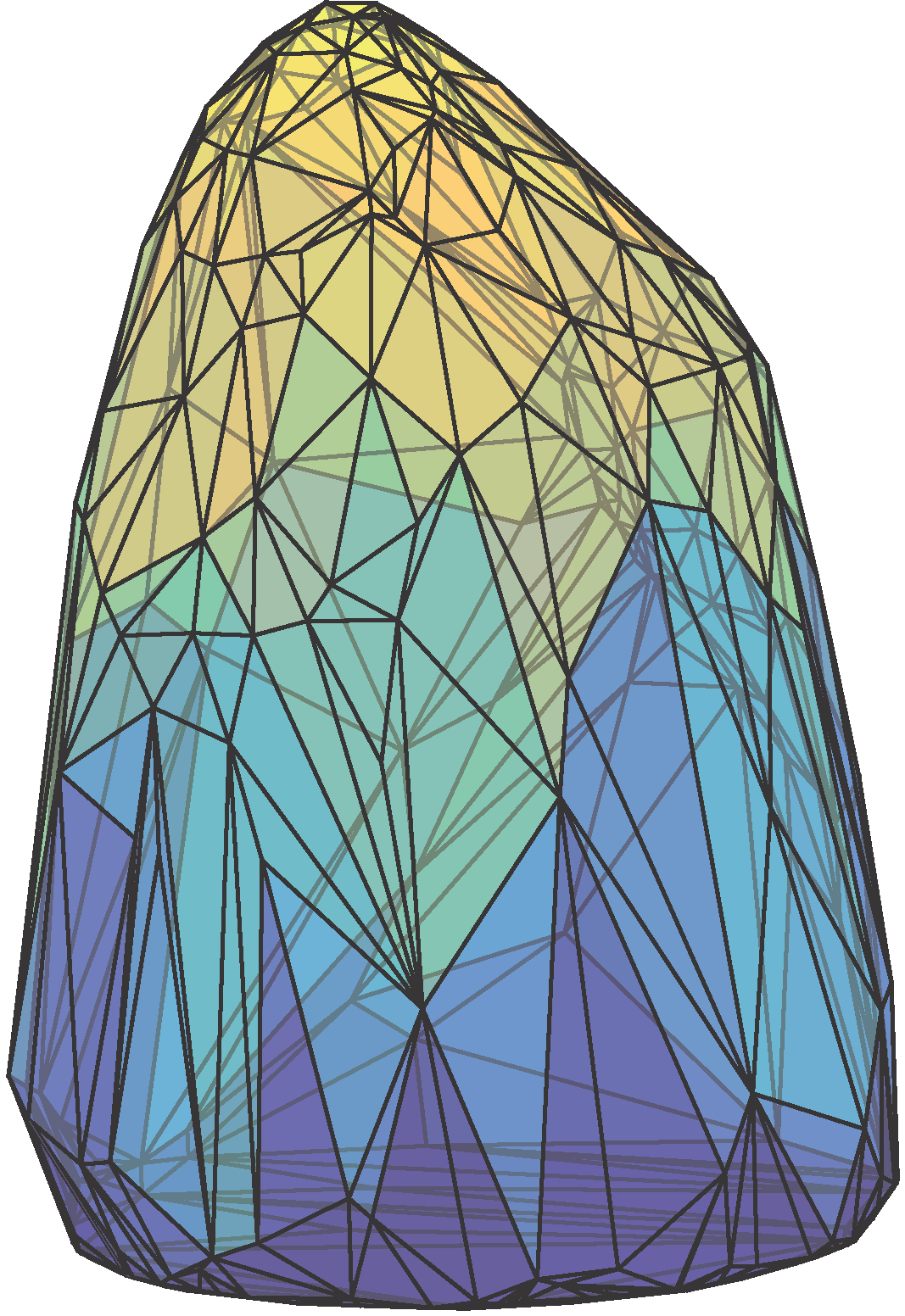}
		\end{subfigure}%
		\caption{Reconstructions of the convex mesh of a human lung from $300$ noiseless support function measurements as the projection of $\fs^{6}$ (our approach, left), and the LSE (right).}
	\end{subfigure}%
	\caption{Comparison between our approach and the LSE.}
	\label{fig:intro_comparison}
\end{figure}

\subsection{Related Work}


\subsubsection{Consistency of Convex Set Regression}
There is a well-developed body of prior work on the consistency of convex set regression \eqref{eq:sc_intro_lse}.  Gardner et al. \cite{GarKidMil:06} prove that the (polyhedral) estimates $\hat{\K}^{\mathrm{LSE}}_n$ converge almost surely to the underlying set $\K^\star$ in the Hausdorff metric as $n \rightarrow \infty$ provided the directions $\{u^{(i)}\}_{i=1}^{n}$ cover the sphere in a suitably uniform manner.  Guntuboyina \cite{Guntuboyina:12} analyzes rates of convergence in minimax settings, and also notes that constraining the growth of the number of vertices in the reconstruction as the number of measurements increases provides a form of robustness.  Cai et. al. \cite{CGW:18} study the impact of choosing the directions $\{u^{(i)}\}_{i=1}^{n}$ adaptively in estimating planar convex sets.
In contrast the consistency result in the present paper corresponding to the constrained estimator \eqref{eq:sc_intro_constrainedlse} is qualitatively different.  On the one hand, for a given compact convex set $\C \subset \R^q$, we prove that the sequence of estimates $\{\hat{\K}_n^{\C}\}_{n=1}^{\infty}$ converges to that linear image of $\C$ which is closest to the underlying set $\K^{\star}$; in particular, $\{\hat{\K}_n^{\C}\}_{n=1}^{\infty}$ only converges to $\K^\star$ if $\K^\star$ can be represented as a linear image of $\C$.  On the other hand, there are several advantages to the framework presented in this paper in comparison with prior work.  First, the constrained estimator \eqref{eq:sc_intro_constrainedlse} lends itself to a precise asymptotic distributional characterization that is unavailable in the unconstrained case \eqref{eq:sc_intro_lse}.  Second, under appropriate conditions, the constrained estimator \eqref{eq:sc_intro_constrainedlse} recovers the facial structure of the underlying set $\K^\star$ unlike $\hat{\K}_{\mathrm{LSE}}$.  More significantly, beyond these technical distinctions, the constrained estimator \eqref{eq:sc_intro_constrainedlse} also yields concisely described non-polyhedral reconstructions (as well as associated consistency and asymptotic distributional characterizations) based on linear images of the spectraplex, in contrast to the usual choice of a polyhedral LSE in the previous literature.

\subsubsection{Incorporating Prior Information and Fitting Smooth Boundaries}
The problem of integrating prior information about the underlying convex set has been considered in \cite{PriWil:91}, where the authors propose a method to incorporating certain structural or shape priors for fitting convex sets in two dimensions, with a particular focus on settings in which the underlying set is a disc or an ellipsoid.  However, the reconstructions produced by the method in \cite{PriWil:91} are still polyhedral, and the method assumes that the support function evaluations are available at angles that are equally spaced.
We are also aware of a line of work \cite{FHTW:97,HalTur:99} on fitting convex sets in two dimensions with smooth boundaries to support function measurements.  The first of these papers estimates a convex set with a smooth boundary without any vertices, while the second proposes a two-step method in which one initially estimates a set of vertices followed by a second step that connects these vertices via smooth boundaries.  In both cases, splines are used to interpolate between the support function evaluations with a subsequent smoothing procedure using the von Mises kernel.  The smoothing is done in a local fashion and the resulting reconstruction is increasingly complex to describe as the number of measurements available grows.  In contrast, our approach to producing non-polyhedral estimates based on fitting linear images of spectraplices is more global in nature, and we explicitly regularize for the complexity of our reconstruction based on the dimension of the spectraplex.  Further, the approaches proposed in \cite{FHTW:97,HalTur:99} estimate the singular and the smooth parts of the boundary separately, whereas our framework based on linear images of spectraplices estimates these features in a unified manner (for example, see the illustration in Figure \ref{fig:Exp_F6}).  Finally, the methods described in \cite{FHTW:97,HalTur:99,PriWil:91} are only applicable to two-dimensional reconstruction problems, while problems of a three-dimensional nature arise in many contexts (see Section \ref{sec:numexp_lung} for an example that involves the reconstruction of a human lung).

\subsubsection{Piecewise Affine Convex Regression} \label{sec:maxaffineconnection}
The formulation \eqref{eq:sc_intro_constrainedlse} when $\C = \simp^q$ may be viewed as a fitting a piecewise linear function (with at most $q$ pieces) to the given data.  This is a special case of the \emph{max-affine regression} problem in which one is interested in fitting a piecewise affine function (typically with a bound on the number of pieces) to data, which is a topic that has been studied previously \cite{MagBoy:09,HanDun:13,Bal:16}.  In particular, our algorithm in Section~\ref{sec:algo} when specialized to the setting $\C = \simp^q$ is analogous to the methods described in \cite{MagBoy:09}.  However, our framework and the algorithm in Section~\ref{sec:algo} may also be employed to fit more general convex functions that are not piecewise linear, but that can still be specified in a tractable manner via linear images of the spectraplex.

\subsection{Outline}

In Section \ref{sec:prelims} we discuss the geometric, algebraic, and analytic aspects of the optimization problem \eqref{eq:sc_intro_constrainedlse}; this section serves as the foundation for the subsequent statistical analysis in Section \ref{sec:stats}.  Throughout both of these sections, we give several examples that provide additional insight into our mathematical development.  We describe algorithms for solving \eqref{eq:sc_intro_constrainedlse} in Section \ref{sec:algo}, and we demonstrate the application of these methods in a range of numerical experiments in Section \ref{sec:numexp}.  We conclude with a discussion of future directions in Section \ref{sec:conc}.  

\emph{Notation}:  Given a convex set $\C\subset \R^q$, we denote the associated induced norm by $\| A \|_{\C,2}  := \sup_{\bx\in\C} \| A \bx \|_2$.  We denote the unit $\|\cdot\|$-ball centered at $\bx$ by $B_{\|\cdot\|}(\bx) := \{ \by ~|~ \|\by-\bx\| \leq 1 \}$, and we denote the Frobenius norm by $\|\cdot\|_F$.  Given a point $\ba \in \R^q$ and a subset $U \subseteq \R^q$, we define $\mathrm{dist} (\ba,U) := \inf_{\bb \in U} \| \ba - \bb \|$, where the norm $\|\cdot\|$ is the Euclidean norm.  Last, given any two subsets $U,V \subset \R^q$, the Hausdorff distance between $U$ and $V$ is denoted by $d_{H}(U,V) := \inf_{t \geq 0} \{t ~|~  U \subseteq V + t B_{\|\cdot\|_{2}}(\bzero), V \subseteq U + t B_{\|\cdot\|_{2}}(\bzero) \}$.


\section{Problem Setup and Other Preliminaries} \label{sec:prelims}

In this section, we begin with a preliminary discussion of the geometric, algebraic, and analytical aspects of our procedure \eqref{eq:sc_intro_constrainedlse}; these underpin our subsequent development in this paper.  We make the following assumptions about our problem setup for the remainder of the paper:
\begin{itemize}
	\item[(A1)] The set $\KS \subset \R^d$ is compact and convex.
	\item[(A2)] The set $\C \subset \R^q$ is compact and convex.
	\item[(A3)] \emph{Probabilistic Model for Support Function Measurements}:  We assume that we are given $n$ independent and identically distributed support function evaluations $\{(\bu^{(i)},y^{(i)})\}_{i=1}^n \subset \sph^{d-1} \times \R$ from the following probabilistic model:
	\begin{equation} \label{eq:probdist}
	\pk: \quad y = h_{\KS} (\bu) + \varepsilon.
	\end{equation}
	Here $\bu \in \sph^{d-1}$ is a vector distributed uniformly at random (u.a.r.) over the unit sphere, $\varepsilon$ is a centered random variable with variance $\sigma^2$ (i.e., $\mathbb{E}[\varepsilon] = 0$, and $\mathbb{E}[\varepsilon^2] = \sigma^2$), and $\bu$ and $\varepsilon$ are independent.
\end{itemize}

In our analysis, we quantify dissimilarity between convex sets in terms of a metric applied to their respective support functions.  Let $\K_1, \K_2$ be compact convex sets in $\R^d$, and let $h_{\K_1}(\cdot), h_{\K_2}(\cdot)$ be the corresponding support functions.  We define the $L_p$ metric to be
\begin{equation} \label{eq:sc_formulation_distance}
\rho_{p} (\K_1, \K_2) := \left(\int_{\sph^{d-1}}  | h_{\K_1}(\bu) - h_{\K_2}(\bu) |^p ~ \mu (d \bu) \right)^{1/p},  \quad 1\leq p < \infty,
\end{equation}
where the integral is with respect to the Lebesgue measure over $\sph^{d-1}$; as usual, we denote $\rho_{\infty} (\K_1, \K_2) = \max_{\bu} | h_{\K_1}(\bu) - h_{\K_2}(\bu) |$.  We prove our convergence guarantees in Section \ref{sec:stats_sc} in terms of the $\rho_p$-metric.  This metric represents an important class of distance measures over convex sets.  For instance, it features prominently in the literature on approximating convex sets as polytopes \cite{Bronstein08}.  In addition, the specific case of $p=\infty$ coincides with the Hausdorff distance \cite{Sch:93} [p.~66].

Due to the form of the estimator \eqref{eq:sc_intro_constrainedlse}, one may reparametrize the optimization problem in terms of the linear map $A$.  In particular, by noting that $h_{A(\C)}(\bu) = h_{\C}(A^\intercal \bu)$, the problem \eqref{eq:sc_intro_constrainedlse} can be reformulated as follows:
\begin{equation}\label{eq:sc_intro_constrainedlse_lmap}
\hat{A}_n \in \underset{A \in L(\R^{q},\R^d)}{\mathrm{argmin}} \frac{1}{n} \sum_{i=1}^{n} \left( y^{(i)} - h_{\C} (A^\intercal \bu^{(i)}) \right)^2.
\end{equation}
Based on this observation, we analyze the properties of the set of minimizers of \eqref{eq:sc_intro_constrainedlse} via an analysis of \eqref{eq:sc_intro_constrainedlse_lmap}.  (The reformulation \eqref{eq:sc_intro_constrainedlse_lmap} is also more conducive to the development of numerical algorithms for solving \eqref{eq:sc_intro_constrainedlse}.)  In turn, a basic strategy for investigating the asymptotic properties of the estimator \eqref{eq:sc_intro_constrainedlse_lmap} is to analyze the minimizers of the loss function at the \emph{population level}.  Concretely, for any probability measure $P$ over pairs $(\bu,y) \in \sph^{d-1} \times \R$, the loss function with respect to $P$ is defined as:
\begin{equation} \label{eq:populationloss}
\pc(A,P) := \E_{P} [ ( \lmax(A^{\intercal}\bu) - y )^2 ].
\end{equation}
Thus, the focus of our analysis is on studying the set of minimizers of the population loss function $\pc(\cdot,\pk)$:
\begin{equation}
\mc := \underset{A}{\mathrm{argmin}} ~ \pc(A,\pk). \label{eq:populationminimizers}
\end{equation}

\subsection{Geometric Aspects} \label{sec:prelims_geom}
In this subsection, we focus on the \emph{convex sets} defined by the elements of the set of minimizers $\mc$.  In the next subsection, we consider the elements of $\mc$ as \emph{linear maps}.  To begin with, we state a simple lemma on the continuity of $\pc(\cdot,P)$:

\begin{proposition}\label{thm:stats_sc_cont}
	Let $P$ be any probability distribution over measurement pairs $(\bu,y) \in \sph^{d-1} \times \R$ satisfying $\E_{P}[|y|]<\infty$.  Then the function $\pc(\cdot,P)$ is continuous, where $\pc(\cdot,P)$ is defined in \eqref{eq:populationloss}.
\end{proposition}

The proof uses a simple bound.  We state the result formally as we require it at a later part.
\begin{lemma}\label{thm:linmap_diff}
Given any pair of linear maps $A_1, A_2 \in L(\R^q,\R^d)$, any unit-norm vector $\bu$, and any scalar $y$, we have $||\lmax(A_1^{\intercal}\bu) - y |  - |\lmax(A_2^{\intercal}\bu) - y || \leq |\lmax(A_1^{\intercal}\bu) - \lmax(A_2^{\intercal}\bu)| \leq \|A_1 - A_2\|_{\C,2} $.
\end{lemma}

\begin{proof}[Proof of Lemma \ref{thm:linmap_diff}]
First note that $h_{\C}(A_1^{\intercal}u) \leq h_{\C}(A_2^{\intercal}u) + h_{\C}((A_1-A_2)^{\intercal}u)$.  Since $u$ is unit-norm, we have $h_{\C}(A_1^{\intercal}u) - h_{\C}(A_2^{\intercal}u) \leq h_{\C}((A_1-A_2)^{\intercal}u) \leq \| A_1 - A_2 \|_{\C,2}$.  Similarly we have $h_{\C}(A_2^{\intercal}u) - h_{\C}(A_1^{\intercal}u) \leq \| A_1 - A_2 \|_{\C,2}$.  It follows that $||\lmax(A_1^{\intercal}\bu) - y |  - |\lmax(A_2^{\intercal}\bu) - y || \leq |\lmax(A_1^{\intercal}\bu) - \lmax(A_2^{\intercal}\bu)| \leq \|A_1 - A_2\|_{\C,2} $.
\end{proof}


\begin{proof}[Proof of Proposition \ref{thm:stats_sc_cont}]  Let $\epsilon>0$ be arbitrary.  Let $r = 1+\| A \|_{\C,2}$, and pick $\delta =\min \{ \epsilon/(3 \E[r+|y|]), r\}$.  Then for any $A_0 $ satisfying $\| A - A_0 \|_{\C,2} < \delta $, we have $|\Phi(A,P) - \Phi (A_0,P)| = |\E_{P} [( \lmax(A^{\intercal}\bu) - y )^2 - ( \lmax(A_0^{\intercal}\bu) - y )^2]| \leq \E_{P} [ | \lmax(A^{\intercal}\bu) - \lmax(A_0^{\intercal}\bu) | |\lmax(A^{\intercal}\bu) + \lmax(A_0^{\intercal}\bu) -  2y | ] $.  We apply Lemma \ref{thm:linmap_diff} to obtain the bound $|\lmax(A^{\intercal}\bu) + \lmax(A_0^{\intercal}\bu) -  2y| \leq 2|\lmax(A^{\intercal}\bu)| + 2|y| + \|A_0-A\|_{\C,2} < 2r+2|y|+\delta$.  Combining this bound with our earlier expression, it follows that $|\Phi(A,P) - \Phi (A_0,P)| <  \E_{P} [\delta(2r+\delta+2|y|)] \leq \epsilon$.
\end{proof}

The following result gives a series of properties about the set $\mc$.  Crucially, it shows that $\mc$ characterizes the optimal approximations of $\KS$ as linear images of $\C$:
\begin{proposition} \label{thm:form_mcproperties}
	Suppose that the assumptions (A1), (A2), (A3) hold.  Then the set of minimizers $\mc$ defined in \eqref{eq:populationminimizers} is compact and non-empty.  Moreover, we have
	\begin{equation*}
		\hat{A} \in \mc \quad \Leftrightarrow \quad \hat{A} \in \underset{A \in \lqd}{\mathrm{argmin}} ~ \rho_2 ( A(\C), \KS  ).
	\end{equation*}
\end{proposition}

{\color{black}
\begin{proof}[Proof of Proposition \ref{thm:form_mcproperties}] 
	Define the event $\grv := \{ (\bu,y) ~|~ \langle \bv,\bu \rangle \geq 1/2, |y| \leq r/4 \}$ over $\bv \in \sph^{d-1}$.  In addition, define the function $s(\bv) := \prob [ (\bu,y) ~|~ \langle \bv,\bu\rangle \geq 1/2 ]$.  For every $r \geq 0$, consider the function $g_{r}(\bv) :=\prob [\grv]$.  By noting that $ g_{r} \leq g_{r^{\prime}}$ whenever $r \leq r^{\prime}$ (i.e. the sequence $\{g_{r}\}_{r\geq 0}$ is monotone increasing), $g_{r}(\cdot) \uparrow s(\cdot)$, and that $g_{r}(\cdot)$ is a continuous function over the compact domain $\sph^{d-1}$, we conclude that $g_{r}(\cdot)$ converges to $s(\cdot)$ uniformly.  Thus there exists $\hat{r}$ sufficiently large such that $(\hat{r}^2/16) \prob[\grhv] >  \pc(0,\pk)$ for all $\bv \in \sph^{d-1}$.
	
	Next, we show that $\mc \subseteq \hat{r} B_{\| \cdot \|_{\C,2}} (0)$.  Let $A \notin \hat{r} B_{\| \cdot \|_{\C,2}} (0)$.  Then, for such an $A$, there exists $\hat{\bx} \in \C$ such that $\|A\hat{\bx}\|_{2} > \hat{r}$.  Define $\hat{\bv} = A \hat{\bx} / \|A\hat{\bx}\|_{2}$.  We have
	\begin{equation*}
		\pc(A,\pk) \geq \E [ \textbf{1}(G_{\hat{r},\hat{\bv}}) (\lmax(A^{\intercal}\bu)-y)^2 ] \geq \prob [G_{\hat{r},\hat{\bv}}] \hat{r}^2 / 16 > \pc(0,\pk).
	\end{equation*}
	Here, $\textbf{1}(G_{\hat{r},\hat{\bv}})$ denotes the indicator function for the event $G_{\hat{r},\hat{\bv}}$.  As such, the above inequality implies that $A \notin \mc$.  Therefore $\mc \subseteq \hat{r} B_{\| \cdot \|_{\C,2}} (0)$, and hence $\mc$ is bounded.
	
	By Proposition \ref{thm:stats_sc_cont}, the function $A \mapsto \pc(A,\pk)$ is continuous.  As the minimizers of $\pc(\cdot,\pk)$, if they exist, must be contained in $ \hat{r} B_{\| \cdot \|_{\C,2}} (0)$, we can view $\mc$ as minimizers of a continuous function restricted to the compact set $ \hat{r} B_{\| \cdot \|_{\C,2}} (0)$, and hence is non-empty.  Moreover, since $\mc$ is the pre-image of a closed set under a continuous function, it is also closed and thus compact.
	
	By Fubini's theorem, we have $\E[ \varepsilon (\lmax (\KS) - \lmax(A^{\intercal}\bu)  )] = \E_{\bu} [ \E_{\varepsilon} [ \varepsilon (\lmax (\KS) - \lmax(A^{\intercal}\bu)  ) ]] =  0$.  Hence $\pc(A,\pk) = \E[ (\lmax (\KS) + \varepsilon - \lmax(A^{\intercal}\bu)  )^2]= \E[(\lmax (\KS) - \lmax(A^{\intercal}\bu) )^2 ] + \E[ \varepsilon^2 ]$, from which the last assertion follows.
\end{proof}
}

It follows from Proposition \ref{thm:form_mcproperties} that an optimal approximation of $\KS$ as the projection of $\C$ always exists.  In Section \ref{sec:stats_sc}, we show that the estimators obtained using our method converge to an optimal approximation of $\KS$ as a linear image of $\C$ if such an approximation is unique.  While this is often the case in applications one encounters in practice, the following examples demonstrate that the uniqueness condition need not always hold:

\begin{example}
	Suppose $\KS$ is the regular $q$-gon in $\R^2$, and $\C$ is the spectraplex $\fs^{2}$. Then $\mc$ uniquely specifies an $\ell_2$-ball.
\end{example}

\begin{example}
	Suppose $\KS$ is the unit $\ell_2$-ball in $\R^2$, and $\C$ is the simplex $\simp^{q}$. Then the sets specified by the elements $\mc$ are not unique; they all correspond to a centered regular $q$-gon, but with an unspecified rotation.
\end{example}

A natural question then is to identify settings in which $\mc$ defines a unique set.  Unfortunately, obtaining a complete characterization of this uniqueness property appears to be difficult due to the interplay between the invariances underlying the sets $\KS$ and $\C$.  However, based on Proposition \ref{thm:form_mcproperties}, we can provide a simple sufficient condition under which $\mc$ defines a unique set:

\begin{corollary} \label{thm:form_mcastar}
	Assume that the conditions of Proposition~\ref{thm:form_mcproperties} hold.  Suppose further that we have $\KS = A^{\star} (\C)$ for some $A^{\star} \in \lqd$.  Then the set of minimizers $\mc$ described in \eqref{eq:populationminimizers} uniquely defines $\KS$; i.e., $\KS = A (\C)$ for all $A \in \mc$.	
\end{corollary}

{\color{black}
\begin{proof}[Proof of Corollary \ref{thm:form_mcastar}]
	It is clear that $A^{\star} \in \mc$.  Note that $\lmax(A^{\intercal}\bu)$ is a continuous function of $\bu$ over a compact domain for every $A \in \lqd$.  Hence it follows that $\hat{A}\in \mc$ if and only if $\lmax (A^{\star\intercal}\bu) = \lmax(\hat{A}^{\intercal}\bu)$ everywhere.  By applying Proposition \ref{thm:form_mcproperties} and using the fact that a pair of compact convex sets that have the same support function must be equal, it follows that $\KS=\hat{A}(\C)$ for all $\hat{A} \in \mc$.
\end{proof}
}

\subsection{Algebraic Aspects of Our Method} \label{sec:prelims_algebraic}

While the preceding subsection focused on conditions under which the set of minimizers $\mc$ specifies a unique convex set, the aim of the present section is to obtain a more refined picture of the collection of linear maps in $\mc$.  We begin by discussing the identifiability issues that arise in reconstructing a convex set by estimating a linear map via \eqref{eq:sc_intro_constrainedlse_lmap}.  Given a compact convex set $\C$, let $g$ be a linear transformation that preserves $\C$; i.e., $g(\C)= \C$.  Then the linear map defined by $Ag$ specifies the same convex set as $A$ because $Ag (\C) = A (g(\C)) = A (\C)$.  As such, every linear map $A\in \lqd$ is a member of the equivalence class defined by:
\begin{equation} \label{eq:sc_formulation_equivrelation}
A \sim A g , \quad g \in \aut(\C).
\end{equation}
Here $\aut(\C)$ denotes the subset of linear transformations that preserve $\C$.  When $\C$ is non-degenerate, the elements of $\aut(\C)$ are invertible matrices and form a subgroup of $\mathrm{GL}(q,\R)$.  As a result, the equivalence class $A \orb := \{ Ag ~|~ g \in \aut(\C) \}$ specified by \eqref{eq:sc_formulation_equivrelation} can be viewed as the orbit of $A \in \lqd$ under (right) action of the group $\aut(\C)$.  In the sequel, we focus our attention on convex sets $\C$ for which the associated automorphism group $\aut(\C)$ consists of isometries:
\begin{itemize}
	\item[(A4)] The automorphism group of $\C$ is a subgroup of the orthogonal group, i.e., $\aut(\C) \lhd O (q,\R)$.
\end{itemize}
This assumption leads to structural consequences that are useful in our analysis.  In particular, as $\aut(\C)$ can be viewed as a compact matrix Lie group, the orbit $A \orb$ inherits structure as a smooth manifold of the ambient space $\lqd$.  The assumption (A4) is satisfied for the choices of $\C$ that are primarily considered in this paper -- the automorphism group of the simplex is the set of permutation matrices, and the automorphism group of the spectraplex is the set of linear operators specified as conjugation by an orthogonal matrix.

Based on this discussion, it follows that the space of linear maps $\lqd$ can be partitioned into orbits $A \orb$.  Further, the population loss $\pc(\cdot,\pk)$ is also invariant over orbits of $A$: for every $g \in \aut(\C)$, we have that $\lmax(A^{\intercal}\bu) = \lmax((Ag)^{\intercal}\bu)$.  Thus, the set of minimizers $\mc$ can also be partitioned into a union of orbits.  Consequently, in our analysis in Section \ref{sec:stats} we view the problem \eqref{eq:sc_intro_constrainedlse_lmap} as one of recovering an orbit rather than a particular linear map.  The convergence results we obtain in Section \ref{sec:stats} depend on the number of orbits in $\mc$, with sharper asymptotic distributional characterizations available when $\mc$ consists of a single orbit.

When $\mc$ specifies multiple convex sets, then $\mc$ clearly consists of multiple orbits; as an illustration, in the example in the previous subsection in which $\KS$ is the unit $\ell_2$-ball in $\R^2$ and $\C = \simp^{q}$, the corresponding set $\mc$ is a union of multiple orbits in which each orbit specifies a unique rotation of the centered regular $q$-gon.  However, even when $\mc$ specifies a unique convex set, it may still be the case that it consists of more than one orbit:
\begin{example}
	Suppose $\KS$ is the interval $[-1,1] \subset \R$ and $\C = \simp^{3}$.  Then $\mc$ is a union of orbits, with an orbit specified as the set of all permutations of the vector $(-1,1,\epsilon)$ for each $\epsilon \in [-1, 1]$.  Nonetheless, $\mc$ specifies a unique convex set, namely, $\KS$.
\end{example}

More generally, it is straightforward to check that $\mc$ consists of a single orbit if $\KS$ is a polytope with $q$ extreme points and $\C = \simp^{q}$.  The situation for linear images of non-polyhedral sets such as the spectraplices is much more delicate.  One simple instance in which $\mc$ consists of a single orbit is when $\KS$ is the image under a bijective linear map $A$ of $\fs^q$.  Our next result states an extension to convex sets that are representable as linear images of an appropriate slice of the outer product of cones of positive semidefinite matrices.  

\begin{proposition} \label{thm:whensingleorbit}  Let $\C = \{ X_1 \times \ldots \times X_k ~|~ X_i \in \Sym^{q_i}, X_i \succeq 0, \sum_{i=1}^{k} \mathrm{trace}(X_i) = 1 \}$, and let $\KS = A^{\star}(\C) \subset \R^d$.  Suppose that there is a collection of disjoint exposed faces $F_i \subseteq \KS$ such that (i) $(A^{\star })^{-1}(F_i) \cap \C$ is the $i$-th block $\{ 0 \times \ldots \times 0 \times X_i \times 0 \times \ldots \times 0 ~|~ X_i \in \Sym^{q_i}, X_i \succeq 0, \mathrm{trace}(X_i) = 1 \} \subset \C$, and (ii) $\dim (F_i) = \dim (\fs^{q_i})$.  Then $\mc$ consists of a single orbit.
\end{proposition}

\begin{example}
	By expressing $\C = \simp^{q} = \{ X_1 \times \ldots \times X_q ~|~ X_i \in \Sym^{1}, X_i \succeq 0, \sum_{i=1}^{q} \mathrm{trace}(X_i) = 1 \}$ and by considering $\KS$ to be a polytope with $q$ extreme points, Proposition \ref{thm:whensingleorbit} simplifies to our earlier remark noting that $\mc$ consists of a single orbit.
\end{example}

\begin{example}
	The nuclear norm ball $B_{\mathrm{nuc.}}:=\{X ~|~ X \in \Sym^2, \|X\|_{\mathrm{nuc.}} \leq 1 \} $ is expressible as the linear image of $\fs^{2} \times \fs^{2}$.  The extreme points of $B_{\mathrm{nuc.}}$ comprise two connected components of unit-norm rank-one matrices specified by $\{ U^{\intercal} E_{11} U ~|~ U \in SO(2,\R) \} $ and $\{ - U^{\intercal} E_{11} U ~|~ U \in SO(2,\R) \} $, where $E_{11}$ is the $2\times 2$ matrix with $(1,1)$-entry equal to one and other entries equal to zero.  Furthermore, each connected component is isomorphic to $\fs^{2}$.  It is straightforward to verify that the conditions of Proposition \ref{thm:whensingleorbit} hold for this instance, and thus $M_{B_{\mathrm{nuc.}},\fs^{2} \times \fs^{2} }$ consists of a single orbit.
\end{example}

The proof of Proposition \ref{thm:whensingleorbit} requires an impossibility result showing that a spectraplex cannot be expressed as the linear image of the outer product of finitely many smaller-sized spectraplices.  The result follows as a consequence of a related result stated in terms of the cone of positive semidefinite matrices \cite{Ave:18,Sau:19}.  In the following, $\Sym^{q}_{+}$ denotes the cone of $q\times q$ dimensional positive semidefinite matrices.

\begin{proposition}\label{thm:singleorbit_impossibility_psdcone}
	Suppose that $\Sym^{q}_{+} = A( \Sym^{q_1}_{+} \times \ldots \Sym^{q_k}_{+} \cap L)$ for some linear map $A$ and some affine subspace $L$.  Then $q \leq \max q_i$.
\end{proposition}

\begin{proposition}\label{thm:singleorbit_impossibility}
	Let $\C = \{ X_1 \times \ldots \times X_k ~|~ X_i \in \Sym^{q_i}_{+}, ~ \sum \mathrm{trace}(X_i) = 1 \}$.  Suppose $\fs^{q} = A( \C )$ for some $A$.  Then $q \leq \max q_i$.
\end{proposition}

\begin{proof}[Proof of Proposition \ref{thm:singleorbit_impossibility}]  Express $\Sym^{q}_{+}$ as the following
	\begin{equation*}
		A \circ \Pi \left( \left\{ X_1 \times \ldots \times X_k \times t  ~\bigg|~ X_i \in \Sym^{q_i}_{+}, t \in \Sym^{1}_{+}, \sum_{i=1}^{k} \mathrm{trace}(X_i) = t \right\} \right),
	\end{equation*}
	where $\Pi$ is a map that projects out the coordinate $t$.  The result follows from Proposition \ref{thm:singleorbit_impossibility_psdcone}.
\end{proof}

\begin{lemma}\label{thm:singleorbit_uptoaut}
	Let $\K = A(\C) \subset \R^d$ where $\C \subset \R^q$ is compact convex, and suppose that $\dim(\K) = \dim(\C)$.  If  $\K = \tilde{A}(\C)$ for some $\tilde{A} \in \lqd$, then $\tilde{A} = A g$ for some $g \in \mathrm{Aut}(\C)$.
\end{lemma}

\begin{proof}[Proof of Lemma \ref{thm:singleorbit_uptoaut}]
	Suppose that $\bzero \notin \aff(\K)$.  By applying a suitable rotation, we may assume that $\K$ is contained in the first $\dim(\K)$ dimensions.  Then the maps $A$ and $\tilde{A}$ are of the form
	\begin{equation*}
		A = \left(\begin{array}{c}
			A_1 \\ 0
		\end{array} \right), \quad 
		\tilde{A} = \left(\begin{array}{c}
			\tilde{A}_1 \\ 0
		\end{array} \right),
		\quad A_1,\tilde{A}_1 \in L(\R^q,\R^q).
	\end{equation*}
	Since $\dim(\K) = \dim(\C)$, the map $A_1$ is invertible.  Subsequently $A_1^{-1} A_1 \in \aut(\C)$, and thus $\tilde{A}_1 = A_1 g$ for some $g \in \aut(\C)$. 
	
	The proof is similar for the case where $\bzero \in \aff(\K)$.  The only necessary modification is that we embed $\K$ into $\R^{d+1}$ via the set $\tilde{\K}:=\{ (\bx,1) ~|~ \bx \in \K \}$, and we repeat the same sequence of steps with $\tilde{\K}$ in place of $\K$.  We omit the necessary details as they follow in a straightforward fashion from the previous case.
\end{proof}

\begin{proof}[Proof of Proposition \ref{thm:whensingleorbit}]
	Let $\tilde{A} \in \mc$.  We show that $\tilde{A}$ defines a one-to-one correspondence between the collection of faces $\{F_i\}_{i=1}^{k}$ and the collection of blocks $\{\{ 0 \times \ldots \times X_j \times \ldots \times 0 ~|~ X_j \in \fs^{q_j} \} \}_{j=1}^{k}$ subject to the condition $\dim(F_i) = \dim(\fs^{q_j})$.  We prove such a correspondence via an inductive argument beginning with the faces of largest dimensions.
	
	We assume (without loss of generality) that $\dim(F_1) = \ldots = \dim(F_{k^{\prime}})  > \ldots$ and that $\dim(\fs^{q_1}) = \ldots = \dim(\fs^{q_{k^{\prime}}}) > \ldots $.  We further denote $q = q_1 = \ldots q_{k^{\prime}}$.  As $F_i$ is an exposed face, the pre-image $\tilde{A}^{-1}(F_i) \cap \C$ must be an exposed face of $\C$, and thus is of the form $U_{i,1} X_{i,1} U_{i,1}^{\prime} \times \ldots \times U_{i,k} X_{i,k} U_{i,k}^{\prime}$, where $X_{i,j} \in \fs^{q_{i,j}}$ for some $q_{i,j} \leq q_j$, and where $U_{i,j} \in \R^{q_{i,j} \times q_j}$ are partial orthogonal matrices.  By Proposition \ref{thm:singleorbit_impossibility}, we have $\max_j q_{i,j} \geq q_i = q$.  Subsequently, by noting that there are $k^{\prime}$ blocks with dimensions $q \times q$, that there are also $k^{\prime}$ faces $F_i$ with $\dim(F_i) = q$, and that the faces $F_i$ are disjoint, we conclude that each block in the collection $\{\{ 0 \times \ldots \times X_j \times \ldots \times 0 ~|~ X_j \in \fs^{q_j} \} \}_{j=1}^{k}$ lies in the pre-image of a unique face $F_i$, $1\leq i \leq k^{\prime}$.  By repeating the same sequence of arguments for the remaining faces of smaller dimensions, we establish a one-to-one correspondence between faces and blocks.  Finally, we apply Lemma \ref{thm:singleorbit_uptoaut} to each face-block pair to conclude that, after accounting for permutations among blocks of the same size, the maps $A_{i}^{\star}$ and $\tilde{A}_{i}$ are equivalent up to conjugation by an orthogonal matrix.  The final assertion that $\mc$ consists of a single orbit is straightforward to establish.
\end{proof}

\subsection{Analytical Aspects of Our Method} \label{sec:prelims_analytic}

In this third subsection, we describe some of the derivative computations that repeatedly play a role in our paper in our analysis, examples, and numerical experiments.  Given a compact convex set $\C$, the support function $\lmax(\cdot)$ is differentiable at $\bu$ if and only if $\mathrm{argmax}_{\bx \in \C} \langle \bx, \bu \rangle$ is a singleton; the derivative in these cases is given by $\mathrm{argmax}_{\bx \in \C} \langle \bx, \bu \rangle$ (see \cite{Sch:93} [p.~47]).  We denote the derivative of $\lmax$ at a differentiable $\bu$ by $\emax(\bu) := \nabla_{\bu} (\lmax(\bu))$.

\begin{example}
	Suppose $\C = \simp^{q} \subset \R^q$ is the simplex.  The function $\lmax(\cdot)$ is the maximum entry of the input vector, and it is differentiable at this point if and only if the maximum is unique with the derivative $\emax(\cdot)$ equal to the corresponding standard basis vector.
\end{example}

\begin{example}
	Suppose $\C = \fs^{p} \subset \Sym^p$ is the spectraplex.  The function $\lmax(\cdot)$ is the largest eigenvalue of the input matrix, and it is differentiable at this point if and only if the largest eigenvalue has multiplicity one with the derivative $\emax(\cdot)$ equal to the projector onto the corresponding one-dimensional eigenspace.
\end{example}

The following result gives a formula for the derivative of $\pc(\cdot,\pk)$. 

\begin{proposition} \label{thm:formulation_diffble} Let $P$ be a probability distribution over the measurement pairs $(\bu,y)$, and suppose that $\E_{P}[y^2] < \infty$.  Let $A \in \lqd$ be a linear map such that $\lmax(\cdot)$ is differentiable at $A^{\intercal}\bu$ for $P$-a.e. $\bu$.  Then the function $\pc(\cdot,P)$ is differentiable with derivative $2 \E_{P} [(\lmax (A^{\intercal}\bu) - y) \bu \otimes \emax (A^{\intercal}\bu) ]$.
\end{proposition}

{\color{black}
\begin{proof}[Proof of Proposition \ref{thm:formulation_diffble}]  Let $\rem(\cdot,A,D)$ denote the remainder term satisfying 
	\begin{equation*}
	(\lmax((A+D)^{\intercal}\bu) - y)^2 = (\lmax(A^{\intercal}\bu) - y)^2 + \langle \nabla_{A} ((\lmax(A^{\intercal}\bu) - y)^2) , D \rangle + \rem(\cdot,A,D) \| D \|_{\C,2}.
	\end{equation*}
	Since the function $\lmax(\cdot)$ is differentiable at $A^{\intercal}\bu$ for $P$-a.e. $\bu$, we have $\rem(\cdot,A,D) \rightarrow 0$ as $\|D\|_{\C,2} \rightarrow 0$ for $P$-a.e.  Suppose $D$ is in a bounded set.  First, we can bound $|\lmax((A+D)^{\intercal}\bu) + \lmax(A^{\intercal}\bu) -2y | \leq c_1 (1 + |y|)$ for some constant $c_1$.  Second, by Lemma \ref{thm:linmap_diff}, we have the inequality $| \lmax((A+D)^{\intercal}\bu) - \lmax(A^{\intercal}\bu)| \leq \|D\|_{\C,2}$.  Third, by noting that $|\lmax (A^{\intercal}\bu) - y|$ can be bounded by $c_2 (1 + |y|)$ for some constant $c_2$, and by noting that the entries of the linear map $ \bu \otimes \emax (A^{\intercal}\bu)$ are uniformly bounded, we may bound $\| (\lmax (A^{\intercal}\bu) - y) \bu \otimes \emax (A^{\intercal}\bu)\|_{\C,2}$ by a function of the form $c_3 (1 + |y|)$ for some constant $c_3$.  Subsequently we may bound
	\begin{align*}
	|\rem(\cdot,A,D)| & \leq \|2 (\lmax (A^{\intercal}\bu) - y) \bu \otimes \emax (A^{\intercal}\bu)\|_{\C,2} \\
	& \quad \quad + | \lmax((A+D)^{\intercal}\bu) - \lmax(A^{\intercal}\bu)| |\lmax((A+D)^{\intercal}\bu) + \lmax(A^{\intercal}\bu) -2y | / \|D\|_{\C,2} \\
	& \leq c(1+|y|)
	\end{align*}
	for some constant $c$.  Since $\E_{P} [y^2]<\infty$, we have $\rem(\cdot,A,D) \in \mathcal{L}^2(P)$, and hence $\rem(\cdot,A,D) \in \mathcal{L}^1(P)$.  The result follows from an application of the Dominated Convergence Theorem.
\end{proof}
}

It turns out to be considerably more difficult to compute an explicit expression of the second derivative of $\pc(\cdot,\pk)$.  For this reason, our next result applies in a much more restrictive setting in comparison to Proposition \ref{thm:formulation_diffble}.

\begin{proposition} \label{thm:stats_specialcasev}
	Suppose that the underlying set $\KS = A^{\star}(\C)$ for some $A^{\star} \in \lqd$.  In addition, suppose that the function $\lmax(\cdot)$ is continuously differentiable at $A^{\star\intercal}\bu$ for $\pk$-a.e. $\bu$.  Then the function $A \mapsto \pc(A,\pk)$ is twice differentiable at $A^{\star}$, and whose second derivative is the operator $\Gamma \in L(\lqd,L^*(\R^q,\R^d))$ defined by
	\begin{equation} \label{eq:secondderivformula}
	\Gamma (D) = 2 \E \left[ \langle \bu \otimes \emax(A^{\star\intercal}\bu) , D \rangle  \bu \otimes \emax(A^{\star\intercal}\bu) \right].
	\end{equation}
\end{proposition}

{\color{black}
\begin{proof}[Proof of Proposition \ref{thm:stats_specialcasev}]  To simplify notation, we denote the operator norm $\|\cdot\|_{\C,2}$ by $\|\cdot\|$ in the remainder of the proof.  By Proposition \ref{thm:formulation_diffble}, the map $A \mapsto \Phi(A,P)$ is differentiable in an open neighborhood around $A^{\star}$ with derivative $2 (\lmax (A^{\intercal}\bu) - y) \bu \otimes \emax (A^{\intercal}\bu)$.  Hence to show that the map is twice differentiable with second derivative $\Gamma$, it suffices to show that
	\begin{align*} \label{eq:stats_specialcasev0}
		& \lim_{\|D\| \rightarrow 0} \frac{1}{\|D\|} \bigl\| \E[ 2 (\lmax ((A^{\star}+D)^{\intercal}\bu) - y) \bu \otimes \emax ((A^{\star}+D)^{\intercal}\bu)] \\
		& \quad \quad \quad \quad \quad \quad \quad \quad - \E[ 2 (\lmax (A^{\star\intercal}\bu) - y) \bu \otimes \emax (A^{\star\intercal}\bu)] - \Gamma (D) \bigr\| = 0.
	\end{align*}
	
	First we note that every component of $\varepsilon(\bu) \bu \otimes \emax((A^{\star}+D)^{\intercal}\bu)$ is integrable because $\mathbb{E}[\varepsilon(\bu)^2] < \infty$, and $\bu \otimes \emax((A^{\star}+D)^{\intercal}\bu)$ is uniformly bounded.  Hence by Fubini's Theorem we have
	\begin{equation} \label{eq:sc_formulation_2ndderiveq1}
	\E[ (\lmax(A^{\star\intercal}\bu) - y) \bu \otimes \emax ( (A^{\star}+D)^{\intercal}\bu ) ]
	= \mathbb{E}_{\bu}[ \mathbb{E}_{\varepsilon(\bu)} [ - \varepsilon(\bu) \bu \otimes \emax ( (A^{\star}+D)^{\intercal}\bu ) ] ] = 0.
	\end{equation}
	Similarly
	\begin{equation} \label{eq:sc_formulation_2ndderiveq2}
	\E[(\lmax(A^{\star\intercal}\bu) - y) \bu \otimes \emax ( A^{\star \intercal}\bu )] = 0.
	\end{equation}
	
	Second by differentiability of the map $A\mapsto \Phi(A,P)$ at $A^{\star}$ we have
	\begin{equation*}
		\lim_{\|D\| \rightarrow 0}  \frac{1}{\|D\|} \bigl\| \E\left[ ( 2(\lmax((A^{\star}+D)^{\intercal}\bu)-y) - 2(\lmax(A^{\star\intercal}\bu)-y) - 2 \langle \bu \otimes \emax(A^{\star\intercal}\bu), D \rangle ) \right] \bigr\| = 0.
	\end{equation*}
	By noting that every component of $\bu \otimes \emax((A^{\star}+D)^{\intercal}\bu)$ is uniformly bounded, and an application of the Dominated Convergence Theorem, we have
	\begin{align} \label{eq:sc_formulation_2ndderiveq3}
		& \lim_{\|D\| \rightarrow 0}  \frac{1}{\|D\|} \bigl\| \E \bigl[ \bigl( 2(\lmax((A^{\star}+D)^{\intercal}\bu)-y) \nonumber \\
		& \quad \quad \quad \quad \quad \quad \quad \quad - 2(\lmax(A^{\star\intercal}\bu)-y) - 2 \langle \bu \otimes \emax(A^{\star\intercal}\bu), D \rangle \bigr) \bu \otimes \emax((A^{\star}+D)^{\intercal}\bu) \bigr] \bigr\| = 0
	\end{align}
	
	Third since $\lmax(\cdot)$ is continuously differentiable at $A^{\star\intercal}\bu$ for $P$-a.e. $\bu$, we have $\emax((A^{\star}+D)^{\intercal}\bu) \rightarrow \emax(A^{\star\intercal}\bu)$ as $\|D\| \rightarrow 0$, for $P$-a.e. $\bu$.  By the Dominated Convergence Theorem we have $\E[\emax((A^{\star}+D)^{\intercal}\bu)] \rightarrow \E[ \emax(A^{\star\intercal}\bu)] $ as $\|D\| \rightarrow 0$.  It follows that
	\begin{align} \label{eq:sc_formulation_2ndderiveq4}
		& \lim_{\|D\| \rightarrow 0} \frac{1}{\|D\|} \bigl\| 2 \E\left[  \langle \bu \otimes \emax(A^{\star\intercal}\bu), D \rangle  \bu \otimes \emax( (A^{\star}+D)^{\intercal}\bu) \right] \nonumber \\
		& \quad \quad \quad \quad \quad \quad \quad \quad  - 2 \E\left[  \langle \bu \otimes \emax(A^{\star\intercal}\bu), D \rangle  \bu \otimes \emax( A^{\star\intercal}\bu ) \right] \bigr\| = 0.
	\end{align}
	
	The result follows by summing the contributions from \eqref{eq:sc_formulation_2ndderiveq3} and \eqref{eq:sc_formulation_2ndderiveq4}, as well as noting that the expressions in \eqref{eq:sc_formulation_2ndderiveq1} and \eqref{eq:sc_formulation_2ndderiveq2} vanish.
\end{proof}
}

\section{Main Results} \label{sec:stats}

In this section, we investigate the statistical aspects of minimizers of the optimization problem \eqref{eq:sc_intro_constrainedlse_lmap}.  Our objective in this section is to relate a sequence of minimizers $\{\hat{A}_n\}_{n=1}^\infty$ of $\pc (\cdot,\pnk)$ to minimizers of $\pc (\cdot,\pk)$.  Based on this analysis, we draw conclusions about properties of sequences of minimizers $\{\KHN\}_{n=1}^{\infty}$ of the problem \eqref{eq:sc_intro_constrainedlse}.  In establishing various convergence results, we rely on an important property of the Hausdorff distance, namely that it defines a metric over collections of non-empty compact sets; therefore, the collection of all orbits $\{A \orb ~|~ A \in \lqd\}$ endowed with the Hausdorff distance defines a metric space.

Our results provide progressively sharper recovery guarantees based on increasingly stronger assumptions.  Specifically, Section \ref{sec:stats_sc} focuses on conditions under which a sequence of minimizers $\{\KHN\}_{n=1}^{\infty}$ of \eqref{eq:sc_intro_constrainedlse} converges to $\KS$; this result only relies on the fact that the optimal approximation of $\KS$ by a convex set specified by an element of $\mc$ is unique (see Section \ref{sec:prelims_geom} for the relevant discussion).  Next, Section \ref{sec:stats_asymnormal} gives a limiting distributional characterization of the sequence $\{\KHN\}_{n=1}^{\infty}$ based on an asymptotic normality analysis of the sequence $\{\hat{A}_n\}_{n=1}^\infty$; among other assumptions, this analysis relies on the stronger requirement that $\mc$ consists of a single orbit.  Finally, based on additional conditions on the facial structure of $\KS$, we describe in Section \ref{sec:faces} how the sequence $\{\KHN\}_{n=1}^{\infty}$ preserves various attributes of the face structure of $\KS$.

\subsection{Strong Consistency} \label{sec:stats_sc}

We describe conditions for convergence of a sequence of minimizers $\{\KHN\}_{n=1}^{\infty}$ of \eqref{eq:sc_intro_constrainedlse}.  Our main result essentially states that such a sequence converges to an optimal $\rho_2$ approximation of $\KS$ as a linear image of $\C$, provided that such an approximation is unique:

\begin{theorem} \label{thm:stats_setconvergence}  Suppose that the assumptions (A1), (A2), (A3), (A4) hold.  Let $\{ \hat{A}_n \}_{n=1}^{\infty}$ be a sequence of minimizers of the empirical loss function $\pc (\cdot,\pnk)$ with the corresponding reconstructions given by $\KHN = \hat{A}_n(\C)$.  We have that $\mathrm{dist}(\hat{A}_n,\mc) \rightarrow 0$ a.s. and $\inf_{A \in \mc} \allowbreak d_{H}(\hat{A}_n \orb, \allowbreak A \orb) \rightarrow 0$ a.s.  As a consequence, if $\mc$ specifies a unique set -- there exists $\KH \subset \R^d$ such that $\KH = A(\C)$ for all $A \in \mc$ -- then $\rho_{p} ( \KHN , \KH ) \rightarrow 0$ a.s. for $1 \leq p \leq \infty$.  Further, if $\KS = A^{\star}(\C)$ for some linear map $A^{\star} \in L(\R^q,\R^d)$, then $\rho_{p} (\KHN,\KS) \rightarrow 0$ a.s. for $1 \leq p \leq \infty$.
\end{theorem}

When $\mc$ defines multiple sets, our result does not imply convergence of the sequence $\{\KHN\}_{n=1}^{\infty}$.  Rather, we obtain the weaker consequence that the sequence $\{\KHN\}_{n=1}^{\infty}$ eventually becomes arbitrarily close to the collection $\{A(\C) ~|~ A \in \mc\}$.

\begin{example}
	Suppose $\KS$ is the unit $\ell_2$-ball in $\R^2$, and $\C= \simp^{q}$. The optimal $\rho_2$ approximation is the regular $q$-gon with an unspecified rotation.  The sequence $\{\KHN\}_{n=1}^{\infty}$ does not have a limit; rather, there is a sequence $\{g_n\}_{n=1}^{\infty} \subset SO(2,\R)$ such that $g_n \KHN$ converges to a centered regular $q$-gon (with fixed orientation) a.s.
\end{example}

The proof of Theorem \ref{thm:stats_setconvergence} comprises two parts.  First, we show that there exists a ball in $L(\R^q,\R^d)$ such that $\hat{A}_n$ belongs to this ball for all sufficiently large $n$ a.s.  Second, we appeal to the following uniform convergence result.  The structure of our proof is similar to that of a corresponding result for $K$-means clustering (see the main theorem in \cite{Pollard:81}).

\begin{lemma} \label{thm:stats_sc_uslln}
	Let $U \subset L(\R^q,\R^d)$ be bounded and suppose $\C \subset \R^q$ satisfies assumption (A3).  Let $P$ be a probability distribution over the measurement pairs $(\bu,y) \subset \sph^{d-1} \times \R$ satisfying $\E_{P}[y^2] < \infty$, and let $P_n$ be the empirical measure corresponding to drawing $n$ i.i.d. observations from the distribution $P$.  Consider the collection of functions $\mathcal{G} := \{ (\lmax(A^{\intercal}\bu) - y)^2 ~|~ A \in U \}$ in the variables $(\bu,y)$.  Then $\sup_{g \in \mathcal{G}} \left| \E_{P_n}[g] - \E_{P}[g] \right| \rightarrow 0 \quad \text{as} \quad n \rightarrow \infty$ a.s.
\end{lemma}

The proof of Lemma \ref{thm:stats_sc_uslln} follows from an application of the following uniform strong law of large numbers (SLLN) \cite{Pol:84} [Theorem~3,~p.~8].

\begin{theorem}[Uniform SLLN] \label{thm:uslln}  Let $Q$ be a probability measure, and let $Q_n$ be the corresponding empirical measure.  Let $\mathcal{G}$ be a collection of $Q$-integrable functions.  Suppose that for every $\epsilon>0$ there exists a finite collection of functions $\mathcal{G}_{\epsilon}$ such that for every $g \in \mathcal{G}$ there exists $\overline{g},\underline{g} \in \mathcal{G}_{\epsilon}$ satisfying (i) $\underline{g} \leq \overline{g}$, and (ii) $\E_{Q}[ \overline{g} - \underline{g} ] < \epsilon$.  Then $\sup_{g \in \mathcal{G}} | \E_{Q_n}[g] - \E[g] | \rightarrow 0  $ a.s..
\end{theorem}

\begin{proof}[Proof of Lemma \ref{thm:stats_sc_uslln}]
	Based on Theorem \ref{thm:uslln}, it suffices to construct the finite collection of functions $\mathcal{G}_{\epsilon}$.  Pick $r$ sufficiently big so that $U \subset r B_{\|\cdot\|_{\C,2}} (0)$. Let $\mathcal{D}_{\delta}$ be a $\delta$-cover for $U$ in the $\|\cdot\|_{\C,2}$-norm, where $\delta$ is chosen so that $4 \delta \E_{P}[ r + |y| ] \leq \epsilon$.  We define $\mathcal{G}_{\epsilon} :\{ ((|\lmax(A^{\intercal}\bu)-y| - \delta)_{+})^2 \}_{A \in \mathcal{D}_{\delta}} \cup \{ (|\lmax(A^{\intercal}\bu)-y| + \delta)^2 \}_{A \in \mathcal{D}_{\delta}} $.
	
	We proceed to verify (i) and (ii).  Let $g = (\lmax(A^{\intercal}\bu) - y)^2 \in \mathcal{G}$ be arbitrary.  Let $A_0 \in \mathcal{D}_{\delta}$ be such that $\|A - A_0 \|_{\C,2} \leq \delta$.  Define $\underline{g} = ((|\lmax(A^{\intercal}_0 \bu)-y| - \delta)_{+})^2$ and $\overline{g} = (|\lmax(A^{\intercal}_0 \bu)-y| + \delta)^2$.  It follows that $\underline{g} \leq g \leq \overline{g}$, which verifies (i).  Next, we have $\E [ \overline{g} - \underline{g} ] \leq 4 \delta \E [ | \lmax(A_0^{\intercal}\bu) - y| ] \leq  4 \delta \E [ r + |y| ] \leq \epsilon$,
	which verifies (ii).
\end{proof}

\begin{proof}[Proof of Theorem \ref{thm:stats_setconvergence}]  To simplify notation in the following proof, we denote $B := B_{\|\cdot\|_{\C,2}}(0)$.
	
	First, we recall the definition of the event $\grv$ and the function $s(\bv)$ from the proof of Proposition \ref{thm:form_mcproperties}.  Using a sequence of arguments identical to the proof of Proposition \ref{thm:form_mcproperties}, it follows that there exists $\hat{r}$ sufficiently large such that $(\hat{r}^2/16) \prob[\grhv] >  \pc(0,\pk)$ for all $\bv \in \sph^{d-1}$.
	We claim that $\hat{A}_{n} \in \hat{r} B$ eventually a.s.  We prove this assertion via contradiction.  Suppose on the contrary that $\hat{A}_{n} \notin \hat{r} B$ i.o.  For every $\hat{A}_n \notin \hat{r} B$, there exists $\hat{\bx}_{n} \in \C$ such that $\| \hat{A}_n \hat{\bx}_{n} \| > \hat{r}$.  The sequence of unit-norm vectors $\hat{A}_n \hat{\bx}_{n} / \|\hat{A}_n \hat{\bx}_{n} \|_{2}$, defined over the subset of indices $n$ such that $\hat{A}_n \notin \hat{r} B$, contains a convergent subsequence whose limit point is $\hat{\bv} \in \sph^{d-1}$.  Then
	\begin{equation*}\begin{aligned}
			\underset{n}{\limsup} ~ \pc(\hat{A}_n,\pnk) & = \underset{n}{\limsup} ~ \E_{\pnk} [(\lmax(\hat{A}_n^{\intercal}\bu)-y)^2] \\
			& \geq \underset{n}{\limsup} ~ \E_{\pnk}[ \textbf{1}(G_{\hat{r},\hat{\bv}}) (\lmax(\hat{A}_n^{\intercal}\bu)-y)^2 ] \\
			& \geq \underset{n}{\limsup} ~ \E_{\pnk}[\textbf{1}(G_{\hat{r},\hat{\bv}}) (r^2 / 16)] \\
			& \geq \underset{n}{\lim} ~ \prob_{\pnk}[G_{\hat{r},\hat{\bv}}] (r^2 / 16) = \prob_{\pk}[G_{\hat{r},\hat{\bv}}] (r^2 / 16) > \pc(0,\pk).
	\end{aligned}\end{equation*}
	Here, the last equality follows from the SLLN.  This implies $\pc(\hat{A}_n, \pnk) > \pc (0, \pnk)$ i.o., which contradicts the minimality of $\hat{A}_n$.  Hence $\hat{A}_{n} \in \hat{r} B$ eventually a.s.
	
	Second, we show that $\mathrm{dist}(\hat{A}_n,\mc) \rightarrow 0$ a.s.  It suffices to show that $\hat{A}_n \in U$ eventually a.s, where $U$ is any open set containing $\mc$.  Let $\hat{A} \in \mc$ be arbitrary.  By Proposition \ref{thm:stats_sc_cont}, the function $A \mapsto \pc(A,\pk)$ is continuous.  By noting that the set of minimizers of $\pc(\cdot,\pk)$ is compact from Proposition \ref{thm:form_mcproperties}, we can pick $\epsilon>0$ sufficiently small so that $\{ A ~|~ \pc(A,\pk) < \pc(\hat{A},\pk) + \epsilon \} \subset U$.  Next, since $\hat{A}_n$ is defined as the minimizer of an empirical sum, we have $\pc (\hat{A}_n,\pnk) \leq \pc(\hat{A},\pnk)$ for all $n$.  By applying Lemma \ref{thm:stats_sc_uslln} with the choice of $P=\pk$, we have $\pc(\hat{A}_n,\pnk) \rightarrow \pc(\hat{A}_n,\pk)$, and $\pc(\hat{A},\pnk) \rightarrow \pc (\hat{A},\pk)$, both uniformly and in the a.s. sense.  Subsequently, by combining the previous two conclusions, we have $\pc(\hat{A}_n,\pk) < \pc(\hat{A},\pk) + \epsilon$ eventually, for any $\epsilon>0$.  This proves that $\mathrm{dist}(\hat{A}_n,\mc) \rightarrow 0$ a.s.
	
	Third, we conclude that $\inf_{A \in \mc} d_{H}(\hat{A}_n \orb, A \orb) \rightarrow 0$ a.s.  Fix an integer $n$, and let $t_n = \mathrm{dist}(\hat{A}_n,\mc)$.  Since $\mc$ is compact, we may pick $\bar{A} \in \mc$ such that $\|\hat{A}_{n} - \bar{A} \|_{F} = t_n$.  Given $A \in \hat{A}_n \orb$, we have $A = \hat{A}_n g$ for some $g \in \mathrm{Aut}(\C)$.  Then $\bar{A} g \in \bar{A} \orb$, and since $g$ is an isometry by Assumption (A4), we have $\|\hat{A}_{n} g - \bar{A} g \|_{F} = \|\hat{A}_{n} - \bar{A} \|_{F} = t_n$.  This implies that $d_{H}(\hat{A}_n \orb,\bar{A} \orb) \leq t_n$.  Since $t_n \rightarrow 0$ as $n \rightarrow 0$, it follows that $\inf_{A \in \mc} d_{H}(\hat{A}_n \orb,A \orb) \rightarrow 0$ a.s.
\end{proof}


\subsection{Asymptotic Normality} \label{sec:stats_asymnormal}

In our second main result, we characterize the limiting distribution of a sequence of estimates $\{\KHN\}_{n=1}^{\infty}$ corresponding to minimizers of \eqref{eq:sc_intro_constrainedlse} by analyzing an associated sequence of minimizers of \eqref{eq:sc_intro_constrainedlse_lmap}.  Specifically, we show under suitable conditions that the estimation error in the sequence of minimizers of the empirical loss \eqref{eq:sc_intro_constrainedlse_lmap} is asymptotically normal.  After developing this theory, we illustrate in Section \ref{sec:stats_asymnormal_examples} through a series of examples the asymptotic behavior of the set $\KHN$, highlighting settings in which $\KHN$ converges as well as situations in which our asymptotic normality characterization fails due to the requisite assumptions not being satisfied.  Our result relies on two key ingredients, which we describe next.

The first set of requirements pertains to non-degeneracy of the function $\pc(\cdot,\pk)$.  First, we require that the minimizers of $\pc(\cdot,\pk)$ constitute a unique orbit under the action of the automorphism group of the set $\C$; this guarantees the existence of a convergent sequence of minimizers of the empirical losses $\pc(\cdot,\pnk)$, which is necessary to provide a Central Limit Theorem (CLT) type of characterization.  Second, we require the function $\pc(\cdot,\pk)$ to be twice differentiable at a minimizer with a positive definite Hessian (modulo invariances due to $\aut(\C)$); such a condition allows us to obtain a quadratic approximation of $\pc(\cdot,\pk)$ around a minimizer $\hat{A}$, and to subsequently compute first-order approximations of minimizers of the empirical losses $\Phi(\cdot,\pnk)$.  These conditions lead to a geometric characterization of confidence regions of the extreme points of the limit of the sequence $\{\KHN\}_{n=1}^\infty$.

Our second main assumption centers around the collection of sets that serves as the constraint in the optimization problem \eqref{eq:sc_intro_constrainedlse}.  Informally, we require that this collection is not ``overly complex,'' so that we can appeal to a suitable CLT.  The field of empirical processes provides the tools necessary to formalize matters. Concretely, as our estimates are obtained via minimization of an empirical process, we require that the following divided differences of the loss function are well-controlled:
\begin{equation*}
	d_{\C,A_1,A_2} (\bu,y) = \frac{1}{\|A_1-A_2\|_{F}} \left( (h_{\C} (A_1^{\intercal}\bu) - y)^2 - (h_{\C} (A_2^{\intercal}\bu) - y)^2 \right).
\end{equation*}
In particular, a natural assumption is that the \emph{graph} associated to these divided differences, indexed over a collection centered at $\hat{A}$, is of suitably ``bounded complexity'':
\begin{equation} \label{eq:check_vc}
\left\{ \left\{ (\bu,y,s) ~\big|~ d_{\C,A,\hat{A}} (\bu,y) \geq s \geq 0 \text{ or } d_{\C,A,\hat{A}} (\bu,y) \leq s \leq 0 \right\} ~\Big|~ A \in B_{\|\cdot\|_{F}}(\hat{A}) \backslash \{ \hat{A}\} \right\}.
\end{equation}
A powerful framework to quantify the `richness' of such collections is based on the notion of a \emph{Vapnik-Chervonenkis} (VC) class \cite{VapChe:68}.  VC classes describe collections of subsets with bounded complexity, and they feature prominently in the field of statistical learning theory, most notably in conditions under which generalization of a learning algorithm is possible.
\begin{definition}[Vapnik-Chervonenkis (VC) Class]
	Let $\mathcal{F}$ be a collection of subsets of a set $F$.  A finite set $D$ is said to be \emph{shattered} by $F$ if for every $A \subset D$ there exists $G \in \mathcal{F}$ such that $A = G \cap D$.  The collection $\mathcal{F}$ is said to be \emph{VC class} if there is a finite $k$ such that all sets with cardinality greater than $k$ cannot be shattered by $\mathcal{F}$.
\end{definition}

Whether or not the collection \eqref{eq:check_vc} is VC depends on the particular choice of $\C$.  If $\C$ is chosen to be either a simplex or a spectraplex then such a property is indeed satisfied -- see Section \ref{sec:stats_asymnorm_vcproof} for further details.  Our analysis relies on a result by Stengle and Yukich showing that certain collection of sets admitting semialgebraic representations are VC \cite{SteYuk:89}.  We are now in a position to state our main result of this section:

\begin{theorem}\label{thm:stats_norm} Suppose that the assumptions (A1), (A2), (A3), and (A4) hold.  Suppose that there exists $\hat{A} \in \lqd$ such that the set of minimizers $\mc =\hat{A} \orb$ (i.e., $\mc$ consists of a single orbit), the function $h_{\C}(\cdot)$ is differentiable at $\hat{A}^{\intercal}\bu$ for $\pk$-a.e. $\bu$, the function $\pc(\cdot,\pk)$ is twice differentiable at $\hat{A}$, and the associated Hessian $\Gamma$ at $\hat{A}$ is positive definite restricted to the subspace $\nspace$, i.e., $\Gamma|_{\nspace} \succ 0$ -- here, $T:=T_{\hat{A}} \mc$ denotes the tangent space of $\mc$ at $\hat{A}$.  In addition, suppose that the collection of sets specified by \eqref{eq:check_vc} forms a VC class.  Let $\tilde{A}_n \in \mathrm{argmin}_{A} \pc (A,\pnk), ~n \in \mathbb{N}$ be a sequence of minimizers of the empirical loss function, and let $\tilde{A}_n \in \mathrm{argmin}_{A \in \hat{A}_n \orb} \| \hat{A} - A \|_F, ~ n \in \mathbb{N}$ specify an equivalent sequence defined with respect to the minimizer $\hat{A}$ of the population loss.  Setting $\nabla = \nabla_{A} ((h_{\C}(A^{\intercal}\bu)-y)^2)$, we have that
	\begin{equation*}
		\sqrt{n} (\tilde{A}_n - \hat{A}) \overset{D}{\rightarrow} \mathcal{N}(0, (\Gamma|_{\nspace})^{-1} (\E_{\pk}[\nabla \otimes \nabla |_{A=\hat{A}}]|_{\nspace}) (\Gamma|_{\nspace})^{-1} ).
	\end{equation*}
\end{theorem}

The proof of this theorem relies on ideas from the literature of empirical processes, which we describe next in a general setting.

\begin{proposition}\label{thm:vc_to_stochequi}
	Suppose $P$ is a probability measure over $\R^{q}$, and $P_n$ is the empirical measure corresponding to $n$ i.i.d. draws from $P$.  Suppose $f(\cdot,\cdot) : \R^{q} \times \R^{p} \rightarrow \R$ is a Lebesgue-measurable function such that $f(\cdot,\bt) : \R^{q} \rightarrow \R$ is $P$-integrable for all $\bt \in \R^{p}$.  Denote the expectations $F(\bt) = \mathbb{E}_{P}[ f(\cdot,\bt)]$ and $F_n(\bt) = \mathbb{E}_{P_n}[f(\cdot,\bt)]$.  Let $\hat{\bt} \in \mathrm{argmin}_{\bt} \, F(\bt)$ be a minimizer of the population loss, let $\{ \hat{\bt}_n \}_{n=1}^{\infty}$, $\hat{\bt}_n \in \mathrm{argmin}_{\bt} \, F_n(\bt)$, be a sequence of empirical minimizers, and let $f(\cdot,\bt) = f(\cdot,\hat{\bt}) + \langle \bt - \hat{\bt}, \Delta(\cdot) \rangle + \| \bt - \hat{\bt} \|_{2} r(\cdot,\bt)$ denote the linearization of $f(\cdot,\bt)$ about $\hat{\bt}$ (we assume that $f(\cdot,\bt)$ is differentiable with respect to $\bt$ at $\hat{\bt}$ for $P$-a.e., with derivative denoted by $\Delta(\cdot)$).  Let $\mathcal{D} = \{ d_{\bt, \hat{\bt}}(\cdot) ~|~ \bt \in B_{\| \cdot\|_{2}} (\hat{\bt}) \backslash \{\hat{\bt}\} \}$, where $d_{\bt_1,\bt_2}(\cdot) = (f(\cdot,\bt_1) - f(\cdot,\bt_2)) / \|\bt_1 - \bt_2 \|_{2}$, denote the collection of divided differences.
	
	Suppose (i) $\hat{\bt}_n \rightarrow \hat{\bt}$ a.s., (ii) the Hessian $\nabla^{2}:=\nabla^{2}F(\bt)|_{\bt = \hat{\bt}}$ about $\hat{\bt}$ is positive definite, (iii) the collection of sets $\{\{ (\cdot, s) ~|~ d_{\bt,\hat{\bt}}(\cdot) \geq s \geq 0 \text{ or } d_{\bt,\hat{\bt}}(\cdot) \leq s \leq 0 \} ~|~ \bt \in B_{\| \cdot\|_{2}} (\hat{\bt}) \backslash \{ \hat{\bt} \} \}$ form a VC class, and (iv) there is a function $\bar{d}(\cdot): \R^q \rightarrow \R$ such that $|d(\cdot)|\leq \bar{d}(\cdot)$ for all $d \in \mathcal{D}$, $|\langle \Delta(\cdot), \be \rangle| \leq \bar{d}(\cdot)$ for all unit-norm vectors $\be$, $\bar{d}(\cdot)>0$, and $\bar{d}(\cdot)\in \mathcal{L}^{2}(P)$.  Then $\sqrt{n} (\hat{\bt}_n - \hat{\bt}) \overset{D}{\rightarrow} \mathcal{N}(0, (\nabla^{2})^{-1} [ (\mathbb{E}_{P}[\Delta \Delta^{\intercal}]) - (\mathbb{E}_{P}[\Delta])(\mathbb{E}_{P}[\Delta])^{\intercal} ] (\nabla^{2})^{-1} )$.
\end{proposition}


The proof of Proposition \ref{thm:vc_to_stochequi} is based on a series of ideas developed in \cite{Pol:84} [Ch.~VII] (see in particular Example 19).  These rely on computing entropy numbers for certain function classes.  More formally, given a probability measure $P$ and a collection of functions $\mathcal{F}$ with each member being a mapping from $\R^q$ to $\R$, we define $\entr (\epsilon, P, \mathcal{F})$ to be the size of the smallest $\epsilon$-cover of $\mathcal{F}$ in the $\mathcal{L}^{2}(P)$-distance.  As these steps builds on substantial background material from \cite{Pol:84}, we state the key conceptual arguments in the following proof and refer the reader to specific pointers in the literature for further details.

\begin{proof}[Proof of Proposition \ref{thm:vc_to_stochequi}]  To simplify notation in the following proof, we denote $B_{\hat{\bt}} := B_{\|\cdot\|_{2}}(\hat{\bt}) \backslash \{ \hat{\bt} \}$.
	
	The graphs of $\mathcal{D} := \{ d_{\bt,\hat{\bt}}(\cdot) ~|~ \bt \in B_{\hat{\bt}} \}$ form a VC class by assumption.  By \cite{Pol:84} [Corollary~17,~p.~20], these graphs have polynomial discrimination (see \cite{Pol:84} [Ch.~II]).  In addition, the collection $\mathcal{D}$ have an enveloping function $\bar{d}$ by assumption.  Hence by \cite{Pol:84} [Lemma~36,~p.~34], there exists constants $\alpha_{\mathcal{D}}$,$\beta_{\mathcal{D}}$ such that $\entr (\delta ( \mathbb{E}_{P_n} [\bar{d}^2])^{1/2}, P_n, \mathcal{D}) \leq \alpha_{\mathcal{D}} (1/\delta)^{\beta_{\mathcal{D}}}$, for all $0 < \delta \leq 1$ and all $n$.
	
	We obtain a similar bound for graphs of the collection $\mathcal{E} := \{ \langle \Delta (\cdot) ,(\bt-\hat{\bt})/\|\bt-\hat{\bt}\|_2 \rangle ~|~ \bt \in B_{\hat{\bt}} \}$.  First, we note that $\mathcal{E}$ is a subset of a finite dimensonal vector space.  By \cite{Kos:08} [Lemma~9.6,~p.~159], the collection of subgraphs $\{ \{ (t,s)~|~f(t) \leq s \} ~|~ f \in \mathcal{E} \}$ forms a VC class.  Then, by noting that the singleton $\{ \{ (t,s)~|~s \leq 0 \} \}$ is a VC class, and that the collection of sets formed by taking intersections of members of two VC classes is also a VC class (see \cite{Kos:08} [Lemma~9.7,~p.~159]), we conclude that the collection of sets $\{ \{ (t,s)~|~f(t) \leq s \leq 0 \} ~|~ f \in \mathcal{E} \}$ is a VC class.  A similar sequence of steps shows that the collection of sets $\{ \{ (t,s)~|~f(t) \geq s \geq 0 \} ~|~ f \in \mathcal{E} \}$  also forms a VC class.  The collection of graphs of $\mathcal{E}$ is the union of the previous two collections.  Hence by \cite{Kos:08} [Lemma~9.7,~p.~159] it is also a VC class.	
	Subsequently, by applying the same sequence of steps as we did for $\mathcal{D}$, there exists constants $\alpha_{\mathcal{E}}$,$\beta_{\mathcal{E}}$ such that $\entr (\delta (\mathbb{E}_{P_n}[\bar{d}^2])^{1/2}, P_n, \mathcal{E}) \leq \alpha_{\mathcal{E}} (1/\delta)^{\beta_{\mathcal{E}}}$, for all $0 < \delta \leq 1$ and all $n$.
	
	Next, consider the collection of functions
	$\mathcal{F} := \{ r(\cdot,\bt) ~|~ \bt \in B_{\hat{\bt}} \}$.  We have
	$r(\cdot,\bt) = d_{\bt,\hat{\bt}}(\cdot) + \langle \Delta(\cdot), ( \bt - \hat{\bt} ) / \|\bt - \hat{\bt}\|_2 \rangle$, and hence every element in $\mathcal{F}$ is expressible as a sum of functions in $\mathcal{D}$ and $\mathcal{E}$ respectively.  Given $\delta/\sqrt{2}$-covers for $\mathcal{D}$ and $\mathcal{E}$ in any $\mathcal{L}^{2}$-distance, one can show via the AM-GM inequality that the union of these covers forms a $\delta$-cover for $\mathcal{F}$.  Subsequently, we conclude that there exists constants $\alpha$ and $\beta$ such that $\entr (\delta (\mathbb{E}_{P_n}[\bar{d}^2])^{1/2},P_n,\mathcal{F}) \leq \alpha (1/\delta)^{\beta}$.
	
	The remaining sequence of arguments is identical to those in of \cite{Pol:84} [Ch.~VII]. Our bound on the quantity $\entr (\delta (\mathbb{E}_{P_n}[\bar{d}^2])^{1/2},P_n,\mathcal{F})$ implies that there exists $\theta$ such that $\int_{0}^{\theta} ( \entr (t,P_n,\mathcal{F}) / t)^{1/2} dt < \epsilon$ for all $n$ and all $\epsilon>0$.  We apply \cite{Pol:84} [Lemma~15,~p.~150] to obtain the following \emph{stochastic equicontinuity} property: given $\gamma>0$ and $\epsilon>0$, there exists $\theta$ such that
	\begin{equation*}
	\underset{n}{\lim \sup} ~ \prob \left[ \underset{f_1,f_2 \in \mathcal{F} : \mathbb{E}_{P} [(f_1-f_2)^2] \leq \theta}{\sup} | \mathbb{E}_{E_n} [f_1-f_2]  | > \gamma \right] < \epsilon.
	\end{equation*}
	Here, $E_n$ denotes the signed measure $\sqrt{n}(P_n-P)$.  One can check measurability of the inner supremum using the notion of \emph{permissibility} -- see \cite{Pol:84} [Appendix~C] (in essence, we simply require $f$ to be measurable and the index $\bt$ to reside in an Euclidean space).  Finally we apply \cite{Pol:84} [Theorem~5,~p.~141] to conclude the result.
\end{proof}

\begin{proof}[Proof of Theorem \ref{thm:stats_norm}]
	The first step is to verify that the sequence $\{ \tilde{A}_n \}_{n=1}^{\infty}$ quotients out the appropriate equivalences.  Since $\tilde{A}_n \in \hat{A}_n \orb$, we have $\tilde{A}_n = \hat{A}_n g_n$ for an isometry $g_n$.  Subsequently we have $\|\tilde{A}_n - \hat{A} \|_{F} = \|\hat{A}_n g_n - \hat{A} \|_{F} = \|\hat{A}_n - \hat{A} g_n^{-1} \|_{F} \leq d_{H} (\hat{A}_n,\mc)$.  Following the conclusions of Theorem \ref{thm:stats_setconvergence}, we have $\tilde{A}_n \rightarrow \hat{A}$ a.s.  Furthermore, as a consequence of the optimality conditions in the definition of $\tilde{A}_n$, we also have $\tilde{A}_n - \hat{A} \in \nspace$.
	
	The second step is to apply Proposition \ref{thm:vc_to_stochequi} to the sequence $\{\tilde{A}_n\}_{n=1}^{\infty}$ with $(h_{\C} (A^{\intercal}\bu) - y)^2$ as the choice of loss function $f(\cdot,\bt)$, $\pk$ as the probability measure $P$, $(\bu,y)$ as the argument, and $A$ as the index $\bt$.  First, the measurability of the loss as a function in $A$ and $(\bu,y)$ is straightforward to establish.  Second, the differentiablity of the loss function at $\hat{A}$ for $\pk$-a.e. $(\bu,y)$ follows from the assumption that $h_{\C}(\cdot)$ is differentiable at $\hat{A}^{\intercal}\bu$ for $\pk$-a.e. $\bu$.  Third, we have shown that $\tilde{A}_n \rightarrow \hat{A}$ a.s. in the above.  Fourth, the Hessian $\Gamma|_{\nspace}$ is positive definite by assumption.  Fifth, the graphs of $\{ d_{\C,A,\hat{A}}(\bu,y) ~|~ A \in B_{\|\cdot\|_{F}}(\hat{A}) \backslash \{ \hat{A} \} \}$ form a VC class by assumption.  Sixth, we need to show the existence of an appropriate function $d(\cdot)$ to bound the divided differences $d_{\C,A,\hat{A}}(\cdot)$ and the inner products $\langle \nabla(\cdot) |_{A=\hat{A}}, E \rangle$, where $\|E\|_{F} = 1$.  In the former case, we note that $|(h_{\C} (A_1^{\intercal}\bu) - y) + (h_{\C} (A_2^{\intercal}\bu) - y)| \leq \|A_1 - A_2 \|_{\C,2} \leq c_1\|A_1-A_2\|_{F}$ for some $c_1>0$; here, the first inequality follows from Lemma \ref{thm:linmap_diff} and the second inequality follows from the equivalence of norms.  Then, by noting that $A \in B_{\|\cdot\|_{F}}(\hat{A}) \backslash{\hat{A}}$ is bounded, the expression $| (h_{\C} (A^{\intercal}\bu) - y) + (h_{\C} (\hat{A}^{\intercal}\bu) - y) |$ is bounded above by a function of the form $c_2(1+|y|)$.  By expanding the divided difference expression and by combining the previous two bounds, one can show that $|d_{\C,A,\hat{A}}(\cdot)| \leq c_3(1+|y|)$ for some $c_3 > 0$.  In the latter case, the derivative is given by $2(h_{\C} (\hat{A}^{\intercal}\bu) - y) \bu \otimes \emax(\hat{A}^{\intercal}\bu)$.  By noting that $\lmax(\hat{A}^{\intercal}\bu)$, $\bu$, and $\emax(\hat{A}^{\intercal}\bu)$ are uniformly bounded over $\bu \in \sph^{d-1}$, and by performing a sequence of elementary bounds, one can show that $\langle \nabla, E \rangle$ is bounded above by $c_4(1+|y|)$ uniformly over all unit-norm $E$.  We pick $\bar{d}(\cdot)$ to be $c(1+|y|)$, where $c = \max\{1,c_3,c_4\}$.  Then $\bar{d}(\cdot) > 0$ by construction, and furthermore, $\bar{d} \in \mathcal{L}^2(\pk)$ as $\E_{\pk}[ \varepsilon^2 ]<\infty$.
	
	Finally, the result follows from an application of Proposition \ref{thm:vc_to_stochequi}.
\end{proof}

This result gives an asymptotic normality characterization corresponding to a sequence of minimizers $\{\hat{A}_n\}_{n=1}^\infty$ of the empirical losses $\pc (\cdot,\pnk)$.  In the next result, we specialize this result to the setting in which the underlying set $\KS$ is in fact expressible as a projection of $\C$, i.e., $\KS = A^{\star}(\C)$ for some $A^\star \in \lqd$.  This specialization leads to a particularly simple formula for the asymptotic error covariance, and we demonstrate its utility in the examples in Section \ref{sec:stats_asymnormal_examples}.

\begin{corollary}\label{thm:simplifiedasymnorm}
	Suppose that the conditions of Theorem \ref{thm:stats_norm} and of Proposition \ref{thm:stats_specialcasev} hold.  Using the notation of Theorem \ref{thm:stats_norm}, we have that $\E [ \nabla \otimes \nabla|_{A=A^{\star}} ] = 2\sigma^2 \Gamma$ with $\Gamma$ given by \eqref{eq:secondderivformula}.  In particular, the conclusion of Theorem \ref{thm:stats_norm} simplifies to $\sqrt{n} (\tilde{A}_n - \hat{A}) \overset{D}{\rightarrow} \mathcal{N}(0, 2\sigma^2 (\Gamma|_{\nspace})^{-1})$.
\end{corollary}

\begin{proof}[Proof of Corollary \ref{thm:simplifiedasymnorm}]
	One can check that $\nabla_{A} ((\lmax (A^{\intercal}\bu) - y)^2) |_{A=A^{\star}} = - \varepsilon \bu \otimes \emax (A^{\star\intercal}\bu) $, from which we have that $\E [\nabla \otimes \nabla |_{A=A^{\star}} ] = 2\sigma^2 \Gamma$.  This concludes the result.
\end{proof}

\subsubsection{Examples} \label{sec:stats_asymnormal_examples}

Here we give examples that highlight various aspects of the theoretical results described previously.  In all of our examples, the noise $\{\varepsilon^{(i)}\}_{i=1}^{n}$ is given by i.i.d. centered Gaussian random variables with variance $\sigma^2$.  We begin with an illustration in which the assumptions of Theorem~\ref{thm:stats_norm} are not satisfied and the asymptotic normality characterization fails to hold:
\begin{example}
	Let $\KS:=\{0\} \subset \R$ be a singleton.  As $\sph^{0} \cong \{ -1,1\}$, the random variables $\bu^{(i)}$ are $\pm 1$ u.a.r.  Further, $h_{\KS}(\bu)=0$ for all $\bu$ and the support function measurements are simply $y^{(i)} = \varepsilon^{(i)}$ for all $i=1,\dots,n$.  For $\C$ being either a simplex or a spectraplex, the set $\mc =\{ 0 \} \subset L(\R^q,\R)$ is a singleton consisting only of the zero map.  First, we consider fitting $\KS$ with the choice of $\C = \simp^{1} \subset \R^1$.  Then we have $\hat{A}_n = \frac{1}{n} \sum_{i=1}^{n} \varepsilon^{(i)} \bu^{(i)}$, from which it follows that $\sqrt{n}(\hat{A}_n - 0)$ is normally distributed with mean zero and variance $\sigma^2$ -- this is in agreement with Theorem \ref{thm:stats_norm}.  Second, we consider fitting $\KS$ with the choice $\C = \simp^{2} \subset \R^2$.  Define the sets $U_{-} = \{ i | \bu_i = -1 \}$, and $U_{+} = \{ i | \bu_i = 1 \}$, and define
	\begin{equation*}
		\alpha_{-} = -\frac{1}{|U_{-}| }\sum_{i \in U_{-}} \varepsilon^{(i)} \quad \text{and} \quad \alpha_{+} = \frac{1}{|U_{+}| }\sum_{j \in U_{+}} \varepsilon^{(j)}.
	\end{equation*}
	Then $\KHN = \{ x | \alpha_{-} \leq x \leq \alpha_{+} \}$ if $\alpha_{-} \leq \alpha_{+}$ and $\KHN = \{  \frac{1}{n} \sum_{i=1}^{n} \varepsilon^{(i)} \bu^{(i)} \}$ otherwise.  Notice that $\alpha_{-}$ and $\alpha_{+}$ have the same distribution, and hence $\KHN$ is a closed interval with non-empty interior w.p. $1/2$, and is a singleton w.p. $1/2$.  Thus, one can see that the linear map $\hat{A}_n$ does not satisfy an asymptotic normality characterization.  The reason for this failure is that the function $\pc(\cdot,\pk)$ is twice differentiable everywhere excluding the line $\{ (c,c) | c \in \R \}$; in particular, it is not differentiable at the minimizer $(0,0)$.
\end{example}

The above example is an instance where the function $\pc(\cdot,\pk)$ is not twice differentiable at $\hat{A}$.  The manner in which an asymptotic characterization of $\{\KHN \}_{n=1}^{\infty}$ fails in instances where $\mc$ contains multiple orbits is also qualitatively similar.  Next, we consider a series of examples in which the conditions of Theorem \ref{thm:stats_norm} hold, thus enabling an asymptotic normality characterization behavior of $\KHN$.  To simplify our discussion, we describe settings in which the choices of $\C$ and $A^{\star}$ satisfy the conditions of Corollary \ref{thm:simplifiedasymnorm}, which leads to simple formulas for the second derivative $\Gamma$ of the map $A\mapsto \pc(A,\pk)$ at $A^{\star}$.

\subsubsection*{Polyhedral examples}

We present two examples in which $\KS$ is polyhedral, and we choose $\C = \simp^{q}$ where $q$ is the number of vertices of $\KS$.  With this choice, the set $\mc$ comprises linear maps $A \in \lqd$ whose columns are the extreme points of $\KS$.  

\begin{proposition} \label{thm:poly_gamma}
Let $\KS$ be a polytope with $q$ extreme points, and let $A^{\star}$ be the linear map whose columns are the extreme points of $\KS$ (in any order).  The second derivative of the function $\Phi_{\C}(\cdot,P)$ at $A^{\star}$ is given by $\Gamma (D) = \sum_{j=1}^{q} \int_{\bu \in H_{A^{\star},j} } \langle \bu \otimes \be_j , D \rangle \bu \otimes \be_j ~ d \bu$, where $H_{A,j} := \{ \bu \,|\, \bu \in \sph^{d-1}, \bu^{\intercal} A \be_j > \bu^{\intercal} A \be_k \text{ for all } k \neq j \}$.
\end{proposition}

\begin{proof}
Let $A\in \lqd$ be a linear map whose columns are pairwise distinct.  Then the set $V_{A} := \{ \bu \,|\, \bu^{\intercal} A \be_i = \bu^{\intercal} A \be_j , 1 \leq i < j \leq d \} $ is a subset of $\sph^{d-1}$ with measure zero.  We note that the function $h_{\C}(\cdot)$ is differentiable at $A^{\intercal}\bu$ whenever $\bu \in \sph^{d-1} \backslash \{V_{A}\}$, and hence the function $h_{\C}(A^{\intercal}\bu)$ is differentiable with respect to $\bu$ for $P$-a.e..  By Proposition \ref{thm:formulation_diffble}, the derivative of $\Phi_{\C}(\cdot,P)$ at $A$ is $2 \E_{P} [(\lmax (A^{\intercal}\bu) - y) \bu \otimes \emax (A^{\intercal}\bu) ]$.

The subset of linear maps whose columns are pairwise distinct is open.  Hence the derivative of $\Phi_{\C}(\cdot,P)$ exists in a neighborhood of $A^{\star}$.  It is clear that the derivative is also continuous at $A^{\star}$.  Finally, we note that the collection $\{H_{A^{\star},j}\}_{j=1}^{q}$ forms a disjoint cover of $\sph^{d-1}$ up to a subset of measure zero, and apply Proposition \ref{thm:stats_specialcasev} to conclude the expression for $\Gamma$.
\end{proof}

We note that the operator $\Gamma$ has block diagonal structure since each integral $\int_{\bu \in H_{A^{\star},j} } \langle \bu \otimes \be_j , D \rangle \bu \otimes \be_j ~ d \bu$ is supported on a $d\times d$-dimensional sub-block.  Subsequently, by combining the conclusions of Theorem \ref{thm:stats_setconvergence} and Proposition \ref{thm:poly_gamma}, we can conclude the following about $\KHN$: (i) it is a polytope with $q$ extreme points, (ii) each vertex of $\KHN$ is close to a distinct vertex of $\KS$, (iii) the deviations (after scaling by a factor of $\sqrt{n}$) between every vertex-vertex pair are asymptotically normal with inverse covariance specified by a $d\times d$ block of $\Gamma$, and further these deviations are pairwise-independent.

\begin{example}
	Let $\KS$ be the regular $q$-gon in $\R^2$ with vertices $\bv_{k} : = (\cos (2k\pi/q), \sin (2k\pi/q) )^{\intercal}, ~ k=0,\dots,q-1$.  Let $\hat{\bv}_{n,k}$ be the vertex of $\KHN$ closest to $\bv_{k}$.  The deviation $\sqrt{n} (\hat{\bv}_{n,k} - \bv_{k})$ is asymptotically normal with covariance $2\sigma^2 M_{k,k}^{-1}$, where:
	\begin{equation*}
		M_{k,k} = \frac{1}{q} I + \frac{1}{2\pi} \sin (2\pi /q) \left( \begin{array}{cc}
			\cos (4k\pi/q) \ & \sin (4k\pi/q)  \\
			\sin (4k\pi/q) & - \cos (4k\pi/q)
		\end{array} \right).
	\end{equation*}
	The eigenvalues of $M_{k,k}$ are $1/q + (1/2\pi)\sin (2\pi/q)$ and $1/q - (1/2\pi)\sin (2\pi/q)$, and the corresponding eigenvectors are $(\cos(2k\pi/q), \sin (2k\pi/q) )^{\intercal}$ and $(\sin(2k\pi/q), - \cos (2k\pi/q) )^{\intercal}$ respectively.  Consequently, the deviation $\hat{\bv}_{n,k} - \bv_{k}$ has magnitude $\approx \sigma \sqrt{q/n}$ in the direction $\bv_{k}$, and has magnitude $\approx \sigma \sqrt{3q^3/\pi^2 n}$ in the direction $\bv_{k}^{\perp}$.  Figure \ref{fig:sec3_5gon} shows $\KS$ as well as the confidence intervals (ellipses) of the vertices of $\KHN$ for large $n$ for $q=5$.
	
	\begin{figure}[h]
		\centering
		\includegraphics[width=0.2\textwidth]{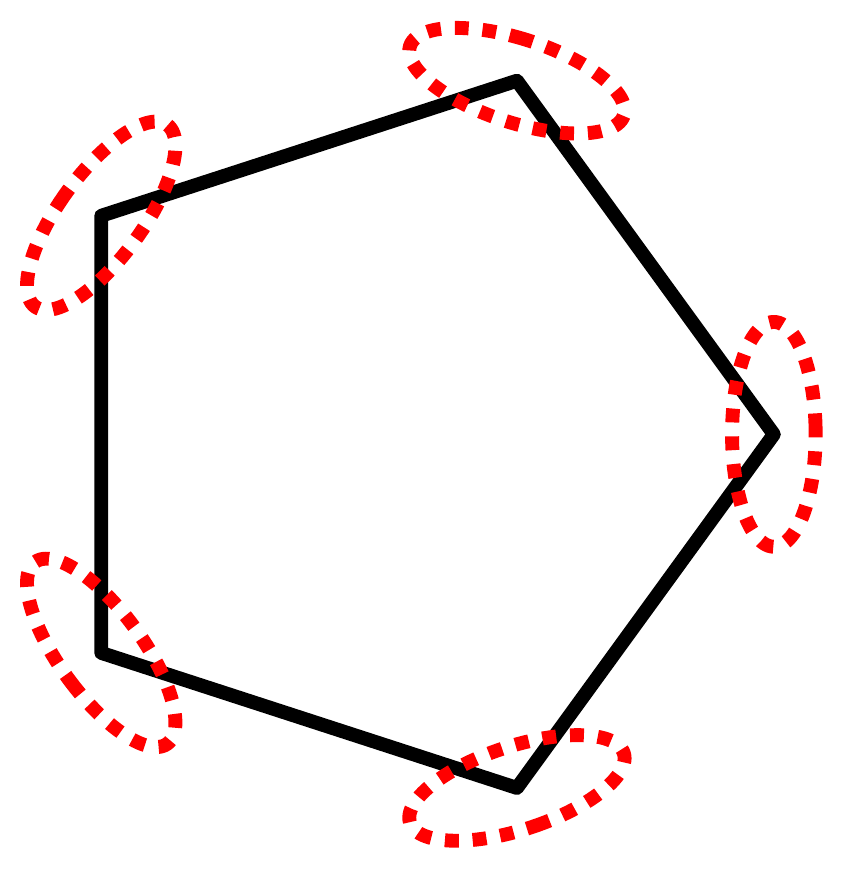}
		\caption{Estimating a regular $5$-gon as the projection of $\simp^{5}$.  In the large $n$ limit, the estimator $\KHN$ is a $5$-gon.  The typical deviation of the vertices of $\KHN$ from that of the $5$-gon (scaled by a factor of $\sqrt{n}$) is represented by the surrounding ellipses.}
		\label{fig:sec3_5gon}
	\end{figure}
\end{example}

\begin{example}
	Let $\KS$ be the $\ell_{\infty}$-ball in $\R^d$.  For any vertex $\bv$ of $\KS$, let $\hat{\bw}_{n,\bv}$ denote the vertex of $\KHN$ closest to $\bv$.  The deviation $\sqrt{n}(\hat{\bw}_{n,\bv} - \bv)$ is asymptotically normal with covariance $2\sigma^2 M_{\bv, \bv}^{-1}$, where:
	\begin{equation*}
		M_{\bv, \bv} = \frac{1}{2^d d}((1-2/\pi) I + (2/\pi) \bv \bv^{\intercal}).
	\end{equation*}
	Hence the deviation $\hat{\bw}_{n,\bv} - \bv$ has magnitude $\approx \sigma 2^{(d+1)/2}(2/\pi + (1-2/\pi)/d)^{-1/2} n^{-1/2}$ in the span of $\bv$ and magnitude $\approx \sigma  2^{(d+1)/2}(1-2/\pi)^{-1/2} \sqrt{d/n} $ in the orthogonal complement $\bv^\perp$.
\end{example}

\subsubsection*{Non-polyhedral examples}

Next we present two examples in which $\KS$ is non-polyhedral.  Unlike the previous polyhedral examples in which columns of the linear map $\tilde{A}_n$ map directly to vertices, our description of $\KHN$ requires a different interpretation of Corollary \ref{thm:simplifiedasymnorm} that is suitable for sets with infinitely many extreme points.  Specifically, we characterize the deviation $\sqrt{n} (\tilde{A}_n - A^\star)$ in terms of a perturbation to the set $\KS = A^{\star}(\C)$ by considering the image of $\C$ under the map $\sqrt{n} (\tilde{A}_n - A^\star)$.

\begin{example}  Suppose $\KS = B_{\| \cdot \|_2} (\bc)$ is unit $\ell_2$-ball in $\R^d$ with center $\bc$.  We consider $\C := \{ (1,\bv)^{\intercal} | \|\bv\|_2 \leq 1 \} \subset \R^{d+1}$, and $A^{\star}$ is any linear map of the form $[\bc ~Q] \in L(\R^{d+1},\R^d)$, where $Q \in O(d,\R)$ is any orthogonal matrix.  Then the restricted Hessian $\Gamma|_{\nspace}$ is a self-adjoint operator with rank $d + {d+1 \choose 2}$.  The eigenvectors of $\Gamma|_{\nspace}$ represent `modes of oscillations', which we describe in greater detail.  We begin with the case $d=2, \bc=0$.  The set resulting from the deviation $\sqrt{n} (\tilde{A}_n - A^\star)$ applied to $\C$ can be decomposed into $5$ different modes of perturbation (these exactly correspond to the eigenvectors of the operator $\Gamma|_{\nspace}$).  Parametrizing the extreme points of $A^{\star}(\C)$ by $\{(\cos (\theta) ~ \sin (\theta))^{\intercal} ~|~ \theta \in [0,2\pi)\}$, the contribution of each mode at the point $(\cos (\theta) ~ \sin (\theta))^{\intercal}$ is a small perturbation in the directions $(1~0)^{\intercal}$, $(0~1)^{\intercal}$, $(\cos(\theta)~\sin(\theta))^{\intercal}$, $( \cos(\theta) ~ -\sin (\theta))^{\intercal}$, and $(\sin (\theta) ~ \cos (\theta))^{\intercal}$ respectively -- Figure \ref{fig:Sec3_Circ} provides an illustration.  The first and second modes represent oscillations of $\KHN$ about $\bc$, the third mode represents dilation, and the fourth and fifth modes represent flattening.
	\begin{figure}[h]
		\begin{subfigure}{.2\textwidth}
			\centering
			\includegraphics[width=0.8\textwidth]{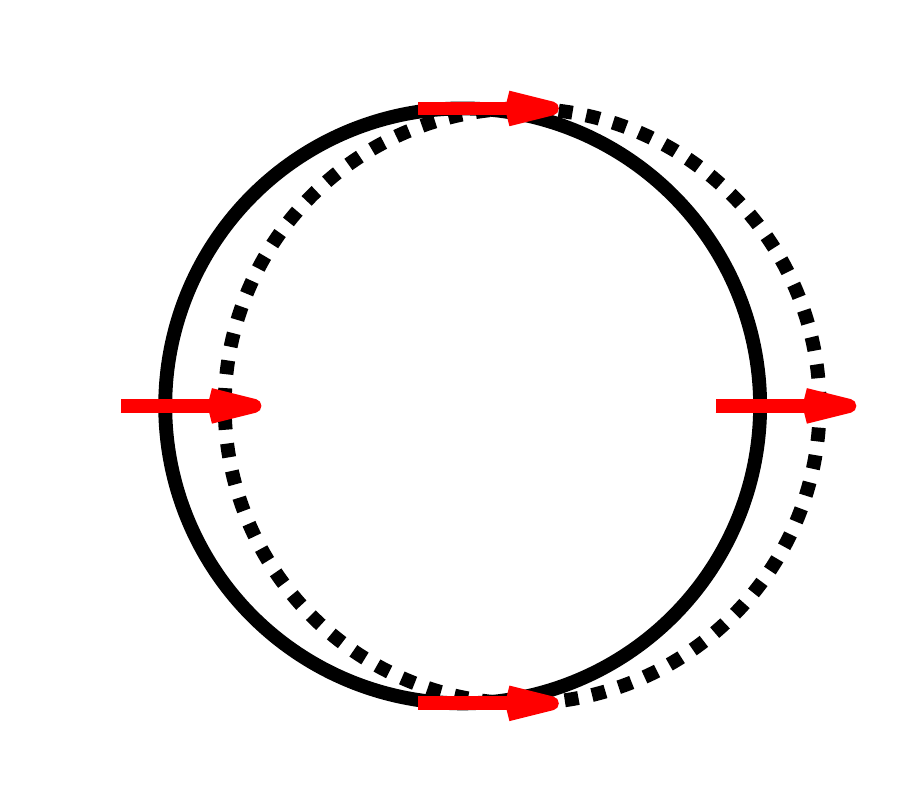}
			\caption{$(1~~0)^{\intercal}$}
		\end{subfigure}%
		\begin{subfigure}{.2\textwidth}
			\centering
			\includegraphics[width=0.8\textwidth]{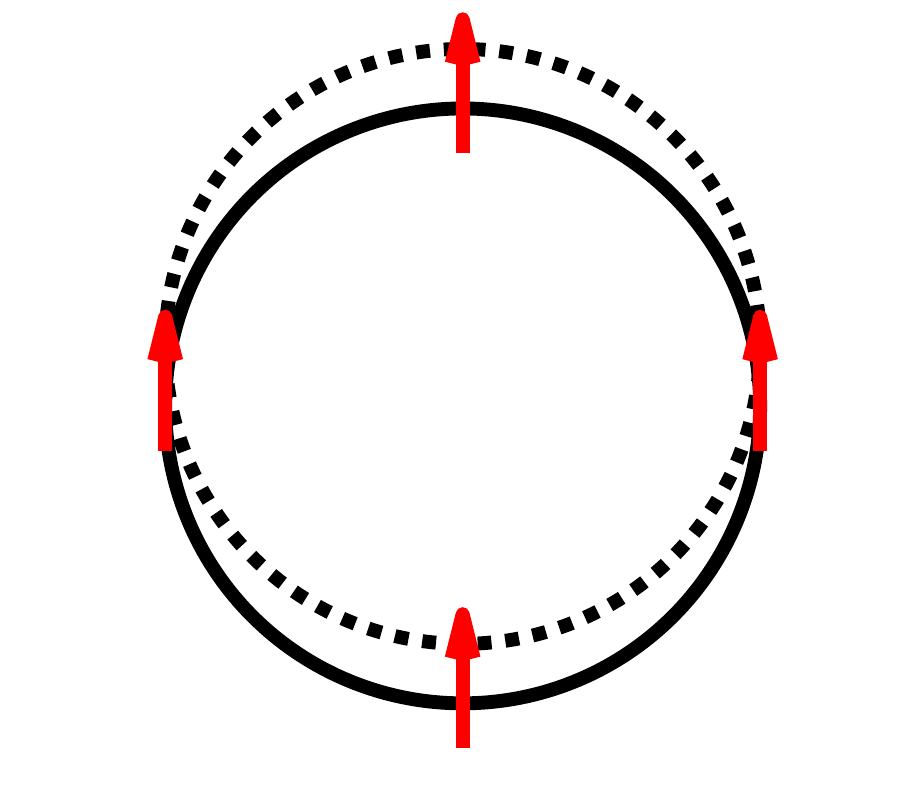}
			\caption{$(0~~1)^{\intercal}$}
		\end{subfigure}%
		\begin{subfigure}{.2\textwidth}
			\centering
			\includegraphics[width=0.8\textwidth]{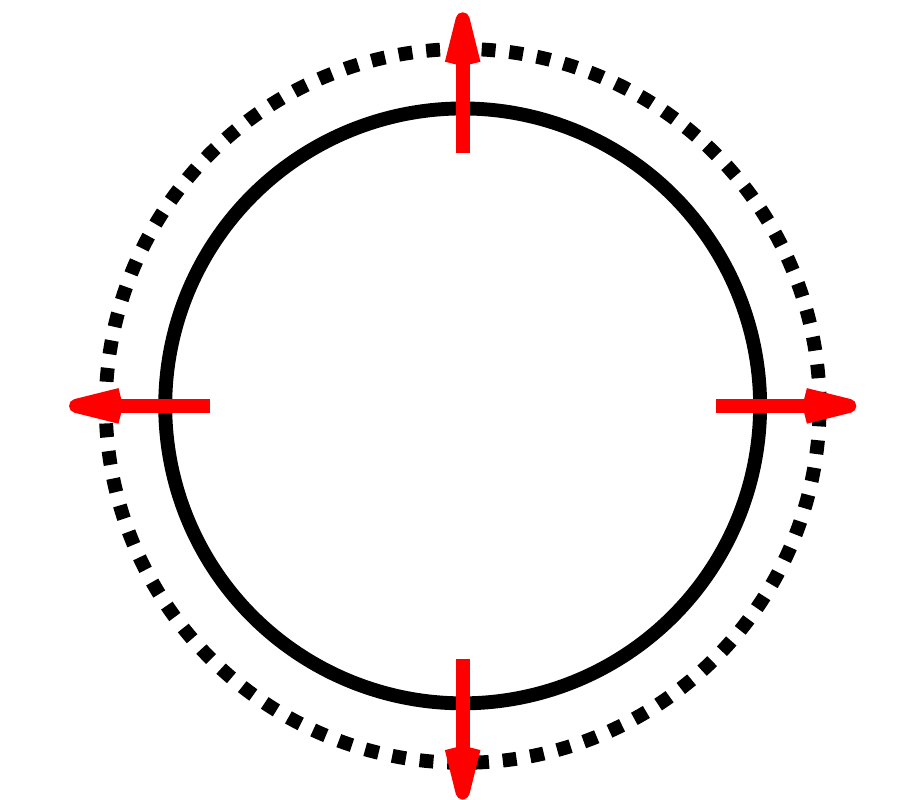}
			\caption{$(\cos(\theta)~\sin(\theta))^{\intercal}$}
		\end{subfigure}%
		\begin{subfigure}{.2\textwidth}
			\centering
			\includegraphics[width=0.8\textwidth]{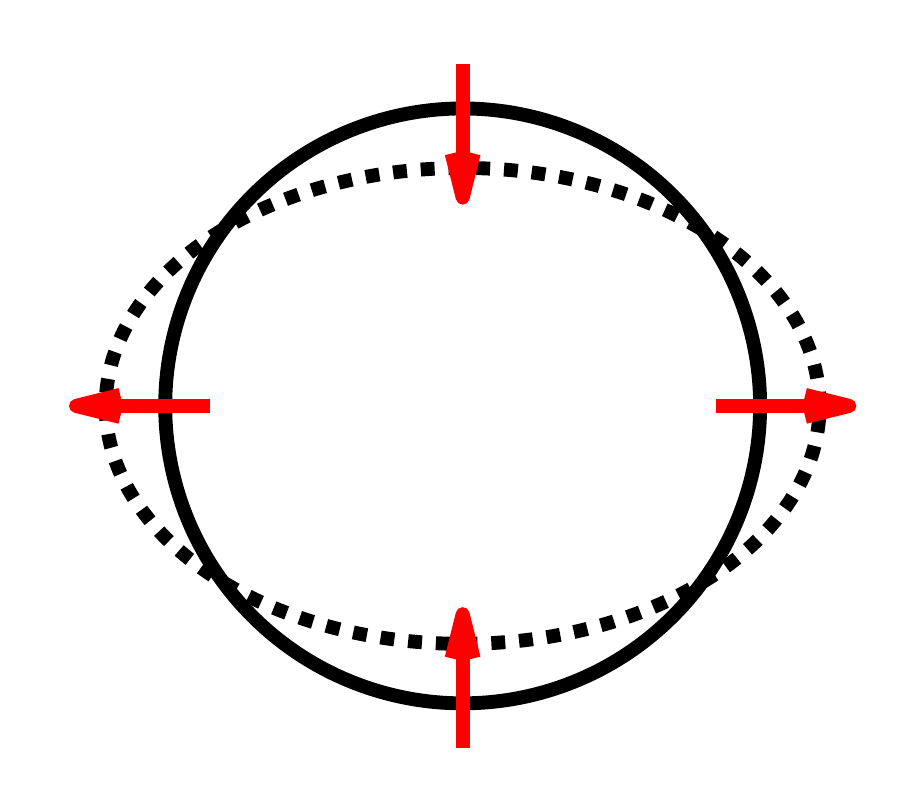}
			\caption{$( \cos(\theta) ~ -\sin (\theta))^{\intercal}$}
		\end{subfigure}%
		\begin{subfigure}{.2\textwidth}
			\centering
			\includegraphics[width=0.8\textwidth]{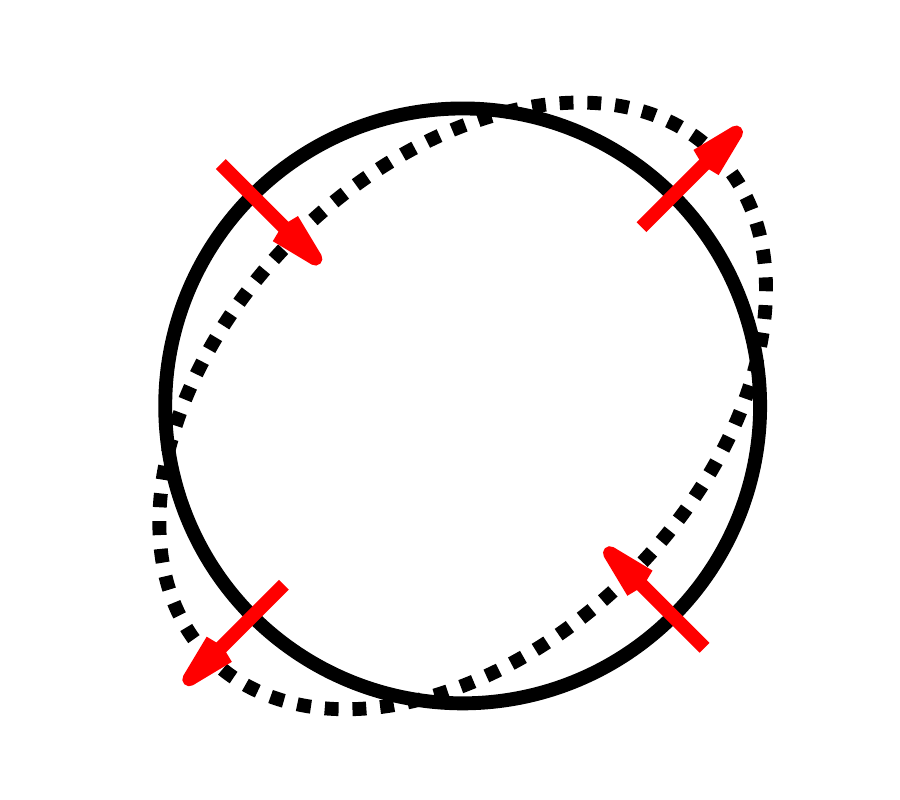}
			\caption{$(\sin (\theta) ~ \cos (\theta))^{\intercal}$}
		\end{subfigure}%
		\caption{Modes of oscillations for an estimate of the $\ell_2$-ball in $\R^2$.}
		\label{fig:Sec3_Circ}
	\end{figure}
	The analysis for a general $d$ is similar, with $d$ modes representing oscillations about $\bc$ and ${d+1\choose 2}$ modes whose contributions represent higher-dimensional analogs of flattening (of which dilation is a special case).
\end{example}

\begin{example}
	Let $\KS$ be the spectral norm ball in $\Sym^2$.  The extreme points of $\KS$ consist of three connected components: $\{I\}$, $\{-I\}$, and $\{ U D U^{\intercal} ~|~ U \in SO(2,\R) \}$ where $D$ is a diagonal matrix with entries $(1,-1)$.  To simplify our discussion, we apply a scaled isometry to $\KS$ so that $\{I\}$, $\{-I\}$, and $\{ U D U^{\intercal} ~|~ U \in SO(2,\R) \}$ are mapped to the points $\{(0,0,1)^{\intercal}\}$, $\{(0,0,-1)^{\intercal}\}$, and $\{(\cos(\theta),\sin(\theta),0)^{\intercal} ~|~ \theta \in [0,2\pi) \}$ in $\R^3$, respectively.  We choose
	\begin{equation} \label{eq:specnormcset}
	\begin{aligned}
	\C  ~ := & ~ \{ X | X \in \fs^{4}, X_{12} = X_{13} = X_{14} = X_{23} = X_{24} = X_{21} = X_{31} = X_{41} = X_{32} = X_{42} = 0 \} \\
	~ \cong & ~ \fs^1 \times \fs^1 \times \fs^2,
	\end{aligned}
	\end{equation}
	and $A^{\star}$ to be the map defined by $A^{\star}(X) =  ( \langle A_1,X \rangle, \langle A_2, X \rangle, \langle A_3, X \rangle )^{\intercal}$ where
	\begin{equation*} 
		\begin{aligned}
			A_1 = \left( \begin{array}{cccc}
				0 & 0 & 0 & 0 \\ 0 & 0 & 0 & 0 \\ 0 & 0 & 1 & 0 \\ 0 & 0 & 0 & -1
			\end{array} \right), \quad
			A_2 = \left( \begin{array}{cccc}
				0 & 0 & 0 & 0 \\ 0 & 0 & 0 & 0 \\ 0 & 0 & 0 & 1 \\ 0 & 0 & 1 & 0
			\end{array} \right), \quad
			A_3 = \left( \begin{array}{cccc}
				1 & 0 & 0 & 0 \\ 0 & -1 & 0 & 0 \\ 0 & 0 & 0 & 0 \\ 0 & 0 & 0 & 0
			\end{array} \right) .
		\end{aligned}
	\end{equation*}
	In the large $n$ limit, $\KHN$ is a convex set with extreme points $P_1$ and $P_2$ near $(0,0,1)^{\intercal}$ and $(0,0,-1)^{\intercal}$ respectively, and a set of extreme points specified by an ellipse $P_3$ near $\{(\cos(\theta),\sin(\theta),0)^{\intercal} ~|~ \theta \in [0,2\pi) \}$.	The operator $\Gamma|_\nspace$ is block diagonal with rank $14$ -- it comprises two $3$-dimensional blocks $\Gamma_1$ and $\Gamma_2$ associated with $P_1$ and $P_2$, and an $8$-dimensional block $\Gamma_3$ associated with $P_3$.  One conclusion is that the distributions of $P_1$, $P_2$, and $P_3$ are asymptotically independent.  Moreover, the deviations of $P_1$ and $P_2$ about $\{(0,0,1)^{\intercal}\}$ and $\{(0,0,-1)^{\intercal}\}$ are asymptotically normal with inverse covariance specified by $\Gamma_1$ and $\Gamma_2$ respectively.  We consider the behavior of $P_3$ in further detail.  The operator $\Gamma_3$ is the sum of an operator $\Gamma_{3,xy}$ with rank $5$ describing the variation of $P_3$ in the $xy$-plane, and another operator $\Gamma_{3,z}$ with rank $3$ describing the variation of $P_3$ in the direction of the $z$-axis.  The operator $\Gamma_{3,xy}$, when restricted to the appropriate subspace and suitably scaled, is equal to the operator we encountered in the previous example in the setting in which $\KS$ is the $\ell_2$-ball in $\R^2$.  The operator $\Gamma_{3,z}$ comprises a single mode representing oscillations of $P_{3}$ in the $z$ direction (see subfigure (b) in Figure \ref{fig:Sec3_Wobble}), and two modes representing ``wobbling'' of $P_3$ with respect to the $xy$-plane (see subfigures (c) and (d) in Figure \ref{fig:Sec3_Wobble}).  The set $\C$ we consider in this example is the intersection of the spectraplex $\fs^{4}$ with an appropriate subspace specified by \eqref{eq:specnormcset}.  A natural question is if the same analysis holds for $\C = \fs^{4}$.  Unfortunately, the introduction of additional dimensions introduces degeneracies (in the form of zero eigenvalues into $\Gamma|_{\nspace}$) which violates the requirements of Theorem \ref{thm:stats_norm}.
	\begin{figure}[h]
		\begin{subfigure}{.3\textwidth}
			\centering
			\includegraphics[width=0.6\textwidth]{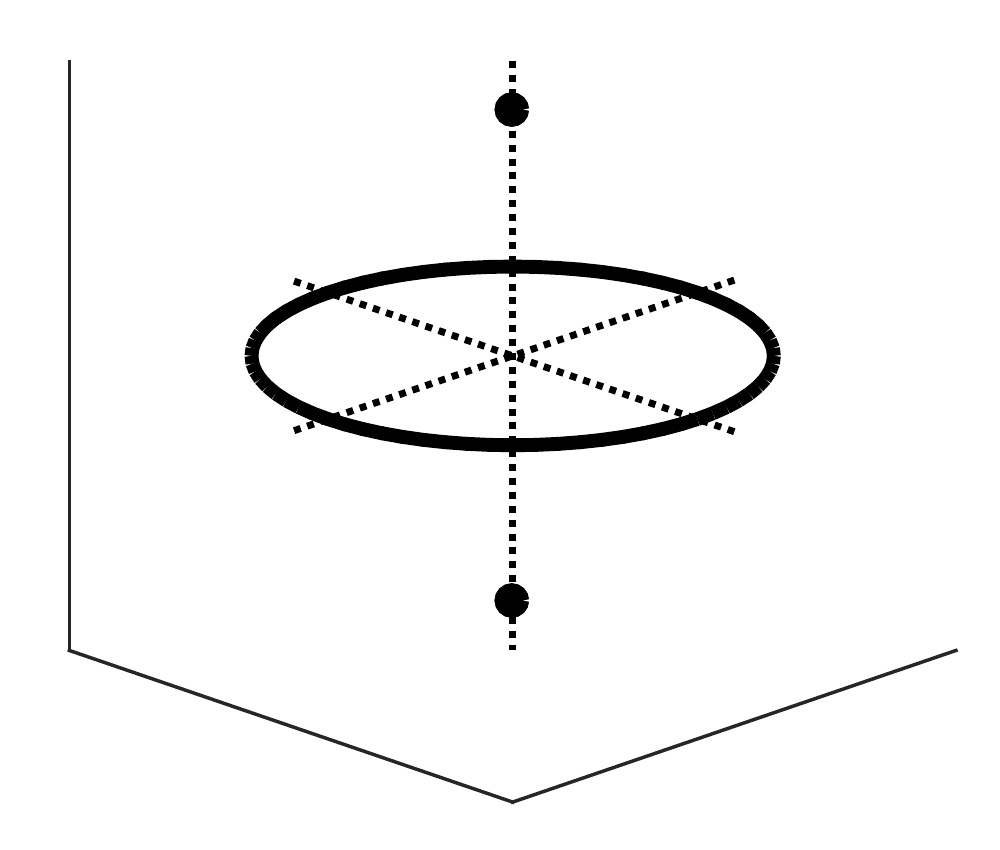}
			\caption{Extreme points of $\KS$}
		\end{subfigure}%
		\begin{subfigure}{.23\textwidth}
			\centering
			\includegraphics[width=0.8\textwidth]{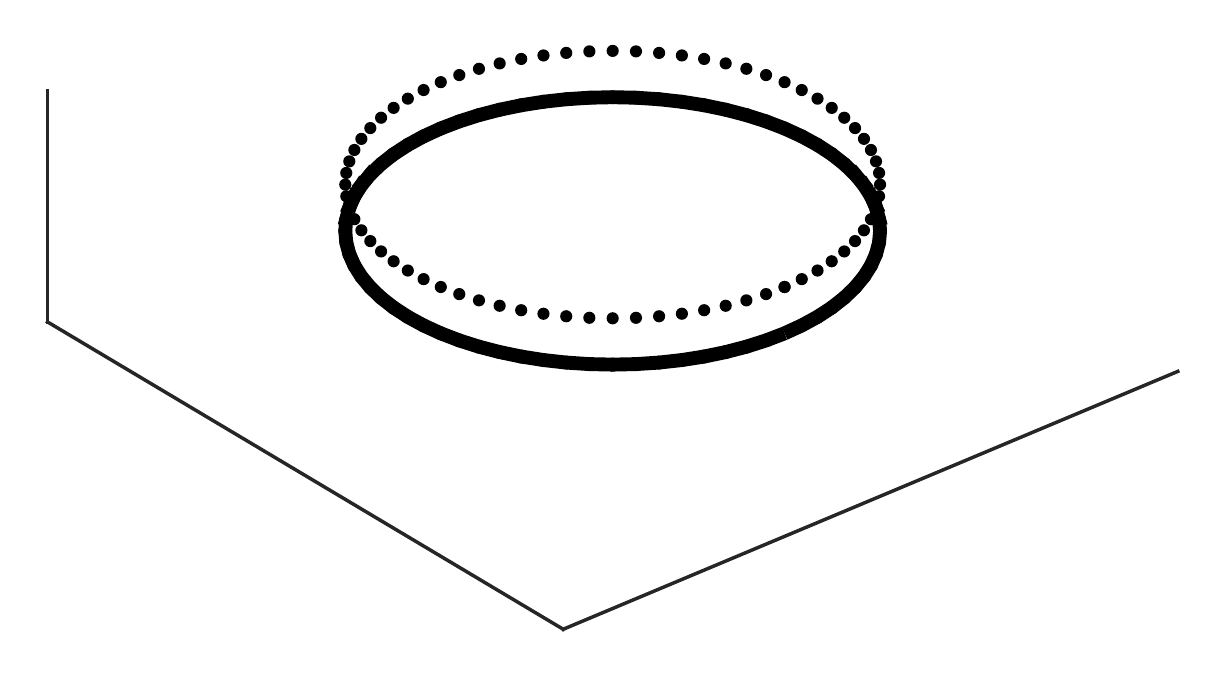}
			\caption{$(0,0,1)^{\intercal}$}
		\end{subfigure}%
		\begin{subfigure}{.23\textwidth}
			\centering
			\includegraphics[width=0.8\textwidth]{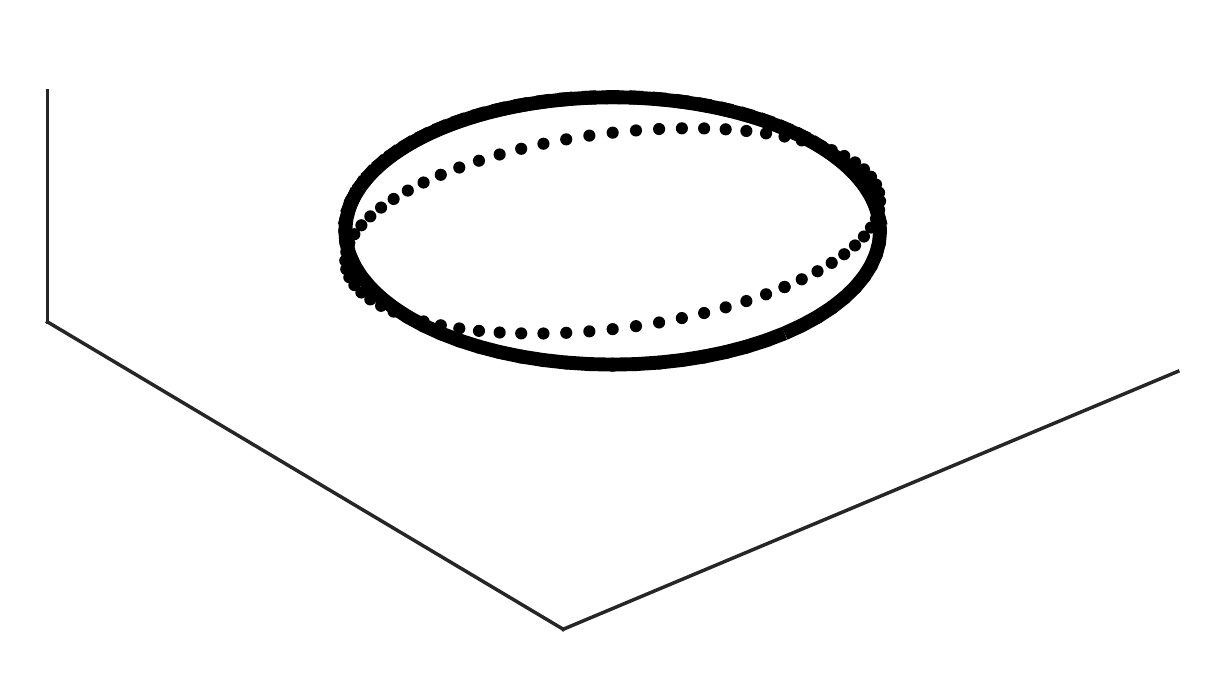}
			\caption{$(0,0,\cos(\theta))^{\intercal}$}
		\end{subfigure}%
		\begin{subfigure}{.23\textwidth}
			\centering
			\includegraphics[width=0.8\textwidth]{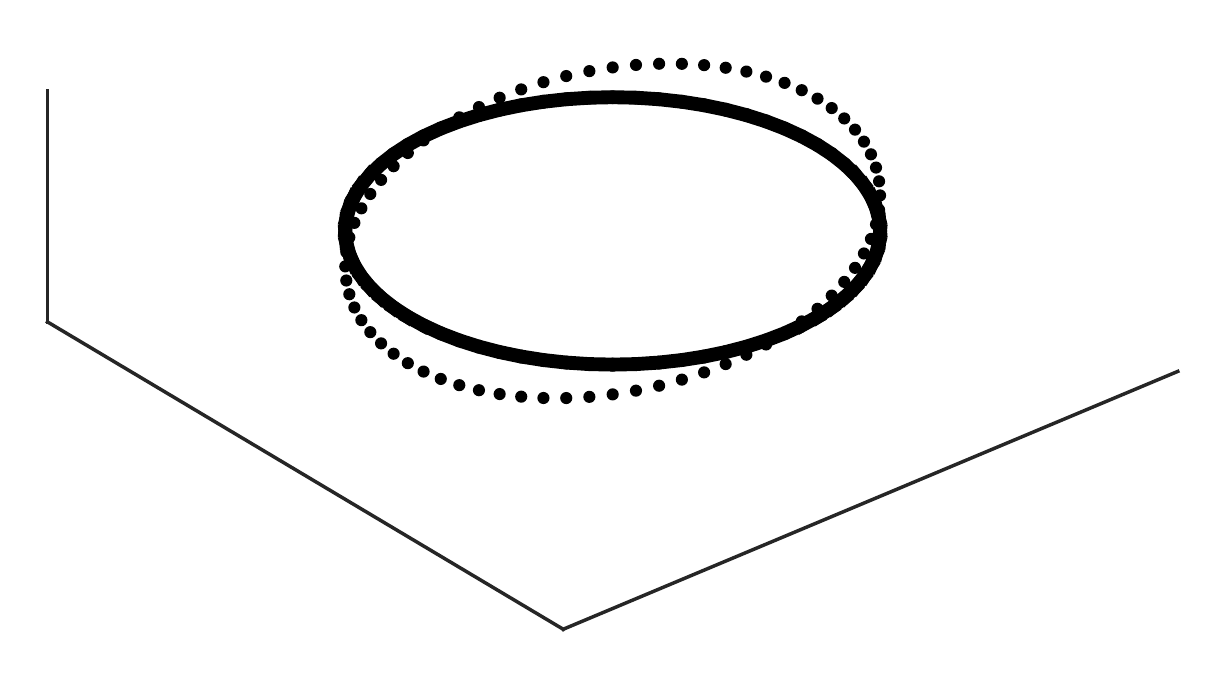}
			\caption{$(0,0,\sin(\theta))^{\intercal}$}
		\end{subfigure}%
		\caption{Estimating $\KS$ the spectral norm ball in $\mathbb{S}^2$ as the projection of the set $\C$ \eqref{eq:specnormcset}.  The above figure describes the modes of oscillation corresponding to the set of extreme points of $\KHN$ given by an ellipse (see the above accompanying discussion).  There are $8$ modes altogether -- $5$ occur in the $xy$-plane described in Figure \ref{fig:Sec3_Circ} and the remaining $3$ are shown in (b),(c), and (d).}
		\label{fig:Sec3_Wobble}
	\end{figure}
\end{example}

\subsubsection{Specialization of Theorem \ref{thm:stats_norm} to Linear Images of Free Spectrahedra}
\label{sec:stats_asymnorm_vcproof}

\begin{proposition} \label{thm:stats_norm_fsvc}
	Let $\C$ be the spectraplex.  Then the collection specified by \eqref{eq:check_vc} forms a VC class.
\end{proposition}

\begin{proof}[Proof of Proposition \ref{thm:stats_norm_fsvc}]  Define the polynomial $p(\bu,y,s,D,\be_1,\be_2,\be_3,\be_4):= \langle \hat{A}^{\intercal}\bu-y I, \be_1 \be_1^{\intercal} \rangle^2 - \langle (\hat{A}+D)^{\intercal}\bu-y I, \be_2 \be_2^{\intercal} \rangle^2 -s \langle \be_3 , D \be_4 \rangle $, where $\be_{1}, \be_{2}, \be_{4} \in \R^q$ and $\be_{3} \in \R^d$.  By Theorem 1 of \cite{SteYuk:89}, the following collection of sets forms a VC class
	\begin{equation*} \label{eq:thm:stats_norm_steq_p}
		\left\{ \left\{ (\bu,y,s) ~\Big|~ \sup_{\|\be_1\| \leq 1} \inf_{\|\be_2\|, \|\be_3\|, \|\be_4\| \leq 1} p (\bu,y,s,D,\be_1,\be_2,\be_3,\be_4) \geq 0 \right\} ~\bigg|~ D \in \lqd \right\}.
	\end{equation*}
	Similarly, by Theorem 1 of \cite{SteYuk:89}, the collection $\{\{ (\bu,y,s) ~|~ s \geq 0 \} ~|~ D \in \lqd \}$ also forms a VC class.  By \cite{Kos:08} [Lemma~9.7,~p.~159], the collection of sets formed by taking intersections of members of two VC classes is a VC class.  Hence it follows that the collection $\{ (\bu,y,s) ~|~ (\lmax((\hat{A}+D)^{\intercal}\bu)-y)^2 - (\lmax(\hat{A}^{\intercal}\bu)-y)^2\geq s \|D\|_{F}  \geq 0 \}$ is a VC class.  Subsequently, by setting $D = A - \hat{A}$ and by noting that a sub-collection of a VC class is still a VC class, the collection $\{\{ (\bu,y,s) ~|~ d_{\C,A,\hat{A}} (\bu,y) \geq s \geq 0 \} ~|~ A \in B_{\|\cdot\|_{F}}(\hat{A}) \backslash \{ \hat{A} \} \}$ forms a VC class.  A similar sequence of arguments shows that the collection $\{\{ (\bu,y,s) ~|~ d_{\C,A,\hat{A}} (\bu,y) \leq s \leq 0 \} ~|~ A \in B_{\|\cdot\|_{F}}(\hat{A}) \backslash \{ \hat{A} \} \}$ also forms a VC class.  Last, by \cite{Kos:08} [Lemma~9.7,~p.~159], a union of VC classes is a VC class, and from which our result follows.
\end{proof}


\subsection{Preservation of Facial Structure} \label{sec:faces}

Our third result describes conditions under which the constrained estimator \eqref{eq:sc_intro_constrainedlse} preserves the facial structure of the underlying set $\KS$.  We begin our discussion with some stylized numerical experiments that illustrate various aspects that inform our subsequent theoretical development.  First, we consider reconstruction of the $\ell_1$ ball in $\R^3$ from $200$ noisy support function evaluations with the choices $\C = \simp^6$ and $\C = \simp^{12}$.  Figure \ref{fig:Exp_FacialGeom} shows these reconstructions along with the LSE.  When $\C = \simp^6$, our results show a one-to-one correspondence between the faces of the reconstruction obtained using our method (second subfigure from the left) with those of the $\ell_1$ ball (leftmost subfigure); in contrast, we do not observe an analogous correspondence in the other two cases.
\begin{figure}[h]
	\centering
	\begin{subfigure}{.2\textwidth}
		\centering
		\includegraphics[width=0.75\textwidth]{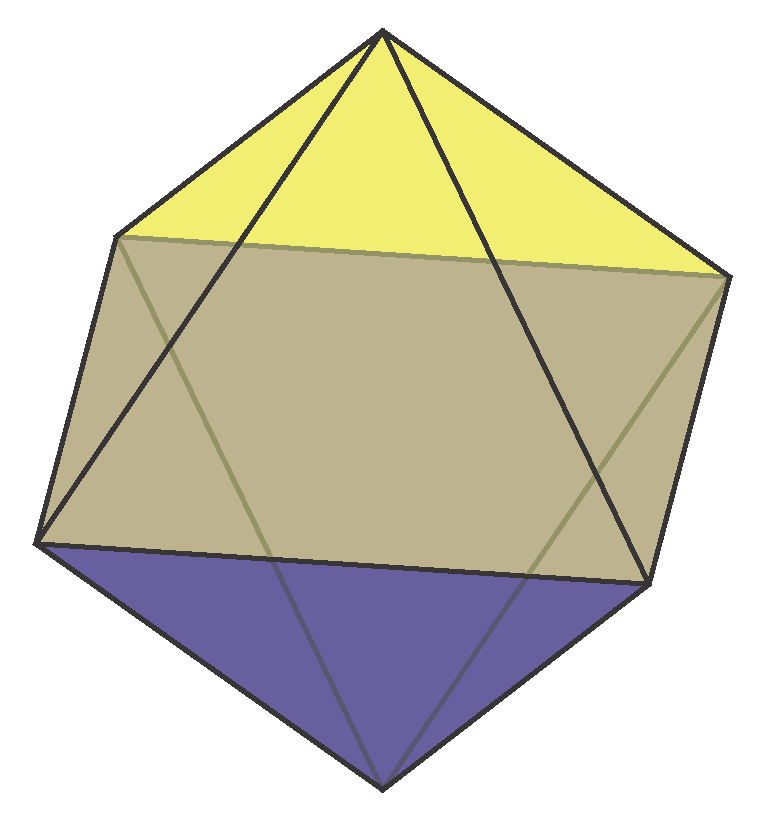}
	\end{subfigure}%
	\hspace{0.5cm}
	\begin{subfigure}{.2\textwidth}
		\centering
		\includegraphics[width=0.75\textwidth]{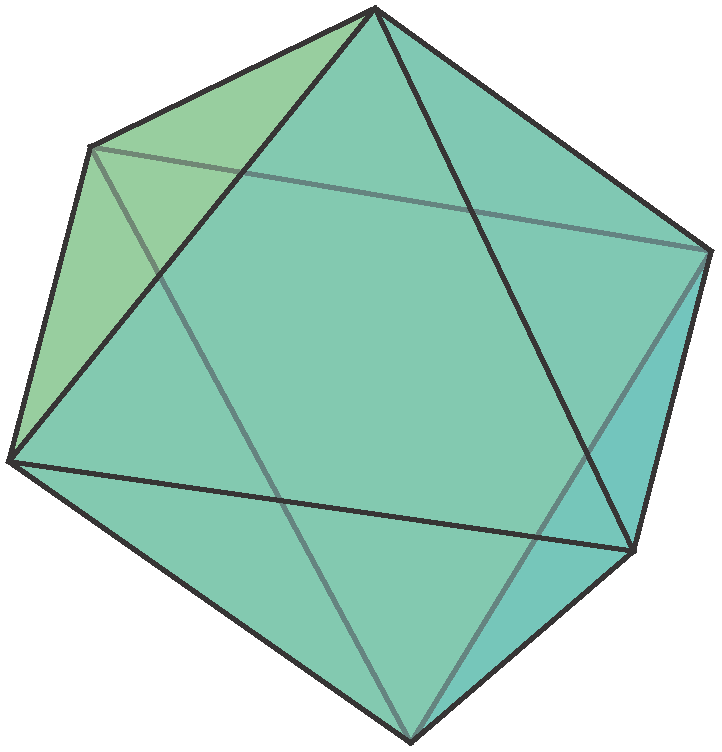}
	\end{subfigure}%
	\hspace{0.5cm}
	\begin{subfigure}{.2\textwidth}
		\centering
		\includegraphics[width=0.8\textwidth]{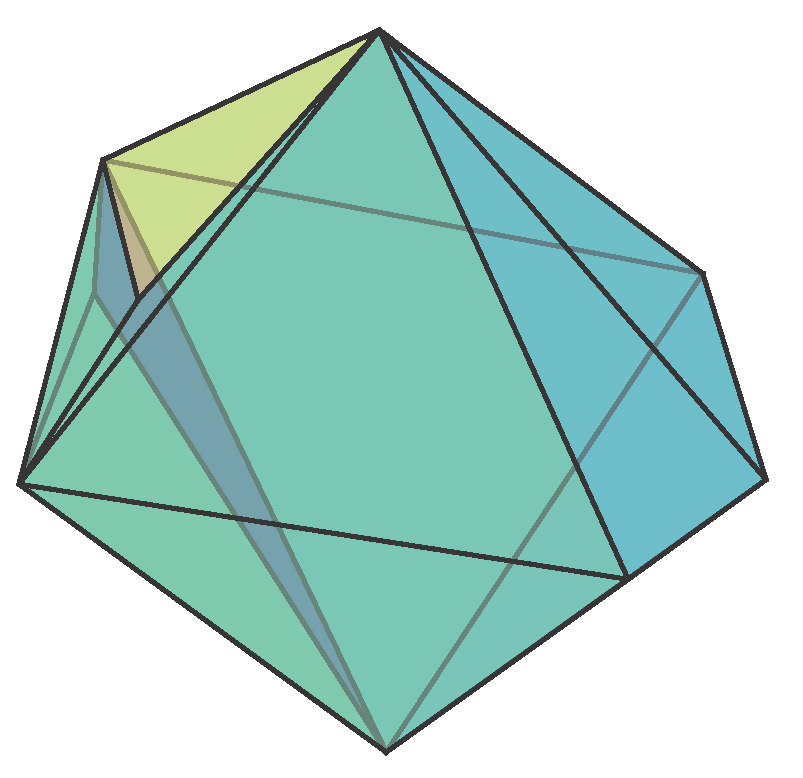}
	\end{subfigure}%
	\hspace{0.5cm}
	\begin{subfigure}{.2\textwidth}
		\centering
		\includegraphics[width=0.8\textwidth]{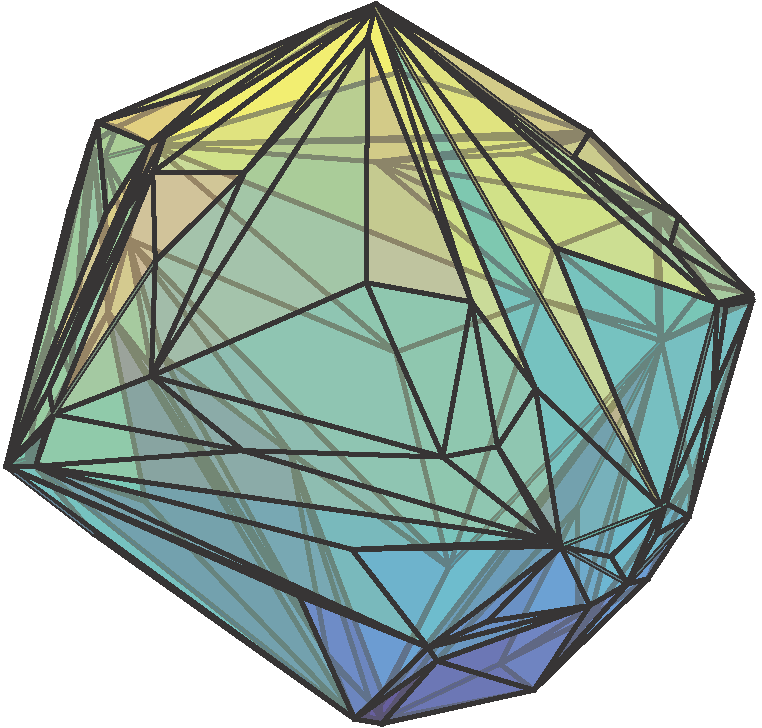}
	\end{subfigure}%
	\caption{Reconstructions of the unit $\ell_1$-ball (left) in $\R^3$ from $200$ noisy support function measurements using our method with $\C = \simp^6$ (second from left), and with $\C = \simp^{12}$ (third from left).  The LSE is the rightmost figure.}
	\label{fig:Exp_FacialGeom}
\end{figure}
Second, we consider reconstruction of the $\ell_\infty$ ball in $\R^3$ from $75$ noisy support function measurements with $\C = \simp^8$.  From Figure \ref{fig:Exp_FacialGeomHypercube} we see that both reconstructions (obtained from two different sets of $75$ measurements) break most of the faces of the $\ell_\infty$ ball.  In these examples the association between the faces of the underlying set and those of the reconstruction is somewhat transparent as the sets are polyhedral.
\begin{figure}[h]
	\centering
	\begin{subfigure}{.2\textwidth}
		\includegraphics[width=0.75\textwidth]{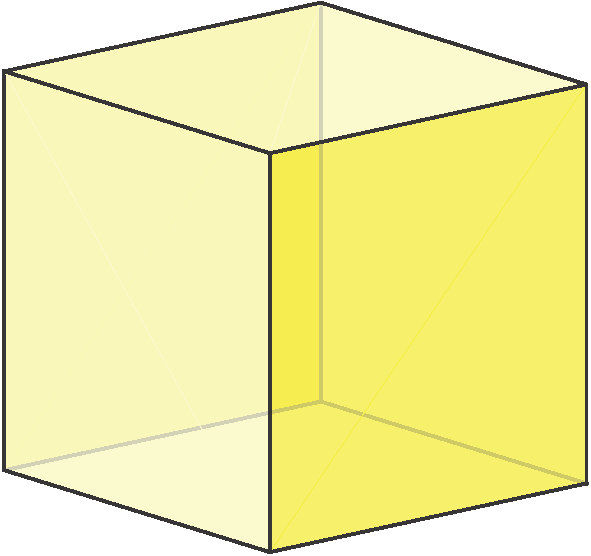}
	\end{subfigure}%
	\hspace{0.5cm}
	\begin{subfigure}{.2\textwidth}
		\includegraphics[width=0.75\textwidth]{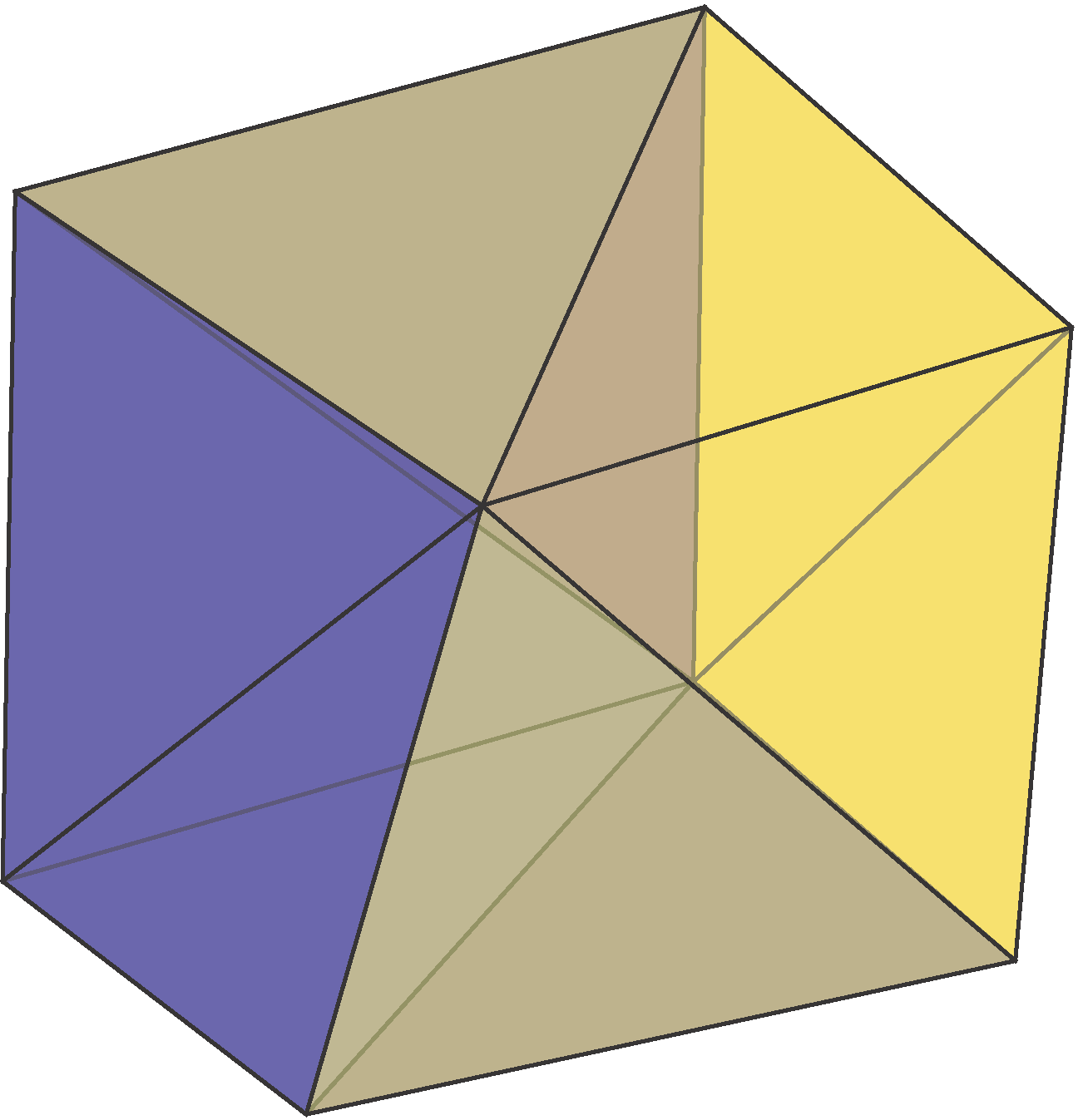}
	\end{subfigure}%
	\hspace{0.5cm}
	\begin{subfigure}{.2\textwidth}
		\includegraphics[width=0.75\textwidth]{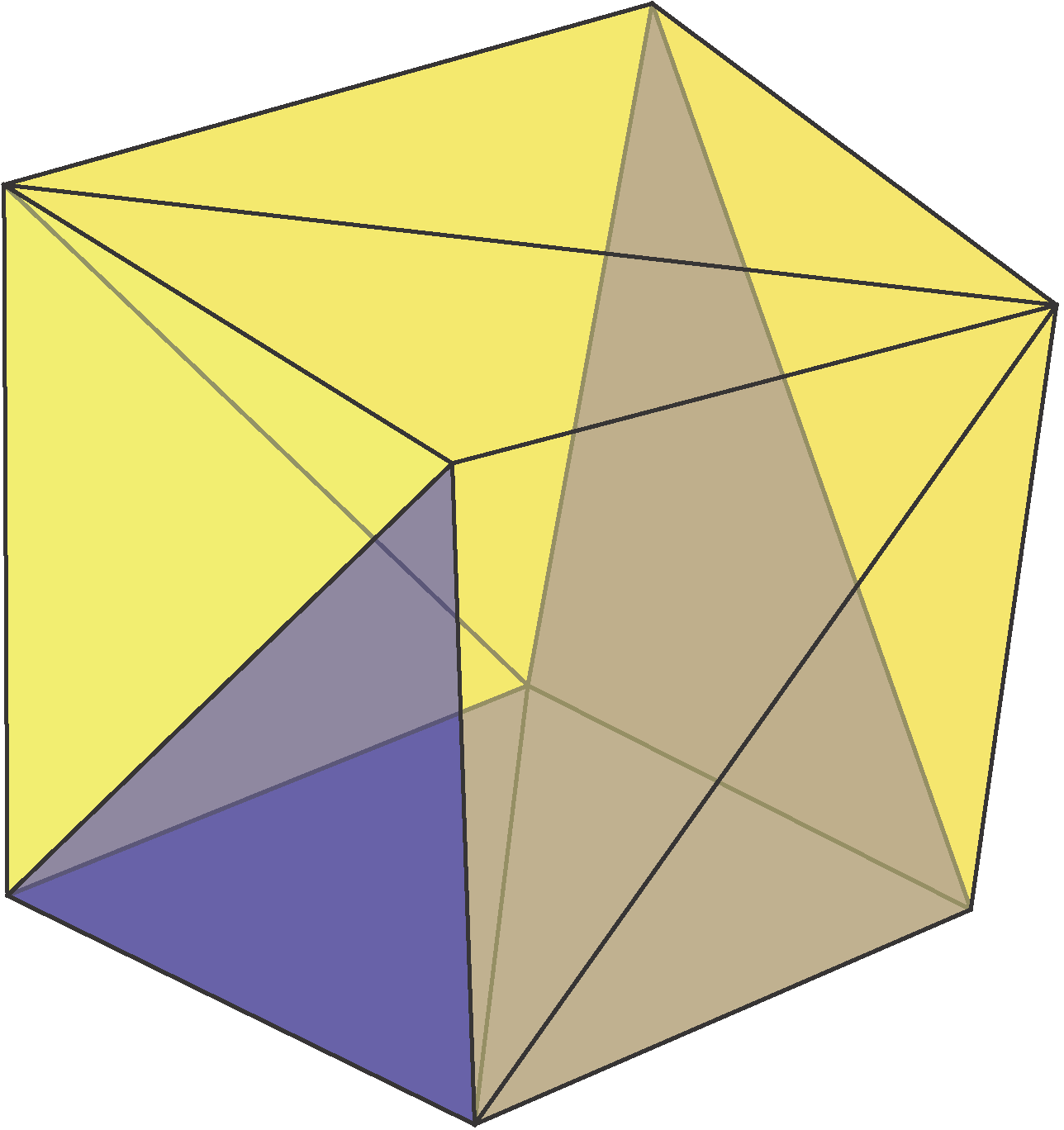}
	\end{subfigure}%
	\caption{Two reconstructions of the unit $\ell_{\infty}$ ball in $\R^3$ from $75$ noisy support function measurements using our method.  The choice of lifting set is $\C = \simp^8$.  The $\ell_{\infty}$ ball is the leftmost figure, and the reconstructions are the second and third figures from the left.}
	\label{fig:Exp_FacialGeomHypercube}
\end{figure}
The situation becomes more delicate with non-polyhedral sets.  We describe next a numerical experiment in which we estimate the \emph{Race Track} in $\R^2$ from $200$ noisy support function measurements with $\C = \fs^4$:
\begin{equation*}
	\text{Race Track} := \mathrm{conv}(\{ (x,y)^{\intercal} | \| (x,y)^{\intercal} - (-1,0)^{\intercal} \|_2 \leq 1 \text{ or } \| (x,y)^{\intercal} - (1,0)^{\intercal} \|_2 \leq 1 \}).
\end{equation*}
From this experiment, it appears that the exposed extreme points of the \emph{Race Track} are recovered, although the two one-dimensional edges are not recovered and seem to be distorted into curves.  However, the exact correspondence between faces of the \emph{Race Track} and those of the reconstruction seems less clear.

\begin{figure}[h]
	\centering
	\includegraphics[width=0.25\textwidth]{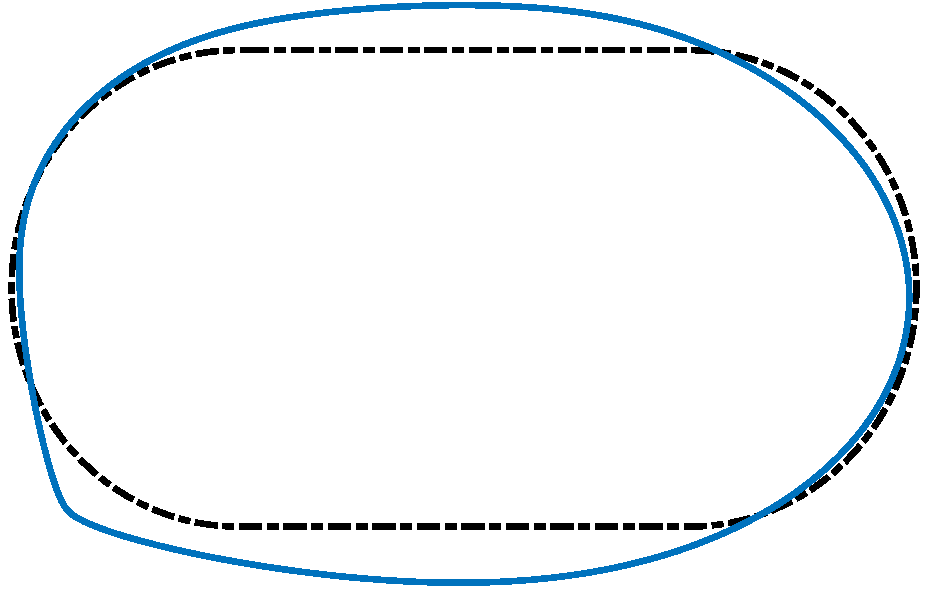}
	\hspace{1cm}
	\caption{Reconstructions of the \emph{Race Track} from $200$ noisy support function measurements using \eqref{eq:sc_intro_constrainedlse} with $\C = \fs^4$.}
	\label{fig:Exp_Facial_RaceTrack}
\end{figure}

Our first technical contribution in this subsection is to give a formal notion of `preservation of face structure'.  Motivated by the preceding example, the precise manner in which we do so is via the existence of an invertible affine transformation between the faces of the underlying set and the faces of the reconstruction.  For analytical tractability, our result focuses on exposed faces (faces that are expressible as the intersection of the underlying set with a hyperplane).
\begin{definition}  Let $\{\K_n \}_{n=1}^{\infty} \subset \R^d$ be a sequence of compact convex sets converging to some $\K \subset \R^d$.  Let $F \subset \K$ be an exposed face.  We say that $F$ is \emph{preserved} by the sequence $\{\K_n \}_{n=1}^{\infty}$ if there is a sequence $\{ F_n \}_{n=n_0}^{\infty}$, $F_n \subseteq \K_n$, satisfying
	\begin{enumerate}
		\item $F_n \rightarrow F$.
		\item $F_n$ are exposed faces of $\K_n$.
		\item There is an invertible affine transformation $b_n$ such that $F = b_n F_n $ and $F_n = b_n^{-1} F$.
	\end{enumerate}
\end{definition}

As our next contribution, we consider conditions under which exposed faces of $\KS$ are preserved.  To gain intuition for the types of assumptions that may be required, we review the results of the numerical experiments presented above.  In the setting with $\KS$ being the $\ell_1$ ball in $\R^3$, all the faces are simplicial and the reconstruction with $\C = \simp^6$ preserve all the faces.  In contrast, in the experiment with $\KS$ being the $\ell_\infty$ ball in $\R^3$ and $\C = \simp^8$, some of the faces of the $\ell_\infty$ ball are broken into smaller simplicial faces in the reconstruction.  These observations suggest that we should only expect preservation of simplicial faces, at least in the polyhedral context.  However, in attempting to reconstruct the $\ell_1$ ball with $\C = \simp^{12}$, some of the simplicial faces of the $\ell_1$ ball are broken into smaller simplicial faces in the reconstruction.  This is due to the overparametrization of the class of polytopes over which the regression \eqref{eq:sc_intro_constrainedlse} is performed (polytopes with at most $12$ vertices) relative to the complexity of the $\ell_1$ ball (a polytope with $6$ vertices).  We address this point in our theorem via the single-orbit condition discussed in Section \ref{sec:prelims_algebraic}, which also plays a role in Theorem \ref{thm:stats_norm}.  Finally, in the non-polyhedral setting with $\KS$ being the \emph{Race Track} and $\C = \fs^4$, the two one-dimensional faces deform to curves in the reconstruction.  This is due to the fact that (generic) small perturbations to the linear image of $\fs^4$ that gives $\KS$ lead to a deformation of the edges of $\KS$ to curves.  To ensure that faces remain robust to perturbations of the linear images, we require that the normal cones associated to faces that are to be preserved must be sufficiently large.

\begin{theorem}\label{thm:preserveface}
	Suppose that $\KS \subseteq \R^d$ is a compact convex set with non-empty interior.  Let $\C \subset \R^q$ be a compact convex set such that $\spn(\C) \cong \R^q$.  Suppose that there is a linear map $A^{\star} \in \lqd$ such that $\KS = A^{\star}(\C)$, and $\mc = A^{\star} \orb$.  Let $\{ \hat{A}_n \}_{n=1}^{\infty}$, $\hat{A}_n \in \mathrm{argmin}_{A} \pc (A,\pnk)$, be a sequence of minimizers of the empirical loss function, and let $\{ \KHN \}_{n=1}^{\infty}$, $\KHN = \hat{A}_n(\C)$, be the corresponding sequence of estimators of $\KS$.  Given an exposed face $F^{\star} \subset \KS$, let $G = \{ \bx | A^{\star} \bx \in F^{\star}\} \cap \C$ be its pre-image, and let $\nc_{\C} (G) := \{ v~|~\langle v, x \rangle \geq \langle v,y \rangle \forall x \in G, y \in \C \}$ be the normal cone of $G$ w.r.t. $\C$.  If
	\begin{enumerate}
		\item the linear map $A^{\star}$ is injective when restricted to $\mathrm{aff}(G)$, and
		\item $\mathrm{dim} (\spn(\nc_{\C} (G) ) ) > q - \mathrm{rank} (A^{\star})$,
	\end{enumerate}
	then $F^{\star}$ is preserved by the sequence $\{ \KHN \}_{n=1}^{\infty}$.
\end{theorem}

Before giving a proof of this result, we remark next on some of the consequences.

\begin{remark}  Suppose $\KS$ is a full-dimensional polytope with $q$ extreme points and we choose $\C = \simp^{q}$.  It is easy to see that there is a linear map $A^{\star}$ such that $\KS = A^{\star}(\simp^{q})$ and that $M_{\KS,\simp^{q}} = A^{\star} \cdot \aut(\simp^{q})$.  Let $F^{\star} \subseteq \KS$ be any face (note that all faces of a polytope are exposed), and let $G$ be its pre-image in $\simp^{q}$.  Note that $G$ is an exposed face, and hence $G$ is of the form $\{ \Pi \bx | \bx \geq 0, \langle \bx, 1 \rangle = 1, \bx_{s+1} = \ldots \bx_{q} = 0 \}$ for some $\Pi \in \aut(\simp^{q})$ and some $s \leq q$.  The map $A^{\star}$ being injective on $\aff(G)$ implies that the image of $G$ under $A^{\star}$ is isomorphic to $G$; i.e., $F^{\star}$ is simplicial.  The normal cone $\nc_{\simp^q}(G)$ is given by $ \{ \Pi \bz | \bz \leq 0, \bz_{1} = \ldots \bz_{s} = 0 \}$, and the requirement $\mathrm{dim} (\aff(\nc_{\simp^q}(G)) ) > q - \mathrm{rank} (A^{\star})$ holds precisely when $s<d$; i.e., the face $F^{\star}$ is \emph{proper}.  Thus, Theorem \ref{thm:preserveface} implies that all \emph{proper simplicial} faces of $\KS$ are preserved in the reconstruction.
\end{remark}

\begin{remark} Suppose $\KS$ is the image under $A^\star$ of the spectraplex $\C = \fs^{p}$ and that $\mc =  A^{\star} \orb$.  Let $F^{\star}$ be an exposed face and let $G$ be its pre-image in $\fs^{p}$.  Then $G$ is a face of $\fs^{p}$, and is of the form
	\begin{equation*}
		G = \left\{ U D U^{\intercal} ~ |  ~ D =
		\left(\begin{array}{cc}
			D_{11} & 0 \\ 0 & 0
		\end{array}\right), ~ D_{11} \in \fs^{r} \right\},
	\end{equation*}
	for some $U\in O(p,\R)$ and some $r \leq p$.  Note that
	\begin{equation*}
		\nc_{\fs^p}(G) = \left\{ U D U^{\intercal} ~ | ~  D =
		\left(\begin{array}{cc}
			0 & 0 \\ 0 & - D_{22}
		\end{array}\right), ~ D_{22} \in \Sym^{p-r}, ~ D_{22} \succeq 0 \right\},
	\end{equation*}
	Thus, the requirement that $\mathrm{dim} (\aff(\nc_{\fs^p}(G)) ) >{p + 1 \choose 2} -  \mathrm{rank} (A^{\star})$ holds precisely when $d > pr - {r-1 \choose 2}$.  We consider this result in the context of our earlier example involving the \emph{Race Track}.  Specifically, one can represent the \emph{Race Track} as a linear image of $\fs^4$ given by the following linear map:
	\begin{equation*}
		A^{\star}(X) =
		\left(\begin{array}{c}
			\langle A_1 , X \rangle \\ \langle A_2 , X \rangle
		\end{array}\right),
		\quad A_1 =
		\left(\begin{array}{cccc}
			-1 & 1 & & \\ 1 & -1 & & \\ & & 1 & 1 \\ & & 1 & 1
		\end{array}\right)
		\quad A_2 =
		\left(\begin{array}{cccc}
			1 & & & \\ & -1 & & \\ & & 1 & \\ & & & -1
		\end{array}\right).
	\end{equation*}
	It is clear that $\mathrm{rank} (A^{\star}) = 2$.  Let $F^{\star}$ be the face connecting $(-1,0)^{\intercal}$ and $(1,0)^{\intercal}$, and let $G_{\fs^4}$ be the pre-image of $F^{\star}$ in $\fs^4$.  One can check that
	\begin{equation*}
		G_{\fs^4} = \left\{
		\left(\begin{array}{cccc}
			x & 0 & 0 & 0 \\ 0 & 0 & 0 & 0 \\ 0 & 0 & y & 0 \\ 0 & 0 & 0 & 0
		\end{array}\right)~\Bigg|~ x,y \geq 0, x+y \leq 1
		\right\},
		\nc_{\fs^4}(G_{\fs^4}) \left\{ Z~\Bigg|~Z =
		\left(\begin{array}{cccc}
			0 & 0 & 0 & 0 \\ 0 & * & 0 & * \\ 0 & 0 & 0 & 0 \\ 0 & * & 0 & *
		\end{array}\right), Z \preceq 0
		\right\}.
	\end{equation*}
	It follows that $\mathrm{dim} (\aff(\nc_{\fs^4}(G_{\fs^4}))) = 3$.  As the dimension of $\fs^4$ is $10$, our requirement on $\mathrm{dim}(\aff(\nc_{\fs^4}(G_{\fs^4})))$ is not satisfied.
\end{remark}

\begin{proof}[Proof of Theorem \ref{thm:preserveface}]
	As we noted in the above, define $\tilde{A}_n \in \mathrm{argmin}_{A \in \hat{A}_n \orb} \| \hat{A}_n - A \|_F$, and denote $F_n = \tilde{A}_n (G)$.
	
	[$F_n \rightarrow F$]:  Since $\mc = A^{\star} \orb$, it follows from Theorem \ref{thm:stats_setconvergence} that $\tilde{A}_n \rightarrow A^{\star}$, from which we have $F_n \rightarrow F^{\star}$.
	
	[$F_n$ are faces of $\K_n$]:  Since $F^{\star}$ is an exposed face of $\KS$, there exists $\by \in \R^d$ and $c \in \R$ such that $\langle \by,\bx \rangle = c$ for all $\bx \in F^{\star}$, and $\langle \by,\bx \rangle > c$ for all $\bx \in \KS \backslash F^{\star}$.  This implies that $\langle A^{\star\intercal} \by,\tilde{\bx} \rangle = c$ for all $\tilde{\bx} \in G$, and $\langle A^{\star\intercal} \by ,\tilde{\bx} \rangle > c$ for all $\tilde{\bx} \in \C \backslash G$.  In particular, it implies that the row space of $A^{\star}$ intersects the relative interior of $\nc_{\C}(G)$ in the direction $A^{\star}\by$.
	
	By combining the earlier conclusion that $\tilde{A}_n \rightarrow A^{\star}$ a.s., and that $\mathrm{dim} (\spn(\nc_{\C}(G)) ) + \mathrm{rank} (A^{\star}) > q$, we conclude that the row spaces of the maps $\tilde{A}_n$ eventually intersect the relative interior of $\nc_{\C}(G)$ a.s.  That is to say, there is exists an integer $n_0$ and sequences $\{ \by_n \}_{n=n_0}^{\infty} \subset \R^d$, $\{c_n\}_{n=n_0}^{\infty} \subset \R $ such that $\langle \by_n, \bx \rangle = c_n$ for all $\bx \in F_n$, and $\langle \by_n, \bx \rangle > c_n$ for all $\bx \in \KHN \backslash F_n$, $n \geq n_0$,  a.s.  In other words, the sets $F_n$ are exposed faces of $\KHN$ eventually a.s.
	
	[One-to-one affine correspondence]:  To establish a one-to-one affine correspondence between $F_n$ and $F$ we need to treat the case where $0 \in \aff (G)$ and the case where $0 \notin \aff (G)$ separately.
	
	First suppose that $0 \in \aff (G)$.  Let $\hh_{F} = \aff(F)$ and $\hh_{G} = \aff(G)$.  Since $0 \in \hh_{G}$, it follows that $\hh_{F}$ and $\hh_{G}$ are subspaces.  Moreover given that $A^{\star}$ is injective restricted to $\hh_{G} = \aff(G)$, it follows that $\hh_{F}$ and $\hh_{G}$ have equal dimensions.  Hence the map $t$ defined as the restriction of $A^{\star}$ onto $L(\hh_{G},\hh_{F})$ is square and invertible.  Next let $\hh_{F_n} = \aff(F_n)$, and let $t_n$ denote the restriction of $\tilde{A}_n$ to $L(\hh_{G},\hh_{F_n})$.  Given that $\tilde{A}_n \rightarrow A^{\star}$, the maps $\{t_n\}_{n=1}^{\infty}$ are also square and invertible eventually a.s.  It follows that one can define a linear map $b_n \in L(\R^d,\R^d)$ that coincides with $t \circ t_n^{-1}$ restricted to $L(\hh_{F_n},\hh_{F})$, is permitted to be any square invertible map on $L(\hh_{F_n}^{\perp},\hh_{F}^{\perp})$, and is zero everywhere else.  Notice that $b_n$ is invertible by construction.  It straightforward to check that $F = b_n F_n$ and $F_n = b_n^{-1} F$.	
	
	Next suppose that $0 \notin \aff (G)$.  The treatment in this case is largely similar as in the previous case.  Let $\hh_{F}$ be the smallest subspace containing $\{ (\bx, 1)~|~\bx \in F \} \subseteq \R^{d+1}$, where the set $F$ is embedded in the first $d$ coordinates.  Let $\hh_{F_n}$ be similarly defined.  Let $\hh_{G} = \aff(G \cup \{0\})$ -- note that this defines a subspace.  Since $0 \notin \aff (G)$, there is a nonzero $\bz \in \R^q$ such that $\langle \bz, \bx \rangle = 1$ for all $\bx \in G$ (i.e. there exists a hyperplane containing $G$).  Define the linear map $t \in L(\hh_{G},\hh_{F})$ as
	\begin{equation*}
		t = \Pi_{\hh_{F}} \left[ \left(\begin{array}{c}
			A^{\star} \\ \bz^{\intercal}
		\end{array} \right)\bigg|_{\hh_{G}} \right]
	\end{equation*}
	where $\Pi_{\hh_{F}}$ is the projection map onto the subspace $\hh_{F}$.  Since $A^{\star}$ is injective on $G$, it follows that $\hh_{F}$ and $\hh_{G}$ have the same dimensions, and that $t$ is square and invertible.  One can define a square invertible map $t_n$ analogously.  The remainder of the proof proceeds in a similar fashion to the previous case, and we omit the details.  Here, note that a linear invertible map operating on the lifted space $\R^{d+1}$ defines an affine linear invertible map in the embedded space $\R^d$.
\end{proof} 


\section{Algorithm} \label{sec:algo}

We describe a procedure based on alternating updates for solving the optimization problem \eqref{eq:sc_intro_constrainedlse}.  In terms of the linear map $A$, the task of solving \eqref{eq:sc_intro_constrainedlse} can be reformulated as follows:
\begin{equation} \label{eq:sec_algo_eq}
\underset{A \in L(\R^q,\R^d)}{\mathrm{argmin}} ~~~ \Phi (A,P_n) := \frac{1}{n} \sum_{i=1}^{n} \left( y^{(i)} - \lmax(A^{\intercal}\bu^{(i)}) \right)^2.
\end{equation}
As described previously, the problem \eqref{eq:sec_algo_eq} is nonconvex as formulated; consequently, our approach is not guaranteed to return a globally optimal solution.  However, we demonstrate the effectiveness of these methods with random initialization in numerical experiments in Section \ref{sec:numexp}.

We describe our method as follows.  For a fixed $A$, we compute $\be^{(i)} \in \C, ~ i=1,\dots,n,$ so that $\langle \be^{(i)}, A^{\intercal}\bu^{(i)} \rangle = \lmax(A^{\intercal}\bu^{(i)})$, i.e., $\be^{(i)} = \emax(A^{\intercal}\bu^{(i)})$.  With these $\be^{(i)}$'s fixed, we update $A$ by solving the following least squares problem:
\begin{equation} \label{eq:leastsquaresTik}
\underset{A \in L(\R^q,\R^d)}{\mathrm{argmin}} ~~~ \frac{1}{n} \sum_{i=1}^{n} \left( y^{(i)} - \langle \be^{(i)}, A^{\intercal} \bu^{(i)} \rangle \right)^2.
\end{equation}
This least squares problem can sometimes be ill-conditioned, in which case we employ Tikhonov regularization (with debiasing); see Algorithm \ref{alg:am}.

When specialized to the choice $\C = \simp^{q}$, our procedure reduces to the algorithm proposed by Magnani and Boyd \cite{MagBoy:09} for max-affine regression.  More broadly, as noted in \cite{MagBoy:09}, for $\C = \simp^{q}$ our method is akin to Lloyd's algorithm for $K$-means clustering \cite{Lloyd:79}.  Specifically, Lloyd's algorithm begins with an initialization of $q$ centers, and it alternates between (i) assigning data-points to centers based proximity (keeping the centers fixed), and (ii) updating the location of cluster centers to minimize the squared-loss error.  In our context, suppose we express the linear map $A = [\ba_1|\ldots|\ba_q] \in \R^{d \times q}$ in terms of its columns.  The algorithm begins with an initialization of the $q$ columns, and it alternates between (i) assigning measurement pairs $(\bu^{(i)},y^{(i)})$, $1\leq i \leq n$, to the respective columns $\{\ba_j\}_{1\leq j \leq q}$ such that the inner product $\langle \bu^{(i)}, \ba_j \rangle$ is maximized (keeping the columns fixed), and (ii) updating the columns $\{\ba_j\}_{1\leq j \leq q}$ to minimize the squared-loss error.

\begin{algorithm}
	\caption{Convex Regression via Alternating Minimization}
	\label{alg:am}
	\textbf{Input}: A collection $\{(\bu^{(i)},y^{(i)})\}_{i=1}^n \subset \R^d \times \R$ of support function evaluations; a compact convex set $\C \subset \R^q$; an initialization $A \in L(\R^{q},\R^d)$; a choice of regularization parameter $\gamma > 0$ \\
	\textbf{Algorithm}: Repeat until convergence \\
	\textbf{1.}[Update optimizers of support function] $\be^{(i)} \leftarrow \emax(A^{\intercal}\bu^{(i)})$ \\
	\textbf{2.}[Update $A$ by solving \eqref{eq:leastsquaresTik} via Tikhonov regularization (with debiasing)] $A \leftarrow (V \otimes V + \gamma I)^{-1} (V Y + \gamma A )$ where $V \leftarrow \left( \bu^{(1)} \otimes \be^{(1)} | \ldots | \bu^{(n)} \otimes \be^{(n)} \right), \quad Y \leftarrow \left( y^{(1)}, \ldots,  y^{(n)} \right)^{\intercal}$ \\
	\textbf{Output}: Final iterate $A$
\end{algorithm} 


\section{Numerical Experiments} \label{sec:numexp}
In this section we describe the results of numerical experiments on fitting convex sets to support function evaluations in which we contrast our framework based on solving \eqref{eq:sc_intro_constrainedlse} to previous methods based on solving \eqref{eq:sc_intro_lse}.  The first few experiments are on synthetically generated data, while the final experiment is on a reconstruction problem with real data obtained from the Computed Tomography (CT) scan of a human lung.  For each experiment, we apply Algorithm \ref{alg:am} described in Section \ref{sec:algo} with multiple random initializations, and we select the solution that minimizes the least squared error.  
The (polyhedral) LSE reconstructions in our experiments are based on the algorithm proposed in \cite[Section $4$]{GarKid:09}.

\subsection{Reconstructing the $\ell_1$-ball and the $\ell_2$-ball}
We consider reconstructing the $\ell_1$-ball $\{ \bg ~|~ \|\bg\|_1 \leq 1 \} \subset \mathbb{R}^3$ and the $\ell_2$-ball $\{ \bg ~|~ \|\bg\|_2 \leq 1 \} \subset \mathbb{R}^3$ from noiseless and noisy support function evaluations based on the model \eqref{eq:probdist}.  In particular, we evaluate the performance of our framework relative to the reconstructions provided by the LSE for $n=20,50,200$ measurements.  For both the $\ell_1$-ball and the $\ell_2$-ball in the respective noisy cases, the measurements are corrupted with additive Gaussian noise of variance $\sigma^2 = 0.1$.  The reconstructions based on our framework \eqref{eq:sc_intro_constrainedlse} of the $\ell_1$-ball employ the choice $\C = \simp^6$, while those of the $\ell_2$-ball use $\C = \fs^3$.  Figure~\ref{fig:Exp_F1a_00} and Figure~\ref{fig:Exp_F1b_00} give the results corresponding to the $\ell_1$-ball and the $\ell_2$-ball, respectively.

Considering first a setting with noiseless measurements, we observe that our approach gives an exact reconstruction for both the $\ell_1$-ball and the $\ell_2$-ball.  For the $\ell_1$-ball this occurs when we have $n=200$ measurements, while the LSE provides a reconstruction with substantially more complicated facial structure that doesn't reflect that of the $\ell_1$-ball.  Indeed, the LSE only approaches the $\ell_1$-ball with respect to the Hausdorff metric, but despite being the best solution in terms of minimizing the least-squares criterion, the reconstruction offered by this method provides little information about the facial geometry of the $\ell_1$-ball.  Further, even with $n=20,50$ measurements, our reconstructions bear far closer resemblance to the $\ell_1$-ball, while the LSE in these cases looks very different from the $\ell_1$-ball.  For the $\ell_2$-ball, our approach provides an exact reconstruction with just $n=20$ measurements, while the LSE only begins to resemble the $\ell_2$-ball with $n=200$ measurements (and even then, the reconstruction is a polyhedral approximation).

Turning our attention next to the noisy case, the contrast between the results obtained using our framework and those of the LSE approach is even more stark.  For both the $\ell_1$-ball and the $\ell_2$-ball, the LSE reconstructions bear little resemblance to the underlying convex set, unlike the estimates produced using our method.  Notice that the reconstructions of the $\ell_2$-ball using our algorithm are not even ellipsoidal when the number of measurements is small (e.g., when $n=20$), as linear images of the spectraplex $\fs^3$ may be non-ellipsoidal in general and need not even consist of smooth boundaries.  Nonetheless, as the number of measurements available to our algorithm increases, the estimates improve in quality and offer improved reconstructions -- with smooth boundaries -- of the $\ell_2$-ball.

In summary, these synthetic examples demonstrate that our framework is much more effective than the LSE in terms of robustness to noise, accuracy of reconstruction given a small number of measurements, and in settings in which the underlying set is non-polyhedral.

\begin{figure}
	\centering
	\begin{subfigure}{.30\textwidth}
		\centering
		\begin{subfigure}{.5\textwidth}
			\centering
			\includegraphics[width=0.8\textwidth]{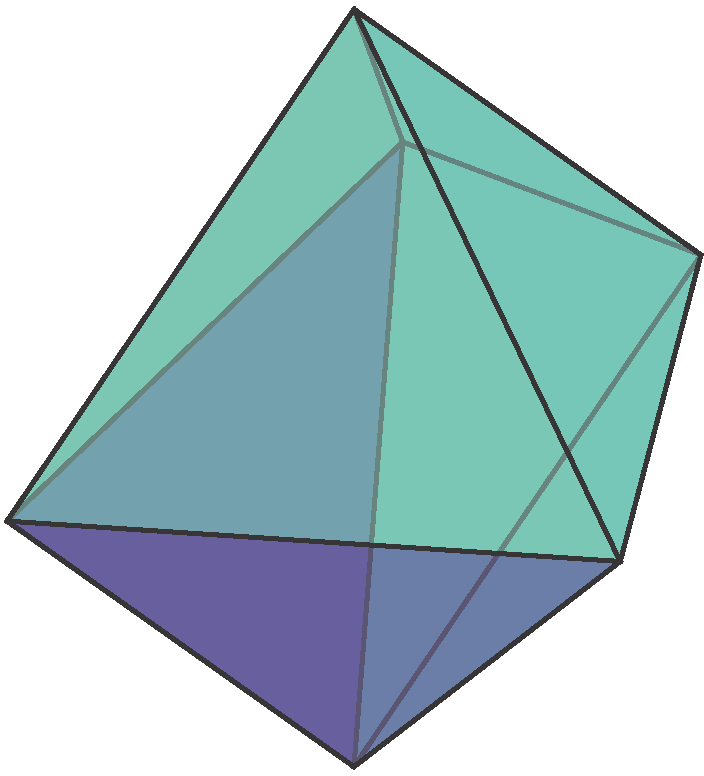}
		\end{subfigure}%
		\begin{subfigure}{.5\textwidth}
			\centering
			\includegraphics[width=0.8\textwidth]{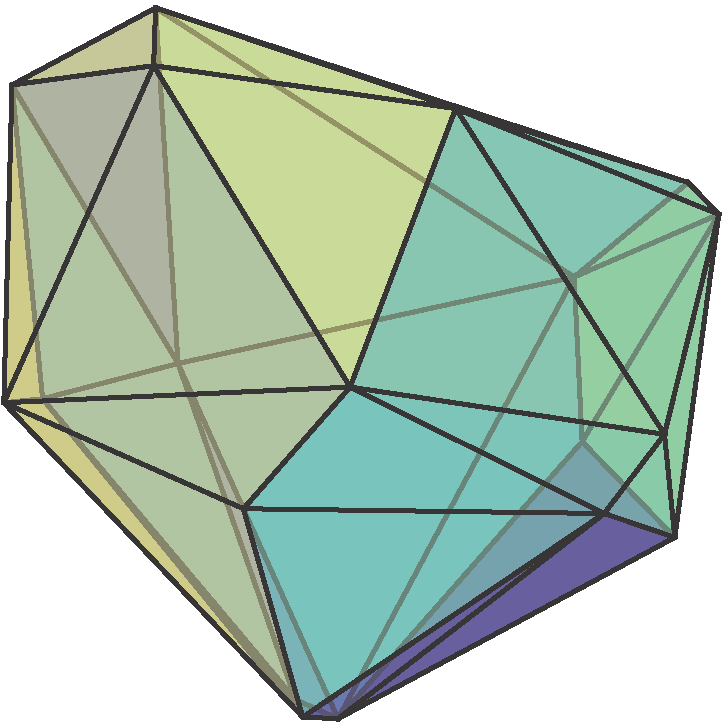}
		\end{subfigure}%
		\caption{$20$ noiseless measurements}
	\end{subfigure}%
	\hspace{0.5cm}
	\begin{subfigure}{.30\textwidth}
		\centering
		\begin{subfigure}{.5\textwidth}
			\centering
			\includegraphics[width=0.8\textwidth]{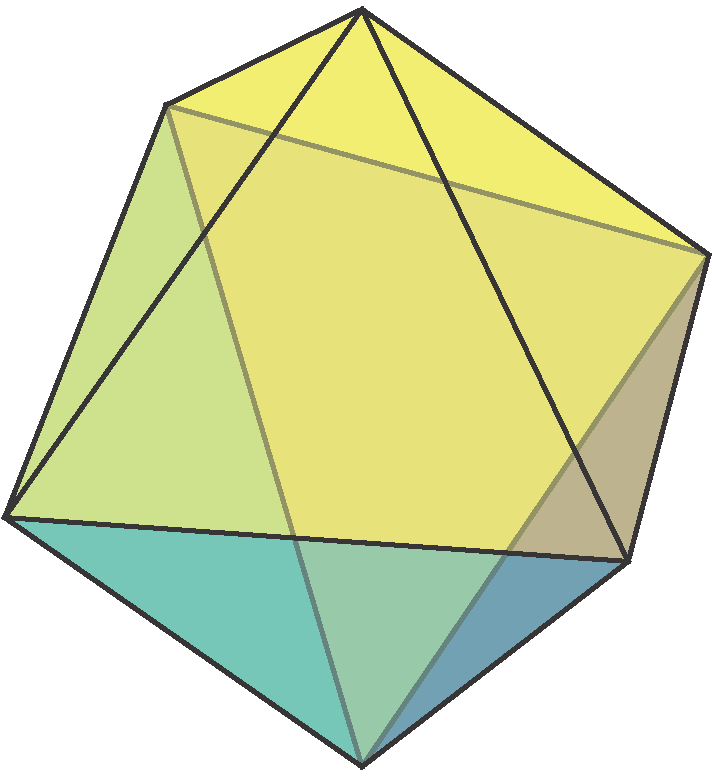}
		\end{subfigure}%
		\begin{subfigure}{.5\textwidth}
			\centering
			\includegraphics[width=0.8\textwidth]{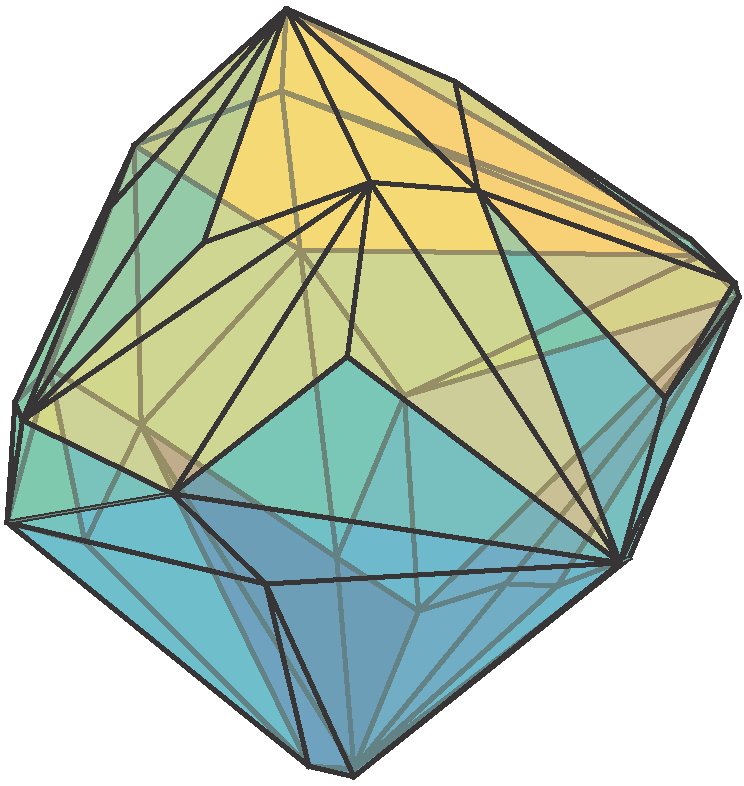}
		\end{subfigure}%
		\caption{$50$ noiseless measurements}
	\end{subfigure}%
	\hspace{0.5cm}
	\begin{subfigure}{.30\textwidth}
		\centering
		\begin{subfigure}{.5\textwidth}
			\centering
			\includegraphics[width=0.8\textwidth]{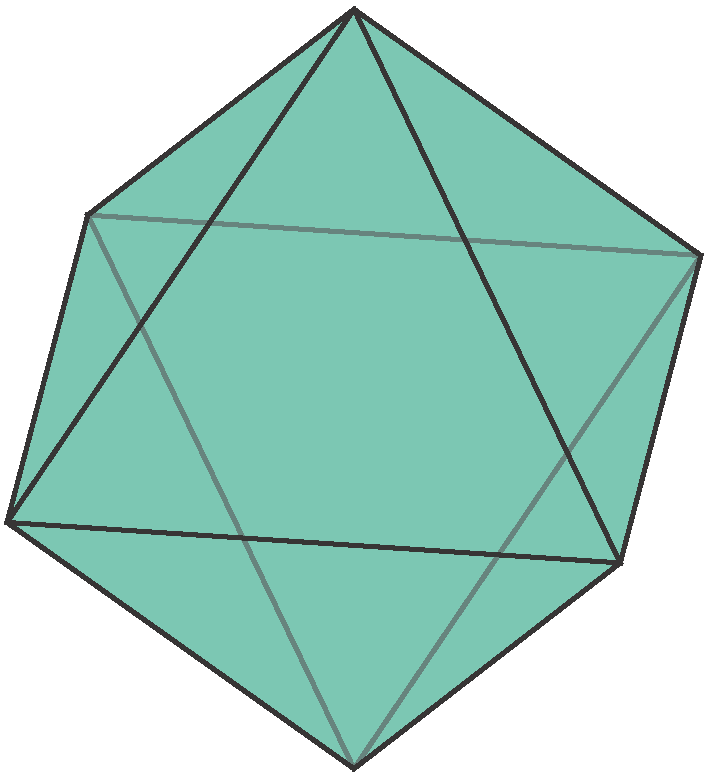}
		\end{subfigure}%
		\begin{subfigure}{.5\textwidth}
			\centering
			\includegraphics[width=0.8\textwidth]{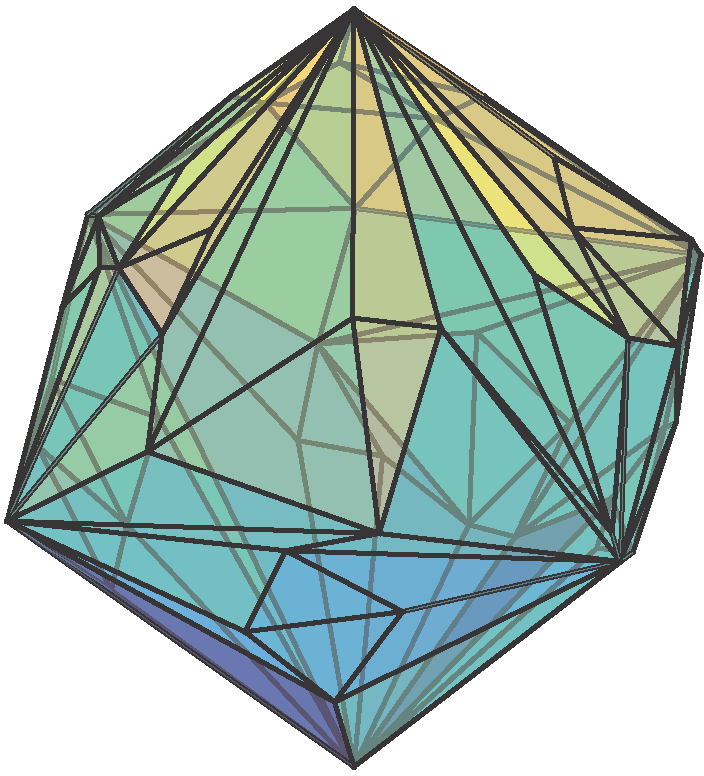}
		\end{subfigure}%
		\caption{$200$ noiseless measurements}
	\end{subfigure}
	
	\vspace{1cm}
	
	\begin{subfigure}{.30\textwidth}
		\centering
		\begin{subfigure}{.5\textwidth}
			\centering
			\includegraphics[width=0.9\textwidth]{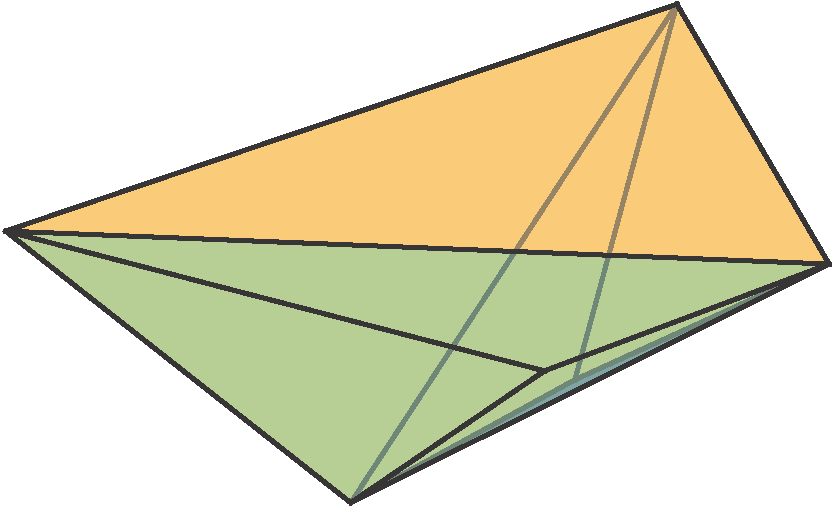}
		\end{subfigure}%
		\begin{subfigure}{.5\textwidth}
			\centering
			\includegraphics[width=0.8\textwidth]{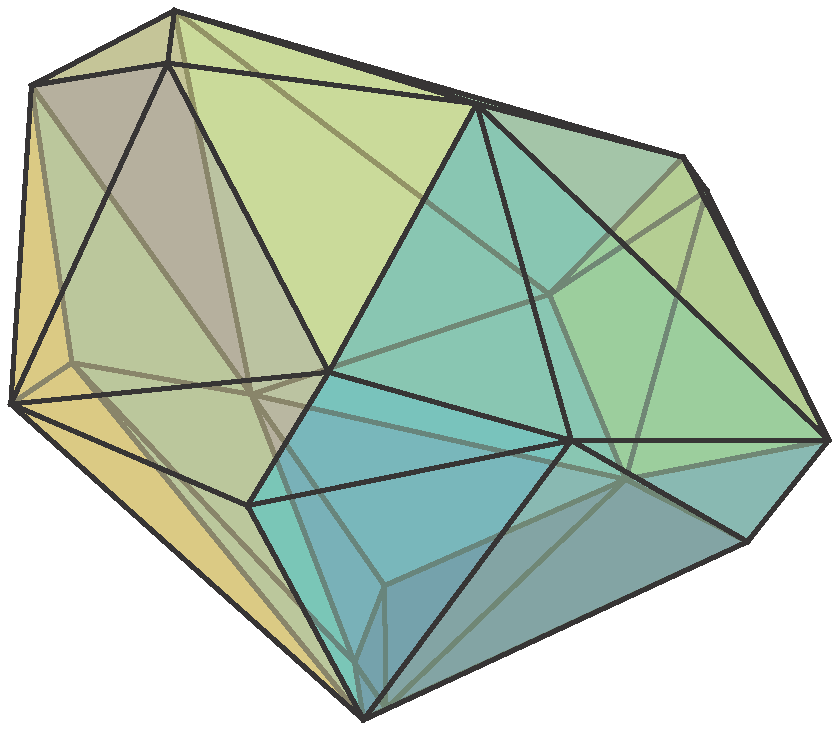}
		\end{subfigure}%
		\caption{$20$ noisy measurements}
	\end{subfigure}%
	\hspace{0.5cm}
	\begin{subfigure}{.30\textwidth}
		\centering
		\begin{subfigure}{.5\textwidth}
			\centering
			\includegraphics[width=0.75\textwidth]{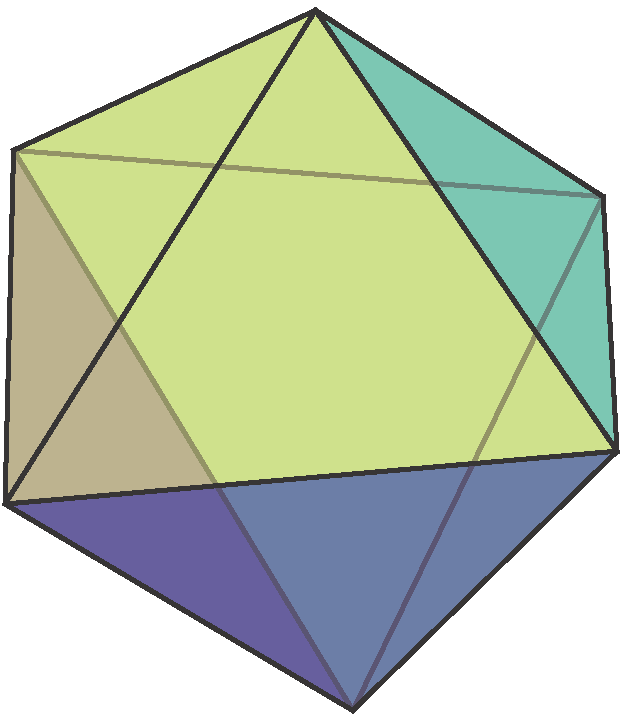}
		\end{subfigure}%
		\begin{subfigure}{.5\textwidth}
			\centering
			\includegraphics[width=0.8\textwidth]{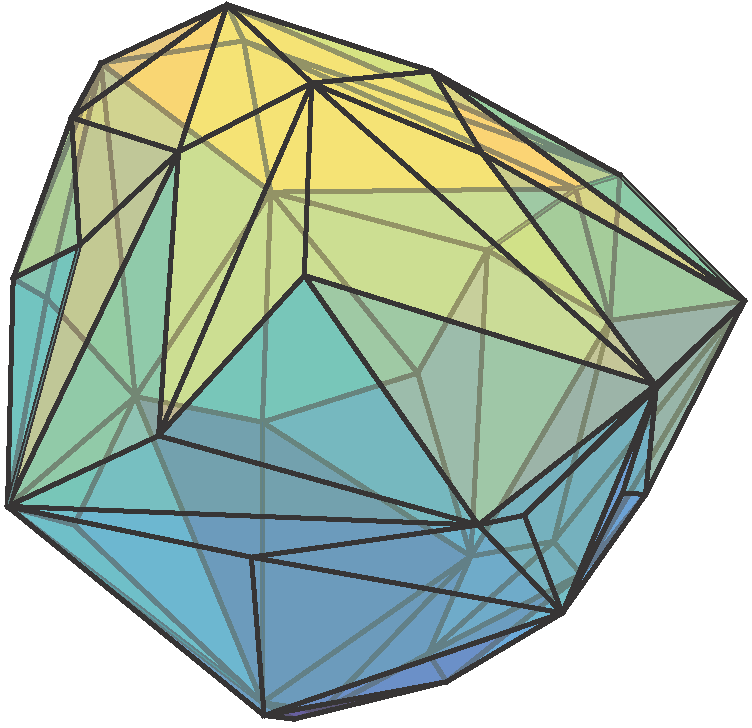}
		\end{subfigure}%
		\caption{$50$ noisy measurements}
	\end{subfigure}%
	\hspace{0.5cm}
	\begin{subfigure}{.30\textwidth}
		\centering
		\begin{subfigure}{.5\textwidth}
			\centering
			\includegraphics[width=0.8\textwidth]{Exp_F1a_s01_AMpoly_n200}
		\end{subfigure}%
		\begin{subfigure}{.5\textwidth}
			\centering
			\includegraphics[width=0.8\textwidth]{Exp_F1a_s01_LS_n200}
		\end{subfigure}%
		\caption{$200$ noisy measurements}
	\end{subfigure}%
	\caption{Reconstruction of the unit $\ell_1$-ball in $\R^3$ from noiseless (first row) and noisy (second row) support function measurements.  The reconstructions obtained using our method (with $\C = \simp^6$ in \eqref{eq:sc_intro_constrainedlse}) are on the left of every subfigure, while the LSE reconstructions are on the right of every subfigure.}
	\label{fig:Exp_F1a_00}
\end{figure}


\begin{figure}
	\centering
	\begin{subfigure}{.30\textwidth}
		\centering
		\begin{subfigure}{.5\textwidth}
			\centering
			\includegraphics[width=0.8\textwidth]{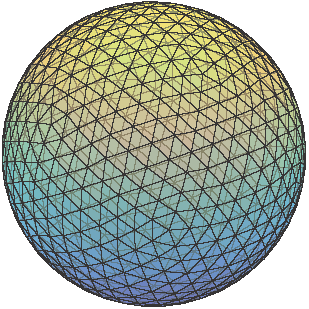}
		\end{subfigure}%
		\begin{subfigure}{.5\textwidth}
			\centering
			\includegraphics[width=0.7\textwidth]{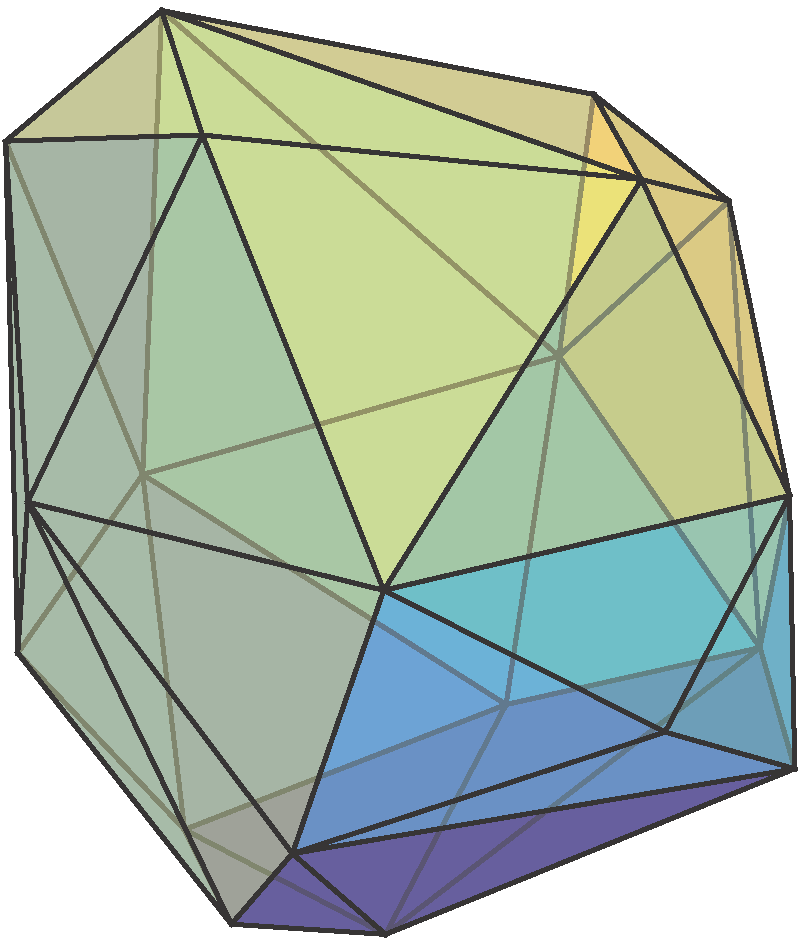}
		\end{subfigure}%
		\caption{$20$ noiseless measurements}
	\end{subfigure}%
	\hspace{0.5cm}
	\begin{subfigure}{.30\textwidth}
		\centering
		\begin{subfigure}{.5\textwidth}
			\centering
			\includegraphics[width=0.8\textwidth]{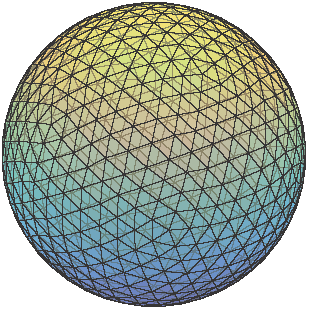}
		\end{subfigure}%
		\begin{subfigure}{.5\textwidth}
			\centering
			\includegraphics[width=0.8\textwidth]{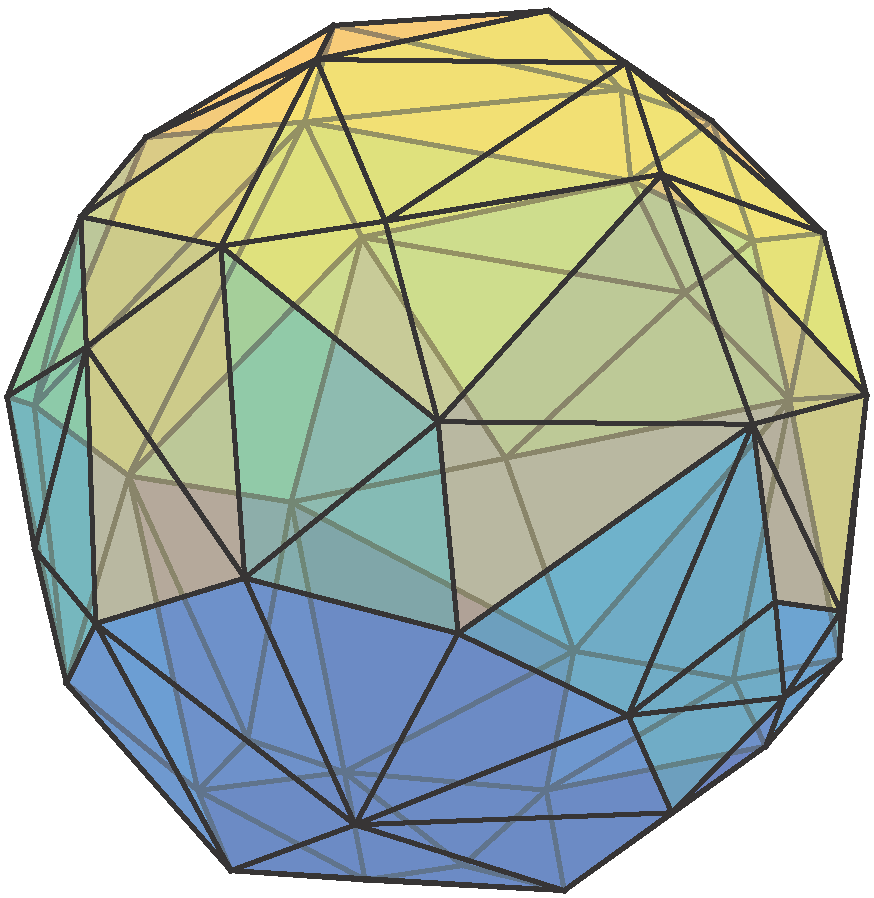}
		\end{subfigure}%
		\caption{$50$ noiseless measurements}
	\end{subfigure}%
	\hspace{0.5cm}
	\begin{subfigure}{.30\textwidth}
		\centering
		\begin{subfigure}{.5\textwidth}
			\centering
			\includegraphics[width=0.8\textwidth]{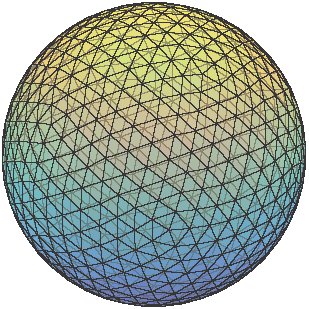}
		\end{subfigure}%
		\begin{subfigure}{.5\textwidth}
			\centering
			\includegraphics[width=0.8\textwidth]{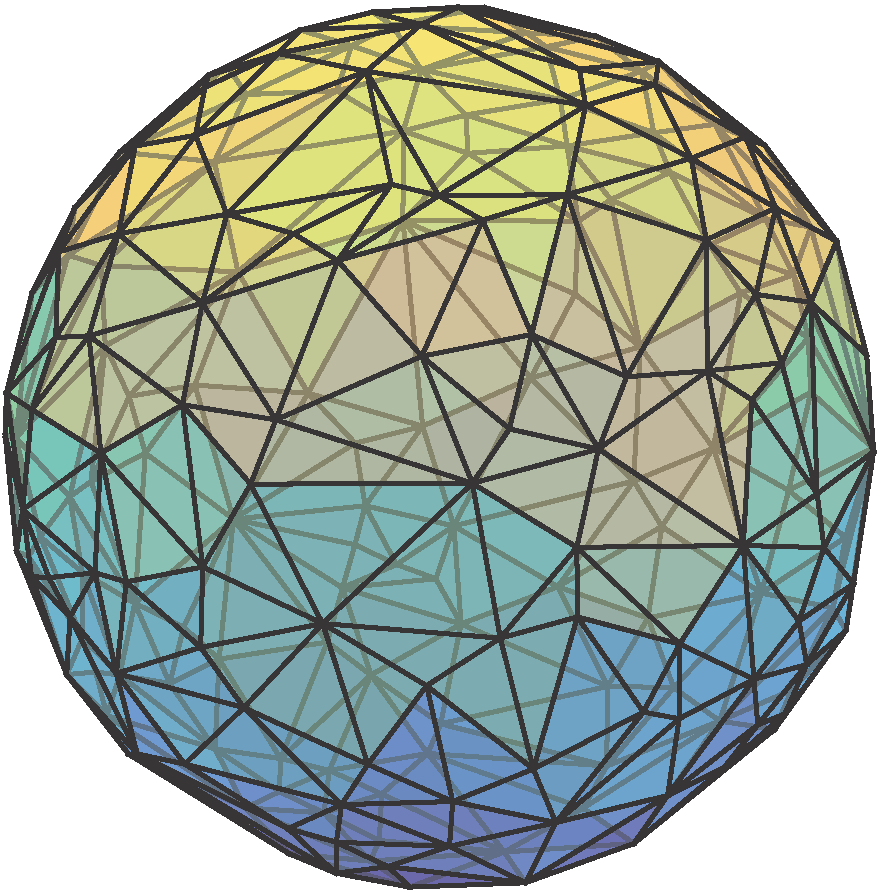}
		\end{subfigure}%
		\caption{$200$ noiseless measurements}
	\end{subfigure}%
	
	\vspace{1cm}
	
	\begin{subfigure}{.30\textwidth}
		\centering
		\begin{subfigure}{.5\textwidth}
			\centering
			\includegraphics[width=0.8\textwidth]{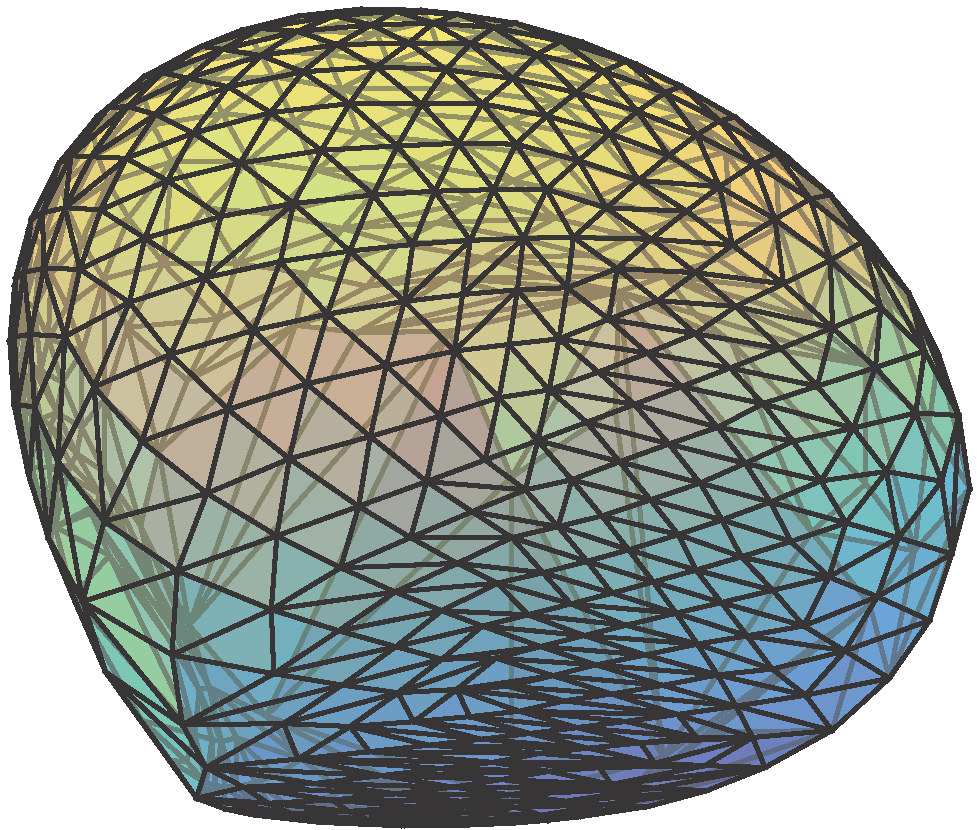}
		\end{subfigure}%
		\begin{subfigure}{.5\textwidth}
			\centering
			\includegraphics[width=0.8\textwidth]{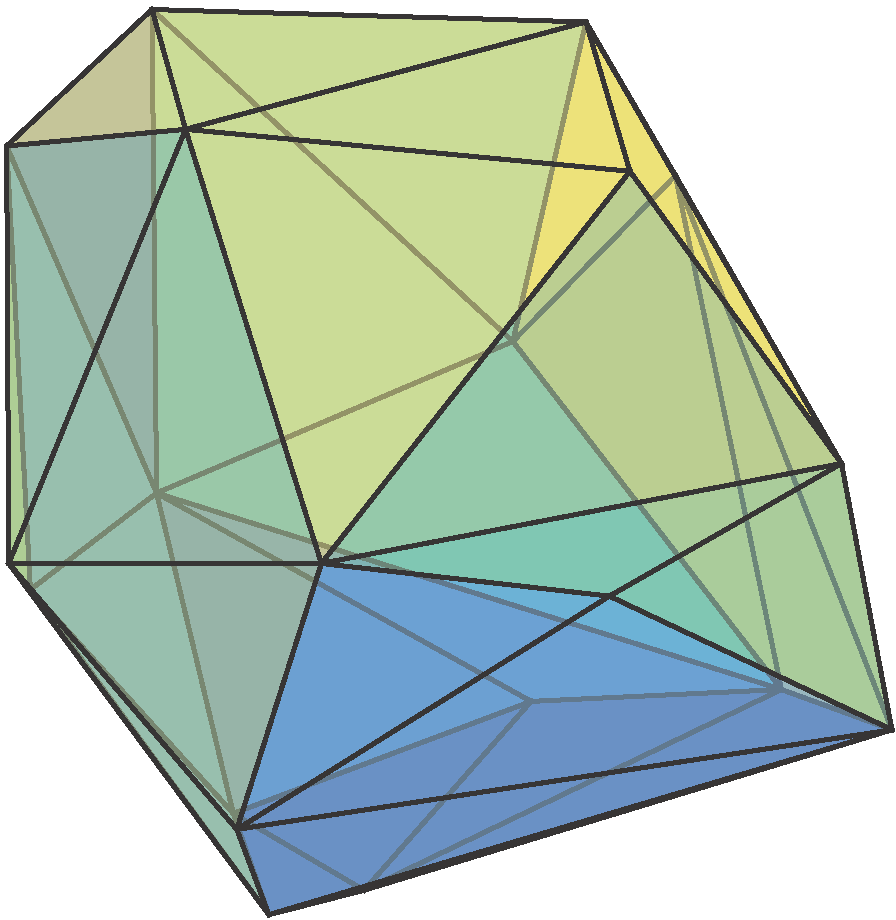}
		\end{subfigure}%
		\caption{$20$ noisy measurements}
	\end{subfigure}%
	\hspace{0.5cm}
	\begin{subfigure}{.30\textwidth}
		\centering
		\begin{subfigure}{.5\textwidth}
			\centering
			\includegraphics[width=0.8\textwidth]{Exp_F1b_s01_AMsd_n50}
		\end{subfigure}%
		\begin{subfigure}{.5\textwidth}
			\centering
			\includegraphics[width=0.8\textwidth]{Exp_F1b_s01_LS_n50}
		\end{subfigure}%
		\caption{$50$ noisy measurements}
	\end{subfigure}%
	\hspace{0.5cm}
	\begin{subfigure}{.30\textwidth}
		\centering
		\begin{subfigure}{.5\textwidth}
			\centering
			\includegraphics[width=0.8\textwidth]{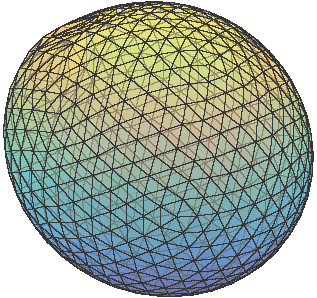}
		\end{subfigure}%
		\begin{subfigure}{.5\textwidth}
			\centering
			\includegraphics[width=0.8\textwidth]{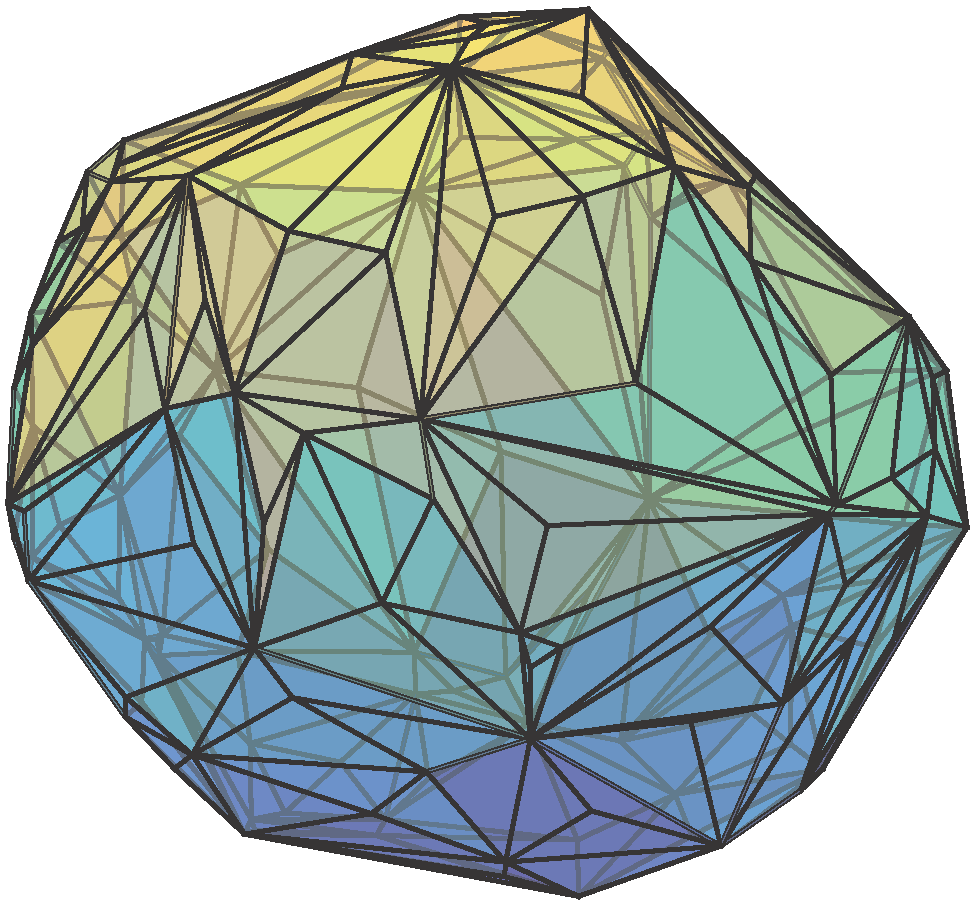}
		\end{subfigure}%
		\caption{$200$ noisy measurements}
	\end{subfigure}%
	\caption{Reconstruction of the unit $\ell_2$-ball in $\R^3$ from noiseless (first row) and noisy (second row) support function measurements.  The reconstructions obtained using our method (with $\C = \fs^3$ in \eqref{eq:sc_intro_constrainedlse}) are on the left of every subfigure, while the LSE reconstructions are on the right of every subfigure.}
	\label{fig:Exp_F1b_00}
\end{figure}


\subsection{Reconstruction via Linear Images of the Free Spectrahedron}

In the next series of synthetic experiments, we consider reconstructions of convex sets with both smooth and non-smooth features on the boundary via linear images of the spectraplex.  In these illustrations, we consider sets in $\R^2$ and in $\R^3$ for which noiseless support function evaluations are obtained and supplied as input to the problem \eqref{eq:sc_intro_constrainedlse}, with $\C$ equal to a spectraplex $\fs^p$ for different choices of $p$.  For the examples in $\R^2$, the support function evaluations are obtained at $1000$ equally spaced points on the unit circle $\sph^1$.  For the examples in $\R^3$, the support function evaluations are obtained at $2562$ regularly spaced points on the unit sphere $\sph^2$ based on an icosphere discretization.

We consider reconstruction of the $\ell_1$-ball in $\R^2$ and in $\R^3$.  Figure \ref{fig:Exp_F5_2d} shows the output from our algorithm when $d=2$ for $p \in \{ 2,3,4 \}$, and the reconstruction is exact for $p=4$.  Figure \ref{fig:Exp_F5_3d} shows the output from our algorithm when $d=3$ for $p \in \{ 3,4,5,6 \}$.  Interestingly, when $d=3$ the computed solution for $p=5$ does not contain any isolated extreme point (i.e., vertices) even though such features are expressible as projections of the spectraplex $\fs^5$.

\begin{figure}
	\centering
	\begin{subfigure}{.3\textwidth}
		\centering
		\includegraphics[width=0.4\textwidth]{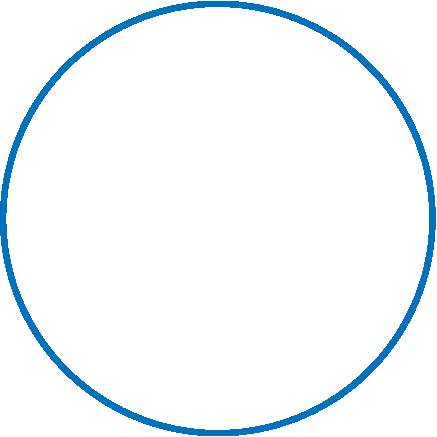}
		\label{fig:Exp_F5_q2}
	\end{subfigure}%
	\begin{subfigure}{.3\textwidth}
		\centering
		\includegraphics[width=0.4\textwidth]{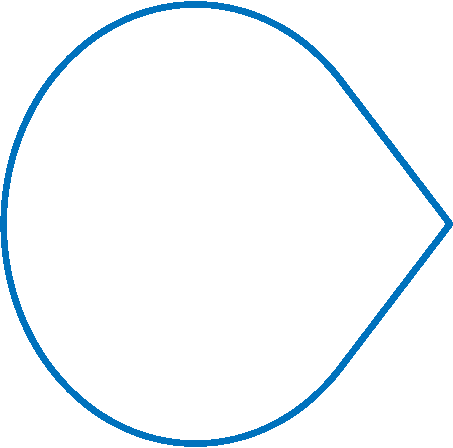}
		\label{fig:Exp_F5_q3}
	\end{subfigure}%
	\begin{subfigure}{.4\textwidth}
		\centering
		\includegraphics[width=0.3\textwidth]{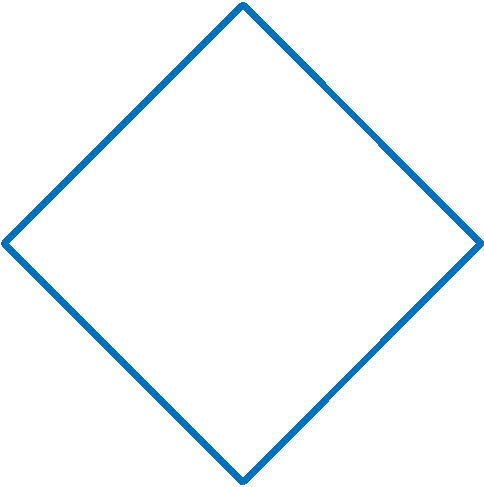}
		\label{fig:Exp_F5_q4}
	\end{subfigure}%
	\caption{Approximating the $\ell_1$-ball in $\mathbb{R}^2$ as a projection of the free-spectrahedra $\fs^2$ (left), $\fs^3$ (center), and $\fs^4$ (right).}
	\label{fig:Exp_F5_2d}
\end{figure}

\begin{figure}
	\centering
	\begin{subfigure}{.25\textwidth}
		\centering
		\includegraphics[width=0.45\textwidth]{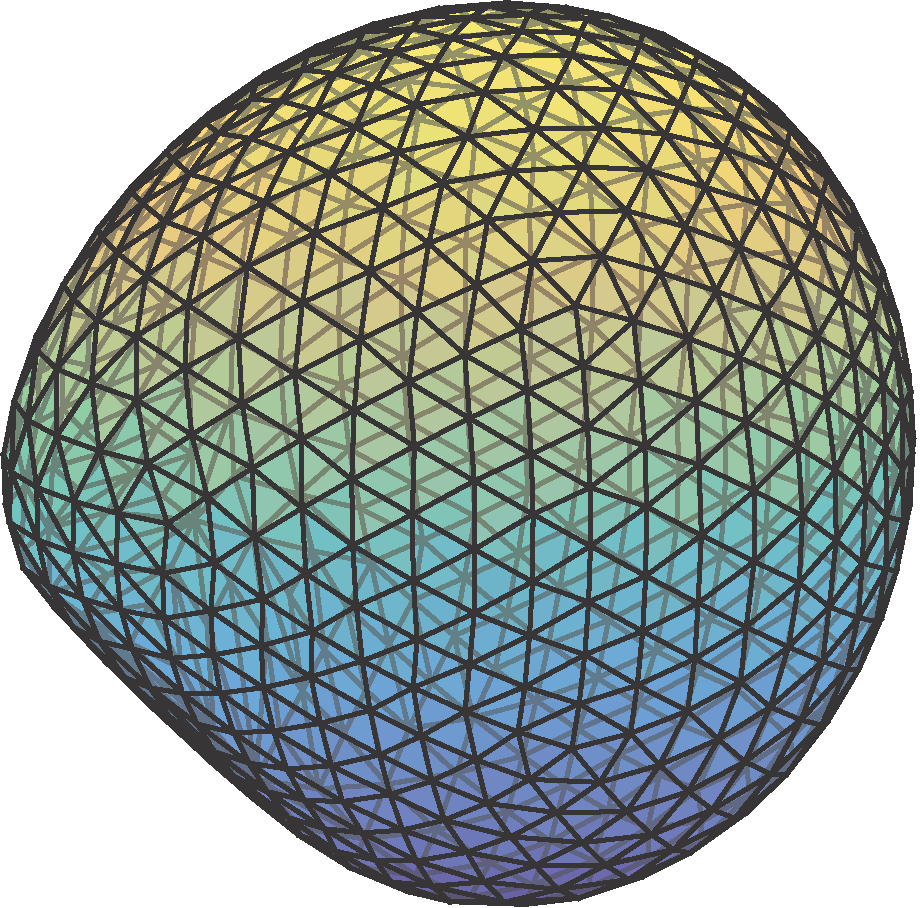}
		\label{fig:Exp_F5_3d_q3}
	\end{subfigure}%
	\begin{subfigure}{.25\textwidth}
		\centering
		\includegraphics[width=0.45\textwidth]{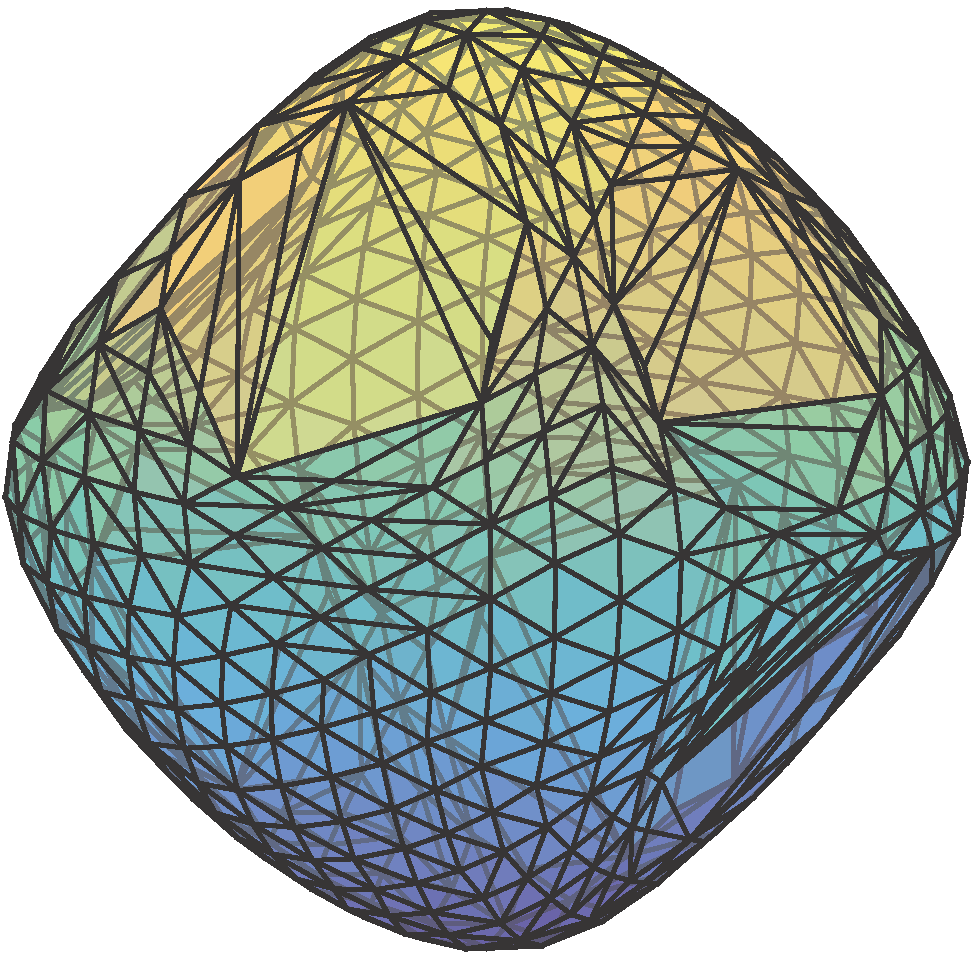}
		\label{fig:Exp_F5_3d_q4}
	\end{subfigure}%
	\begin{subfigure}{.25\textwidth}
		\centering
		\includegraphics[width=0.45\textwidth]{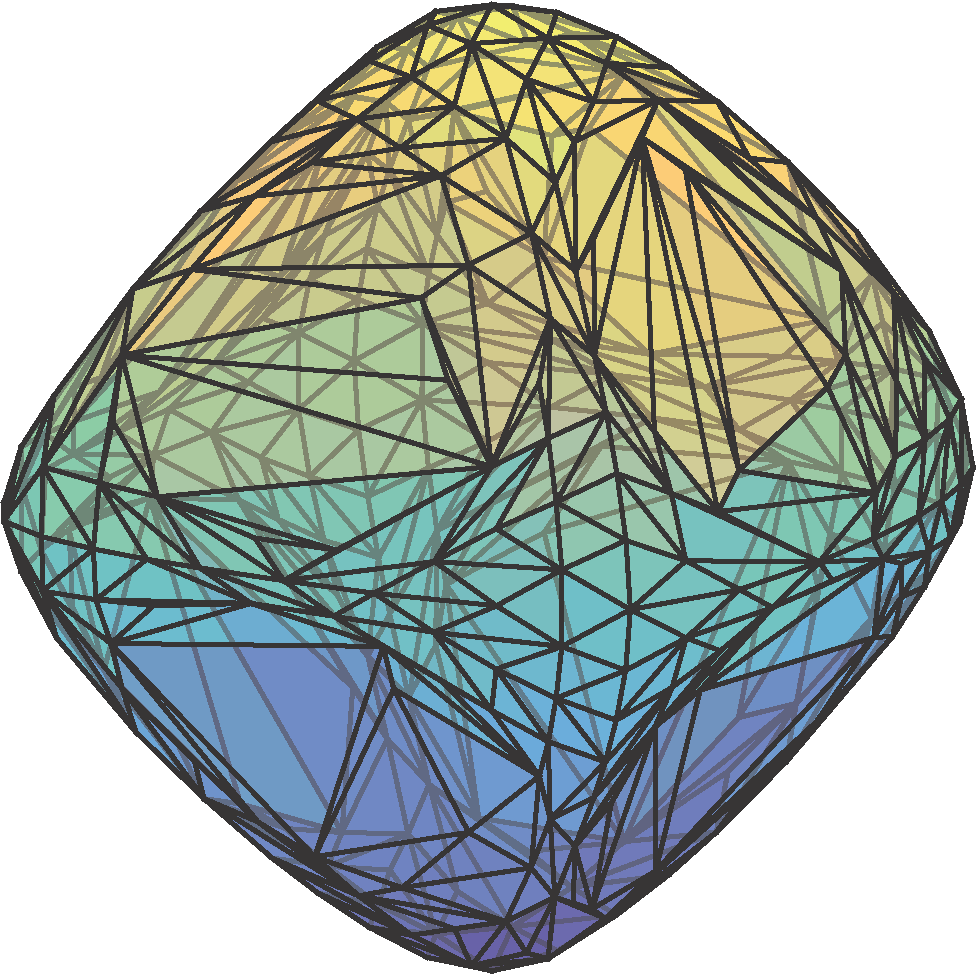}
		\label{fig:Exp_F5_3d_q5}
	\end{subfigure}%
	\begin{subfigure}{.25\textwidth}
		\centering
		\includegraphics[width=0.45\textwidth]{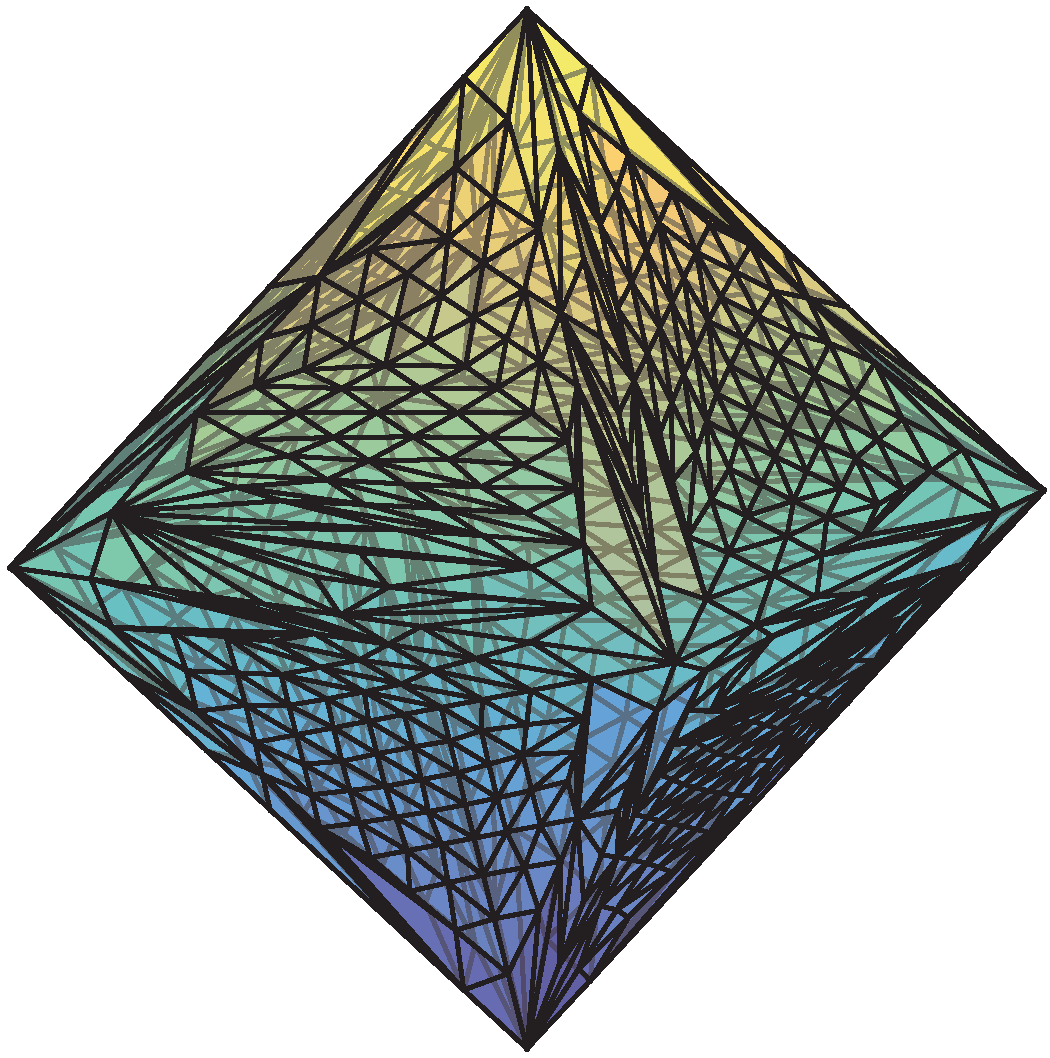}
		\label{fig:Exp_F5_3d_q6}
	\end{subfigure}%
	\caption{Approximating the $\ell_1$-ball in $\mathbb{R}^3$ as a projection of the spectraplices $\fs^3$, $\fs^4$, $\fs^5$, and $\fs^6$ (from left to right).}
	\label{fig:Exp_F5_3d}
\end{figure}

As our next illustration, we consider the following projection of $\fs^4$:
\begin{equation} \label{eq:uncomfortablepillow}
\mathrm{UPillow} = \left\{
( x, y, z )^{\intercal} ~|~ X \in \fs^4, X_{12} = X_{21} = x, X_{23} = X_{32} = y, X_{34} = X_{43} = z
\right\} \subset \R^3.
\end{equation}
We term this convex set the `uncomfortable pillow' and it contains both smooth and non-smooth features on its boundary.  Figure \ref{fig:Exp_F6} shows the reconstruction of $\mathrm{UPillow}$ as linear images of $\fs^3$ and $\fs^4$ computed using our algorithm.  The reconstruction based on $\fs^4$ is exact, while the reconstruction based on $\fs^3$ smoothens out some of the `pointy' features of $\mathrm{UPillow}$; see for example the reconstructions based on $\fs^3$ and on $\fs^4$ viewed in the $(0,1,0)$ direction in Figure \ref{fig:Exp_F6}).

\begin{figure}
	\centering
	\begin{subfigure}{.20\textwidth}
		\centering
		\ifgfx
		\includegraphics[width=0.58\textwidth]{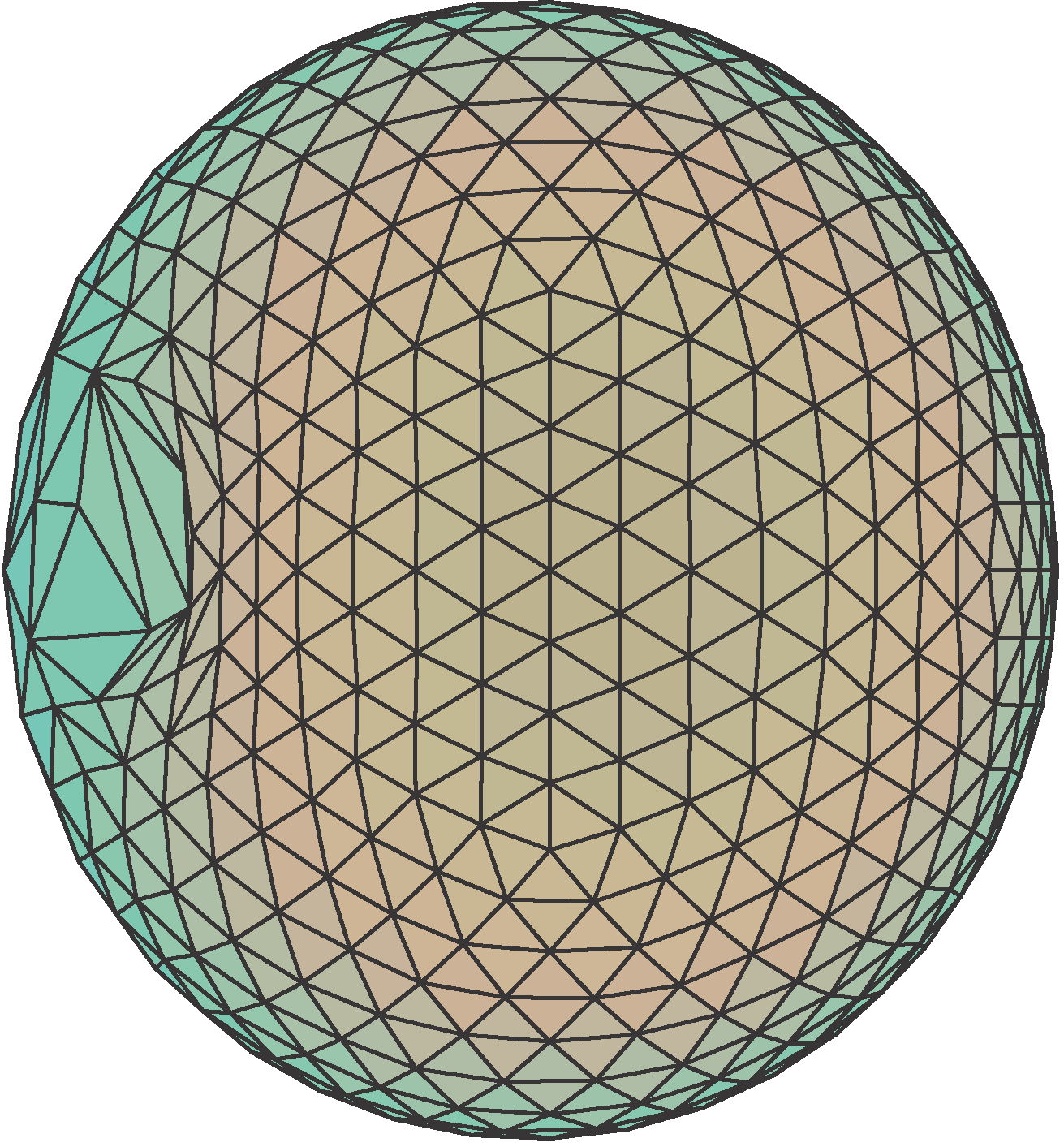}
		\fi
	\end{subfigure}%
	\begin{subfigure}{.20\textwidth}
		\centering
		\ifgfx
		\includegraphics[width=0.58\textwidth]{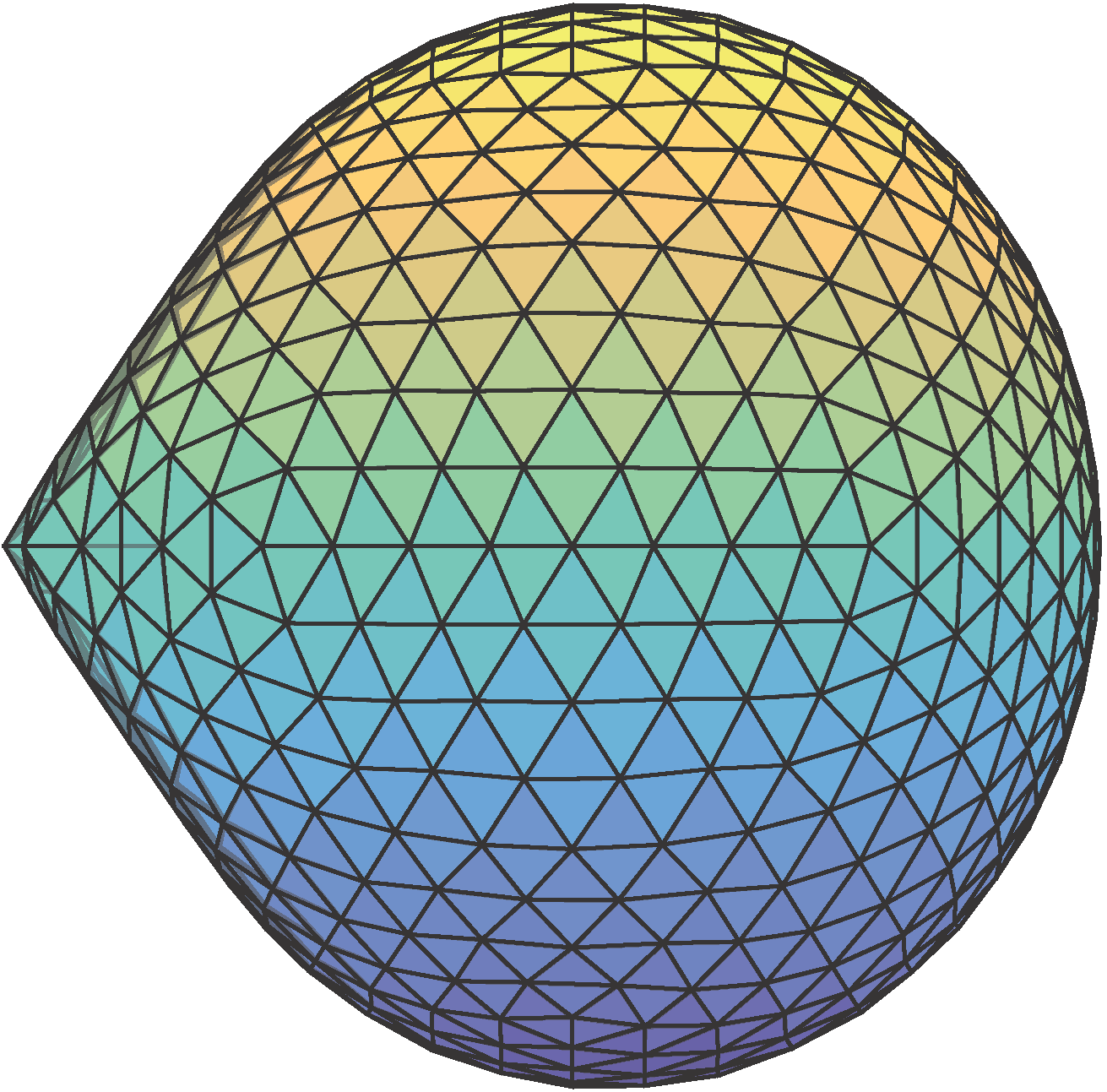}
		\fi
	\end{subfigure}%
	\begin{subfigure}{.20\textwidth}
		\centering
		\ifgfx
		\includegraphics[width=0.58\textwidth]{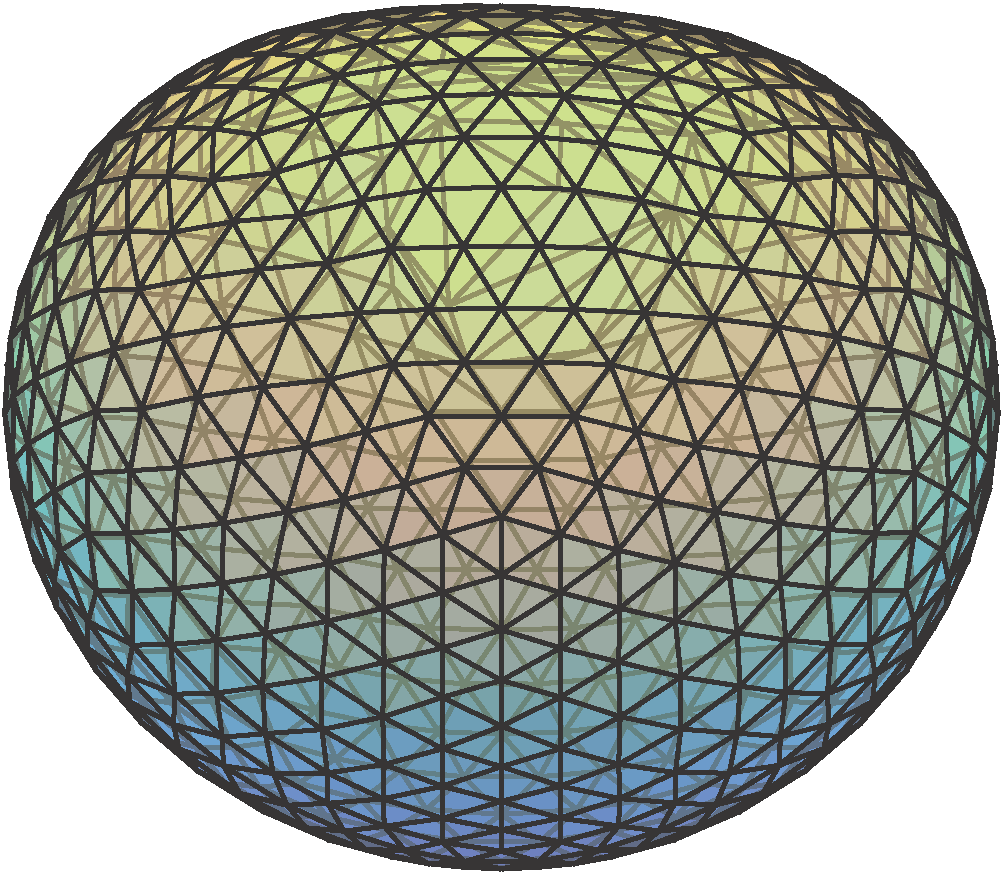}
		\fi
	\end{subfigure}%
	\begin{subfigure}{.20\textwidth}
		\centering
		\ifgfx
		\includegraphics[width=0.58\textwidth]{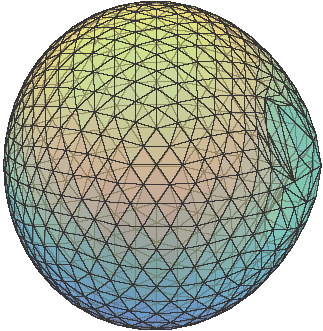}
		\fi
	\end{subfigure}%
	
	\bigskip
	
	\begin{subfigure}{.20\textwidth}
		\centering
		\ifgfx
		\includegraphics[width=0.6\textwidth]{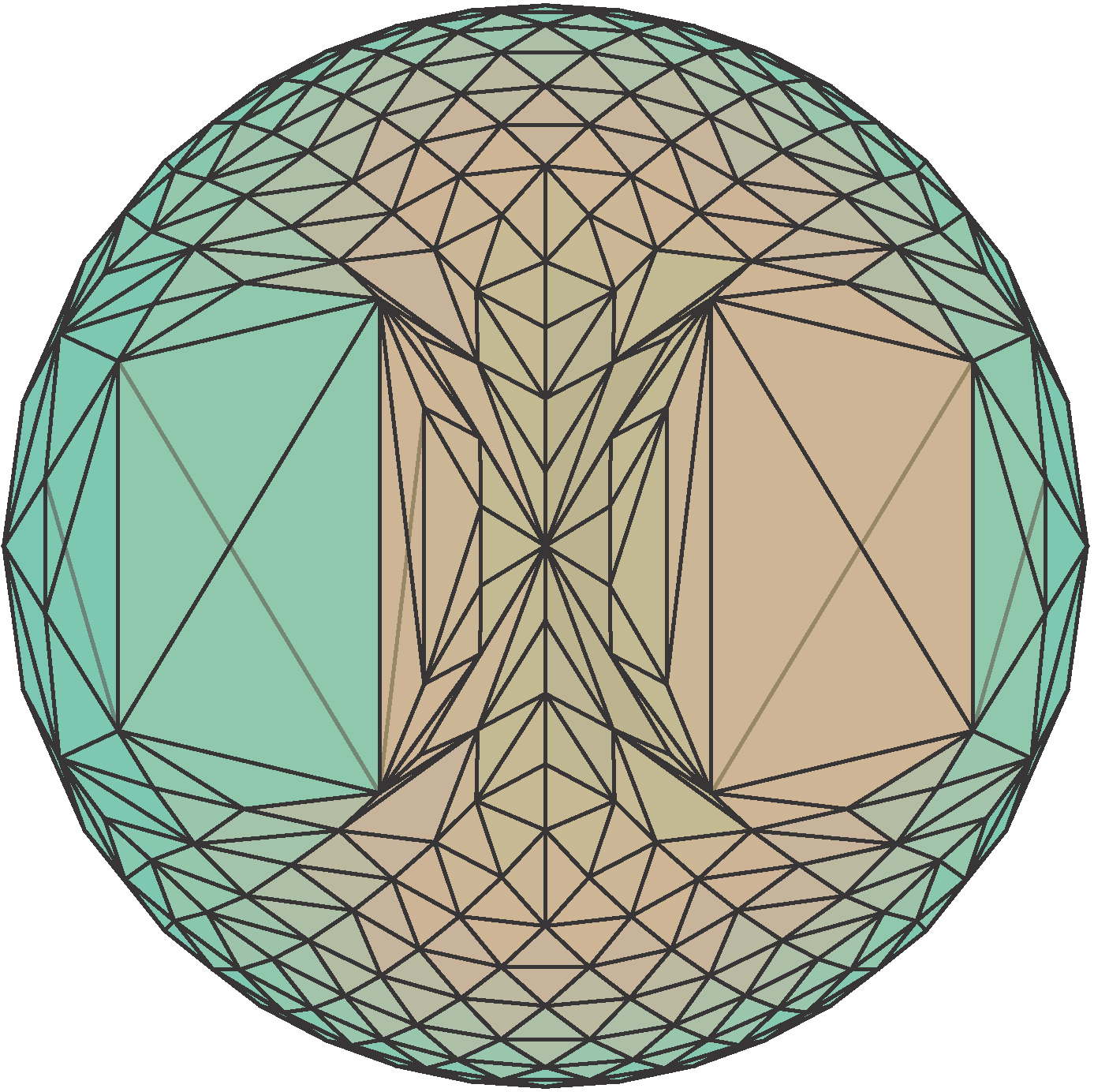}
		\fi
	\end{subfigure}%
	\begin{subfigure}{.20\textwidth}
		\centering
		\includegraphics[width=0.6\textwidth]{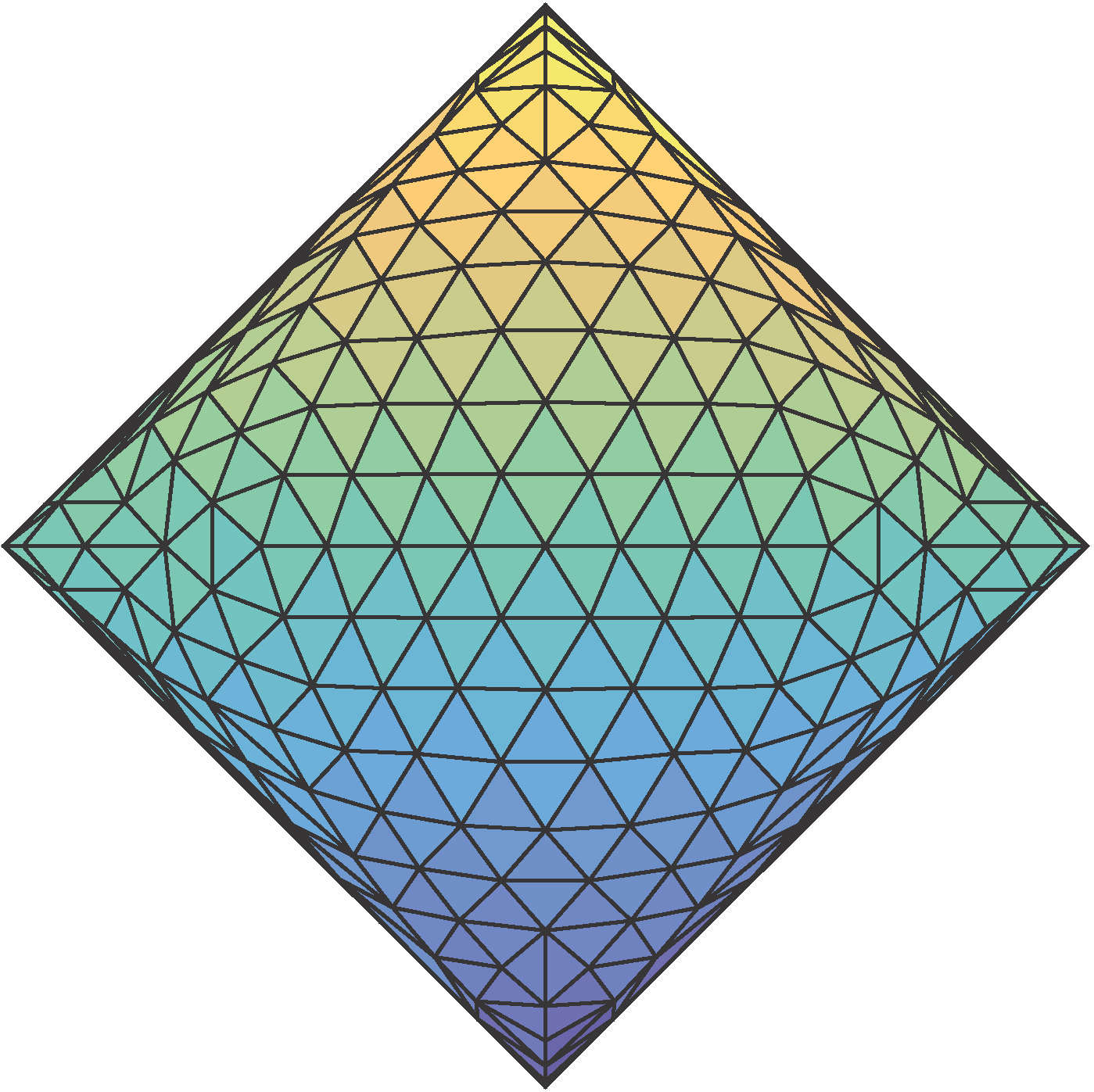}
	\end{subfigure}%
	\begin{subfigure}{.20\textwidth}
		\centering
		\includegraphics[width=0.6\textwidth]{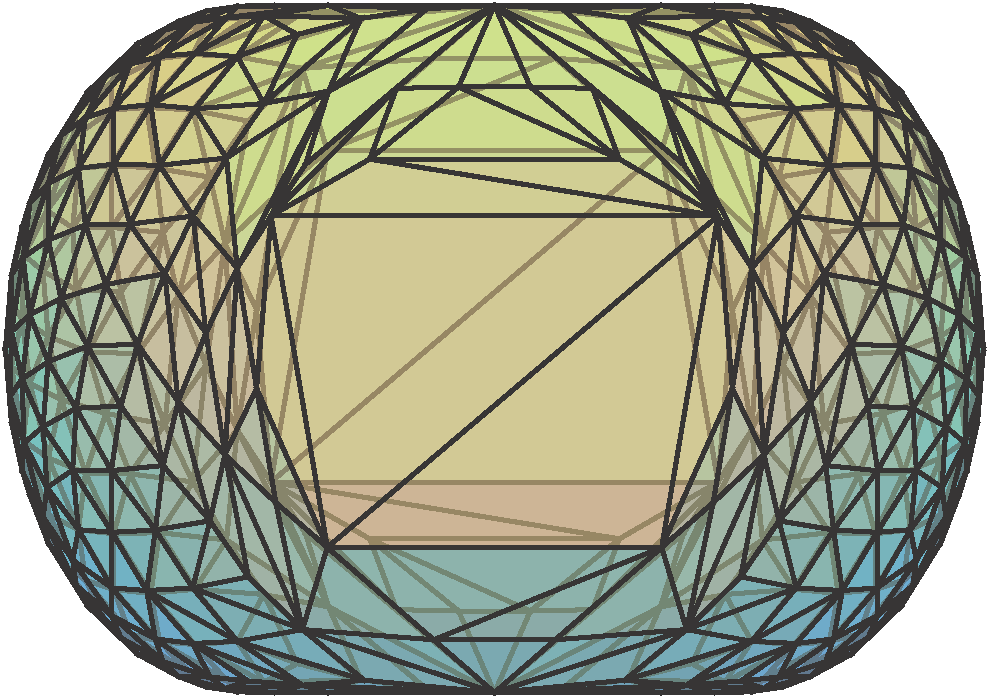}
	\end{subfigure}%
	\begin{subfigure}{.20\textwidth}
		\centering
		\includegraphics[width=0.6\textwidth]{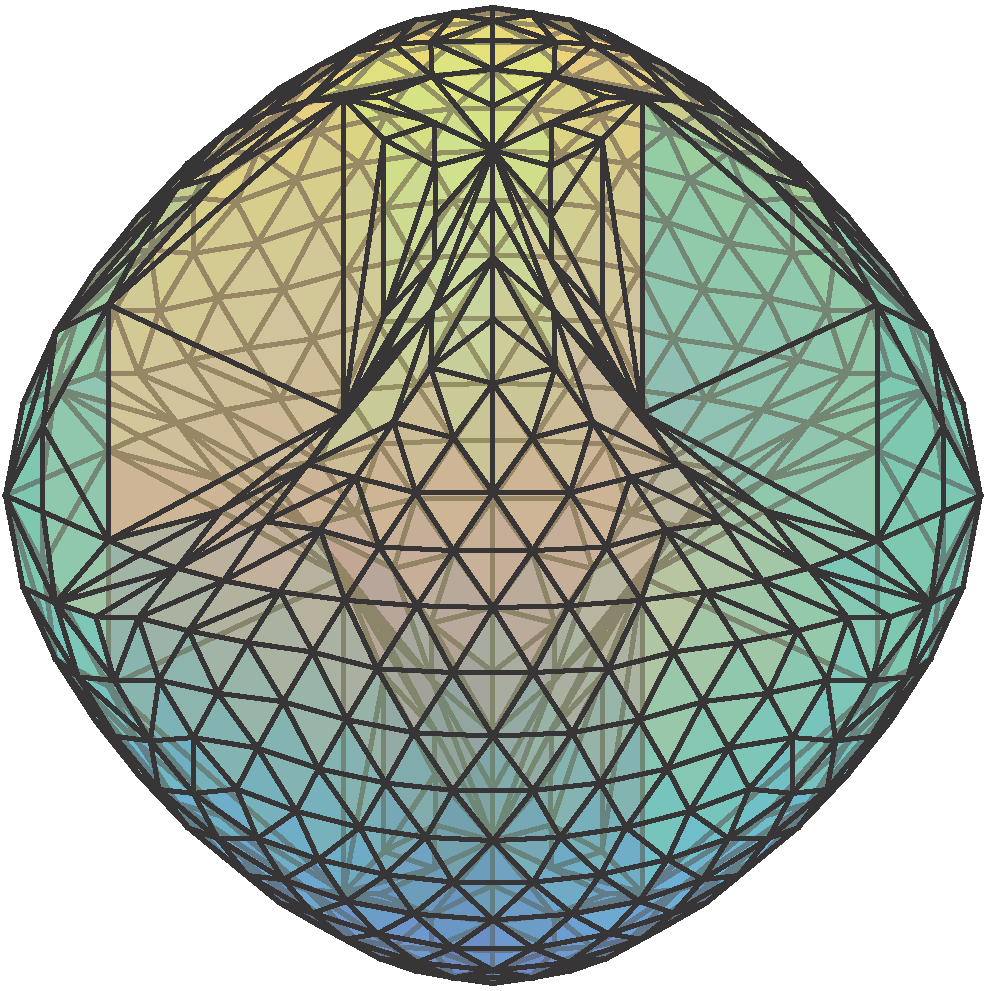}
	\end{subfigure}%
	\caption{Reconstructions of $\K^{\star}$ (defined in \eqref{eq:uncomfortablepillow}) as the projection of $\fs^3$ (top row) and $\fs^4$ (bottom row).  The figures in each row are different views of a single reconstruction, and are orientated in the $(0,0,1)$,$(0,1,0)$,$(1,0,1)$, and $(1,1,0)$ directions (from left to right) respectively.}
	\label{fig:Exp_F6}
\end{figure}


\subsection{Polyhedral Approximations of the $\ell_2$-ball and the Tammes Problem}


In the third set of synthetic experiments, we consider polyhedral approximations of the $\ell_2$-ball in $\R^3$.  This problem has been studied in many contexts under different guises.  For instance, the Tammes problem seeks the optimal placement of $q$ points on $\sph^2$ so as to maximize the minimum pairwise distance, and it is inspired by pattern formation in pollens \cite{Tammes:30}.\footnote{The Tammes problem is a special case of Thompson's problem as well as Smale's 7th problem \cite{Smale:98}.}  Another body of work studies the asymptotics of polyhedral approximations of general compact convex bodies in the (see, for example, \cite{Bronstein08}).  In the optimization literature, polyhedral approximations of the second-order cone have been investigated in \cite{BenNem:01} -- in particular, the approach in \cite{BenNem:01} leads to an approximation that is based on expressing the $\ell_2$-ball via a nested hierarchy of planar spherical constraints, and to subsequently approximate these constraints with regular polygons.

Our focus in the present series of experiments is to investigate polyhedral approximations of the Euclidean sphere from a \emph{computational} perspective by employing the algorithmic tools developed in this paper.  The experimental setup is similar to that of the previous subsection: we supply $2562$ regularly-spaced points in $\sph^2$ (with corresponding support function values equal to one) based on an icosphere discretization as input to \eqref{eq:sc_intro_constrainedlse}, and we select $\C$ to be the simplex $\simp^{q}$ for a range of values of $q$.  Figure \ref{fig:Exp_F4_3d} shows the optimal solutions computed using our method for $q \in \{ 4,5,\ldots,12\}$.  It turns out that the results obtained using our approach are closely related for certain values of $q$ to optimal configurations of the Tammes problem \cite{SchWae:51,Danzer:86}:
\begin{equation} \label{eq:tammes}
\underset{\{\ba_j\}_{j=1}^{q} \subset \sph^{d-1}}{\mathrm{argmax}} ~ \underset{1 \leq k < l \leq q}{\min}  ~~ \mathrm{dist}(\ba_k,\ba_l) = \underset{\{\ba_j\}_{j=1}^{q} \subset \sph^{d-1}}{\mathrm{argmin}} ~ \underset{1 \leq k < l \leq q}{\max}  ~~ \langle \ba_k,\ba_l \rangle.
\end{equation}
Specifically, the face lattice of our solutions is isomorphic to that of the Tammes problem for $q \in \{ 4,5,6,7,12\}$, which suggests that these configurations are stable and optimal for a broader class of objectives.  We are currently not aware if the distinction between the solutions to the two sets of problems for $q\in \{8,9,10,11 \}$ is a result of our method recovering a locally optimal solution (in generating these results, we apply $500$ initializations for each instance of $q$), or if it is inherently due to the different objectives that the two problems seek to optimize.  For the case of $q=8$, the difference appears to be due to the latter reason as an initialization supplied to our algorithm based on a configuration that is isomorphic to the Tammes solution led to a suboptimal local minimum.

\begin{figure}
	\centering
	\begin{subfigure}{.11\textwidth}
		\centering
		\includegraphics[width=0.9\textwidth]{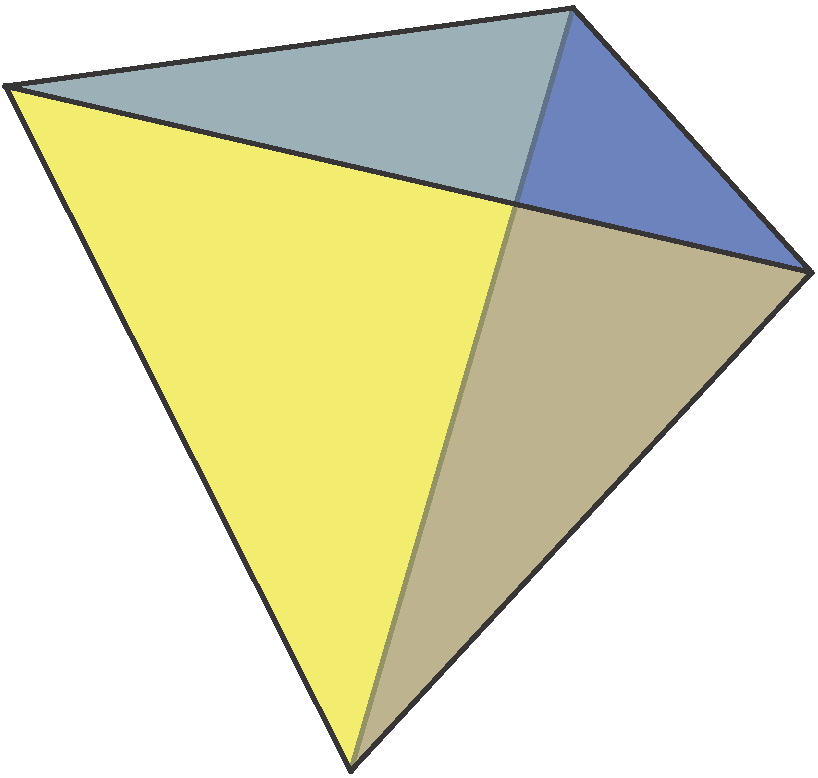}
	\end{subfigure}%
	\begin{subfigure}{.11\textwidth}
		\centering
		\includegraphics[width=0.9\textwidth]{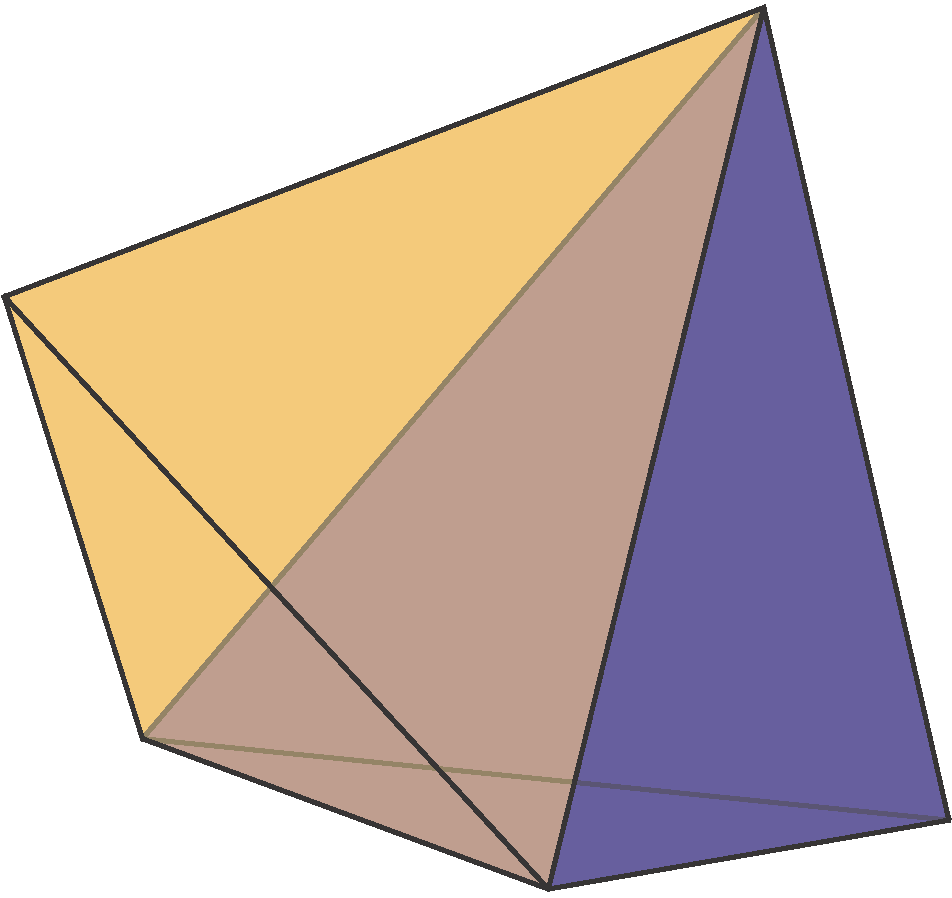}
	\end{subfigure}%
	\begin{subfigure}{.11\textwidth}
		\centering
		\includegraphics[width=0.9\textwidth]{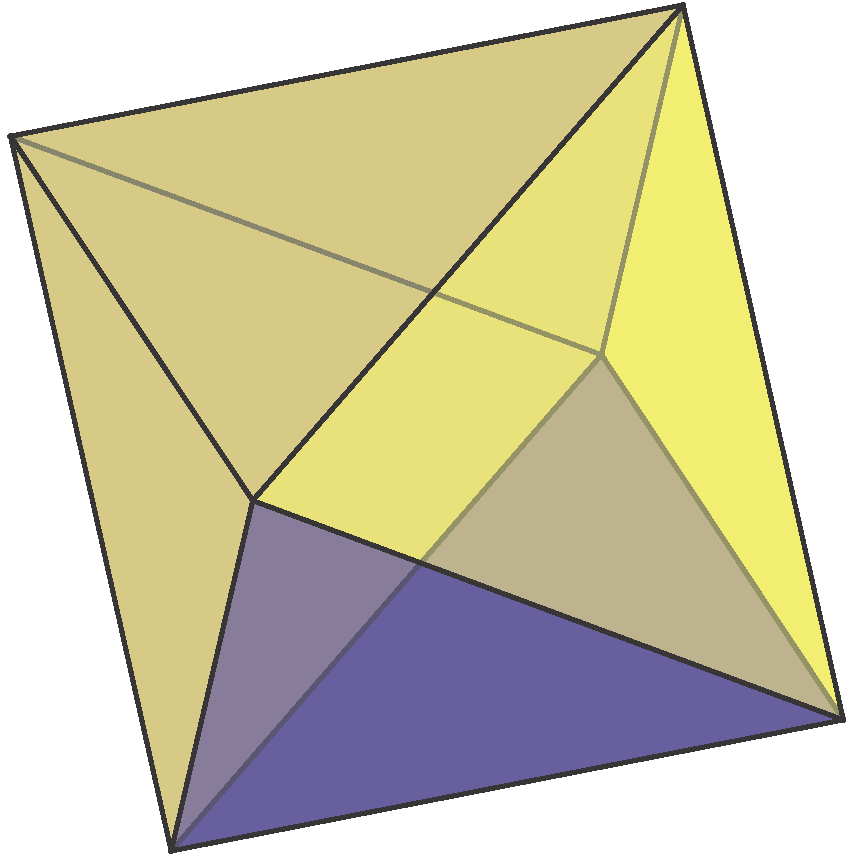}
	\end{subfigure}%
	\begin{subfigure}{.11\textwidth}
		\centering
		\includegraphics[width=0.9\textwidth]{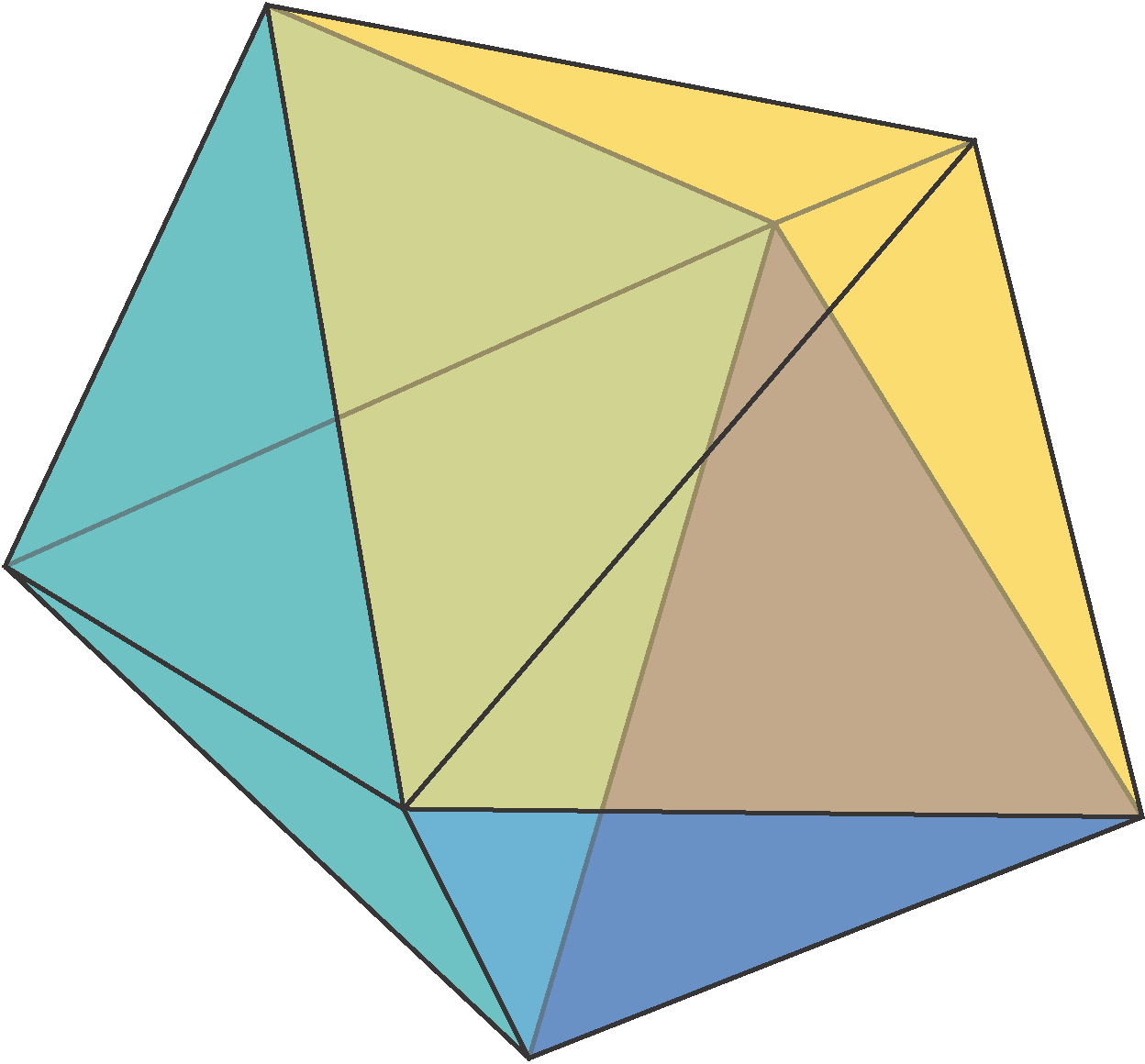}
	\end{subfigure}%
	\begin{subfigure}{.11\textwidth}
		\centering
		\includegraphics[width=0.9\textwidth]{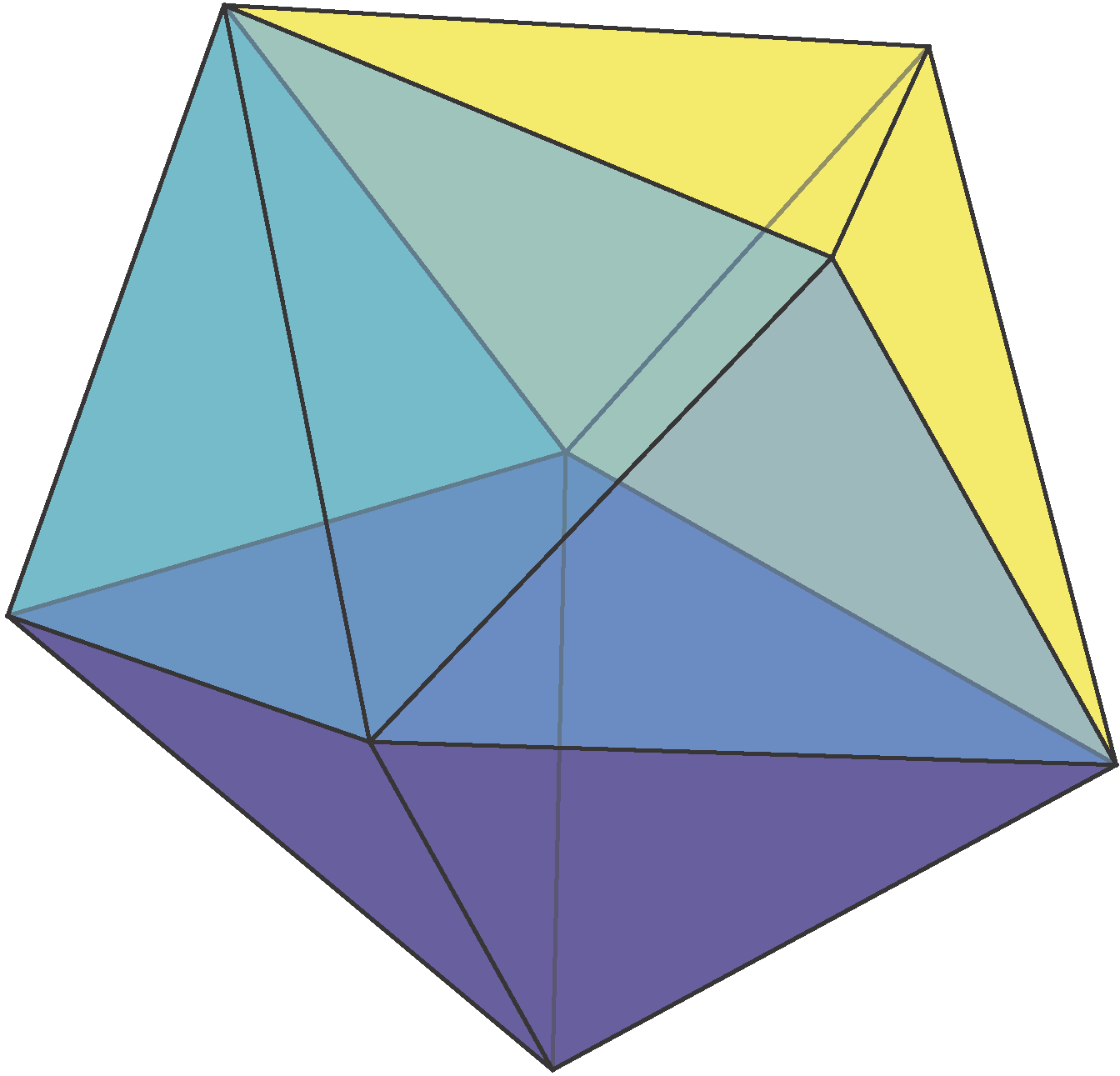}
	\end{subfigure}%
	\begin{subfigure}{.11\textwidth}
		\centering
		\includegraphics[width=0.9\textwidth]{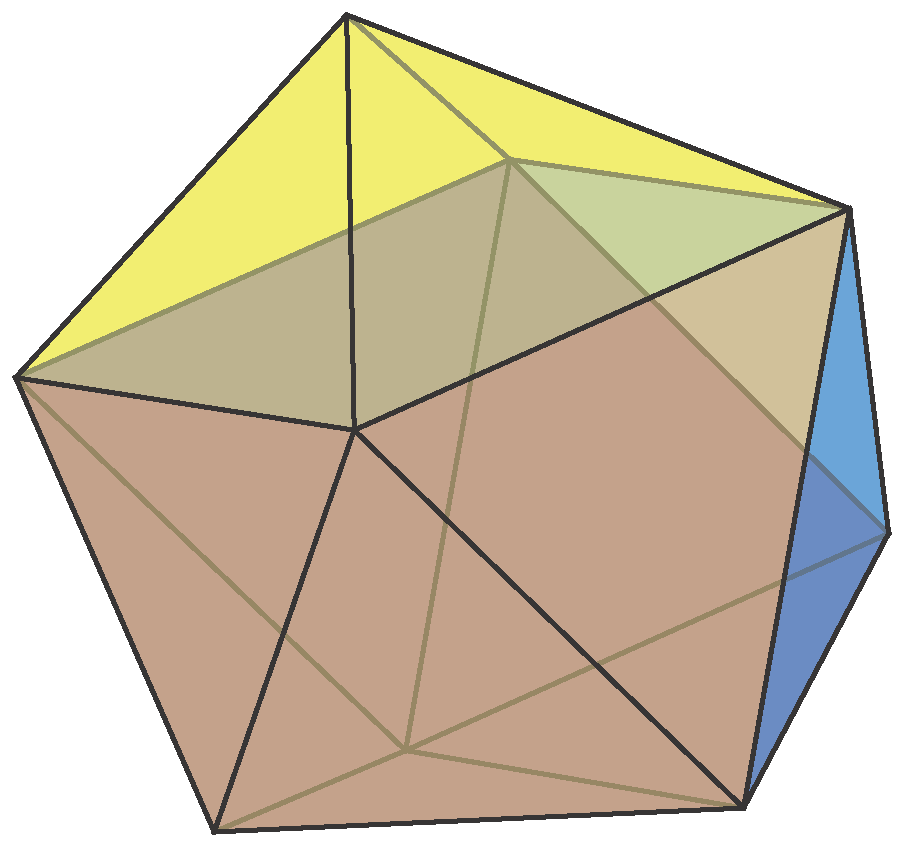}
	\end{subfigure}%
	\begin{subfigure}{.11\textwidth}
		\centering
		\includegraphics[width=0.9\textwidth]{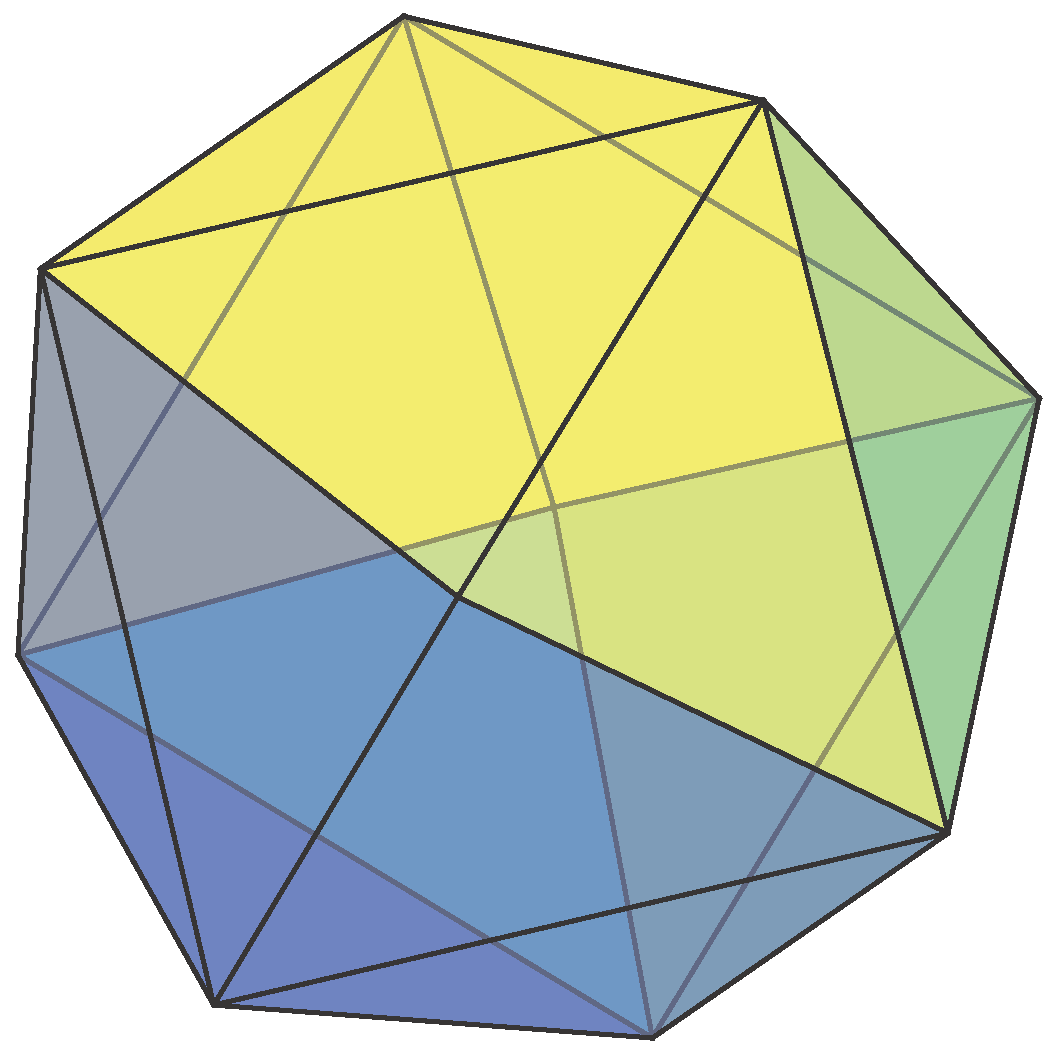}
	\end{subfigure}%
	\begin{subfigure}{.11\textwidth}
		\centering
		\includegraphics[width=0.9\textwidth]{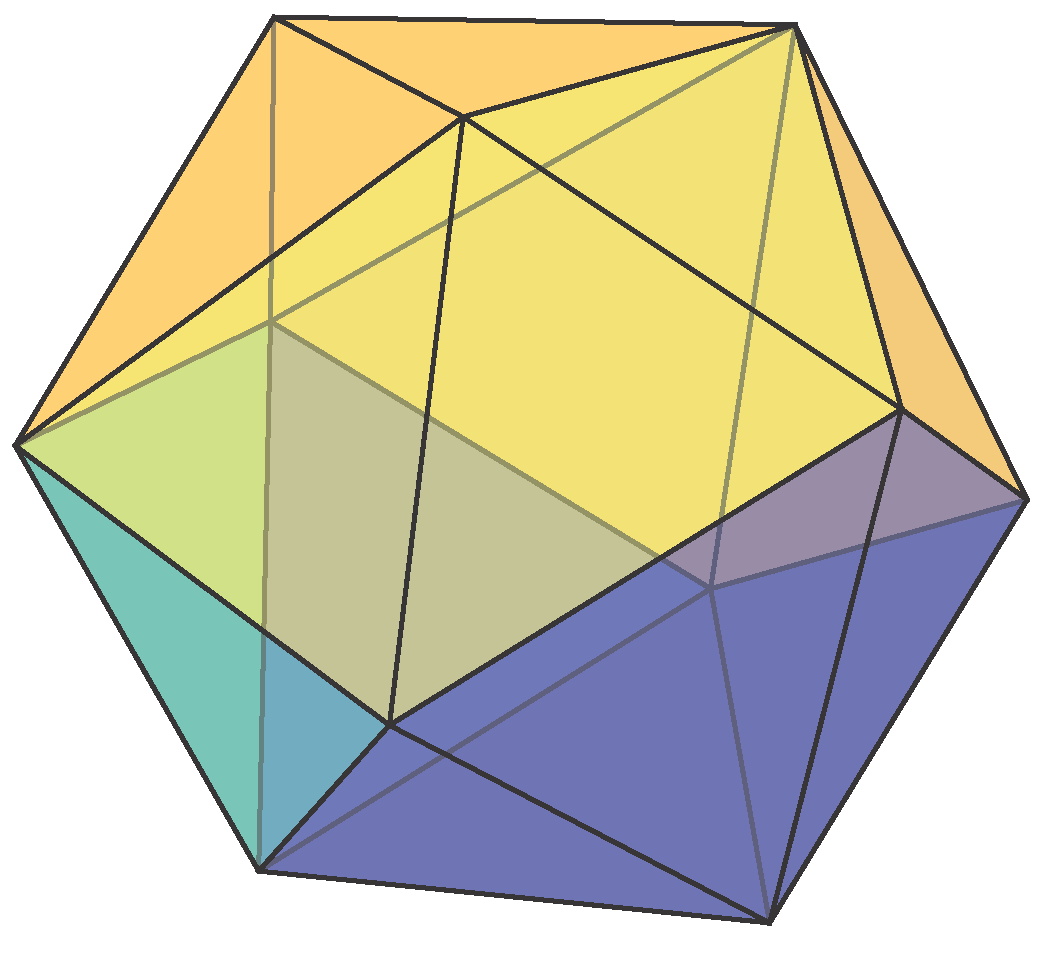}
	\end{subfigure}%
	\begin{subfigure}{.11\textwidth}
		\centering
		\includegraphics[width=0.9\textwidth]{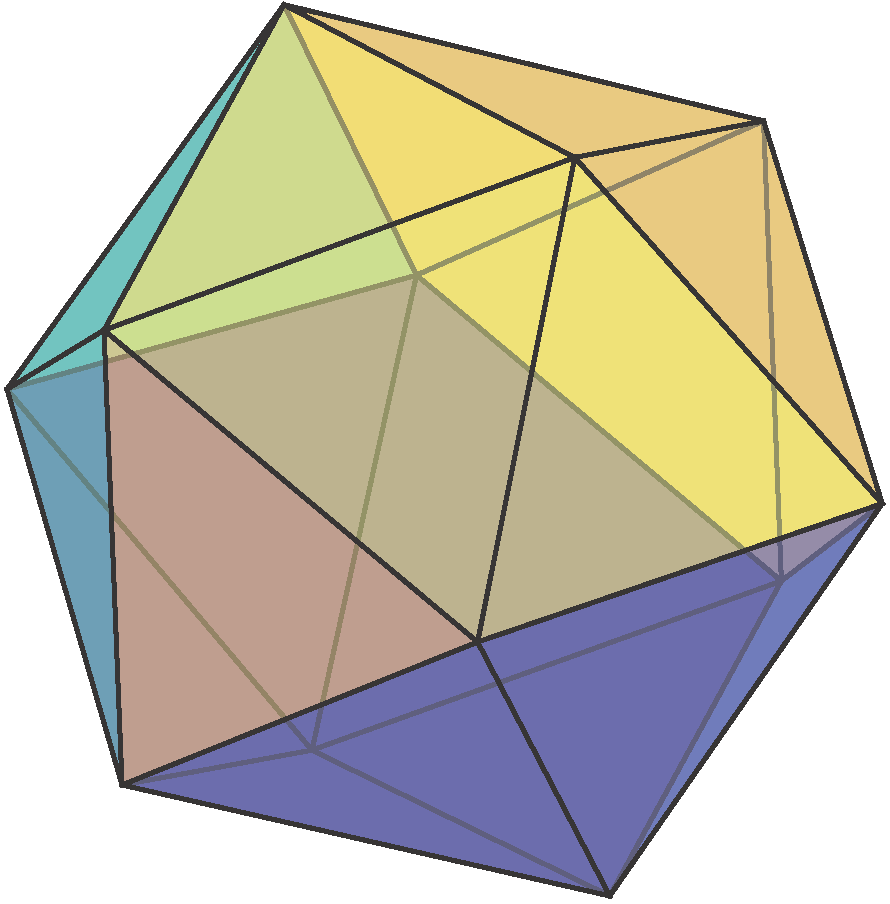}
	\end{subfigure}%
	\caption{Approximating the $\ell_2$-ball in $\mathbb{R}^3$ as the projection of $\simp^{q}$ for $q\in \{4,5,\ldots,12\}$ (from left to right).}
	\label{fig:Exp_F4_3d}
\end{figure}


\subsection{Reconstruction of a Human Lung} \label{sec:numexp_lung}
In the final set of experiments we apply our algorithm to reconstruct a convex mesh of a human lung.  The purpose of this experiment is to demonstrate the utility of our algorithm in a setting in which the underlying object is not convex.  Indeed, in many applications in practice of reconstruction from support function evaluations, the underlying set of interest is not convex; however, due to the nature of the measurements available, one seeks a reconstruction of the convex hull of the underlying set.  In the present example, the set of interest is obtained from the CT scan of the left lung of a healthy individual \cite{ChestCT}.  We note that a priori it is unclear whether the convex hull of the lung is well-approximated as a linear image of either a low-dimensional simplex or a low-dimensional spectraplex.

We first obtain $n=50$ noiseless support function evaluations of the lung (note that this object lies in $\R^3$) in directions that are generated uniformly at random over the sphere $\sph^2$.  In the top row of Figure \ref{fig:Exp_Lung_3d} we show the reconstructions as projections of $\fs^{q}$ for $q\in \{3,4,5,6\}$, and we contrast these with the LSE.  We repeat the same experiment with $n=300$ measurements, with the reconstructions shown in the bottom row of Figure \ref{fig:Exp_Lung_3d}.

To concretely compare the results obtained using our framework and those based on the LSE, we contrast the description complexity -- the number of parameters used to specify the reconstruction -- of the estimates obtained from both frameworks.  An estimator computed using our approach is specified by a projection map $A \in L(\R^q,\R^d)$, and hence it requires $dq$ parameters, while the LSE proposed by the algorithm in \cite{GarKid:09} assigns a vertex to every measurement, and hence it requires $d n$ parameters.  The LSE using $n=300$ measurements requires $3 \times 300$ parameters to specify whereas the estimates obtained using our framework that are specified as projections of $\fs^{5}$ and $\fs^{6}$ -- these estimates offer comparable quality to those of the LSE -- require $3 \times 15$ and $3 \times 21$ parameters, respectively.  This substantial discrepancy highlights the drawback of using polyhedral sets of growing complexity to approximate non-polyhedral objects in higher dimensions.

\begin{figure}
	\centering
	\begin{subfigure}{.20\textwidth}
		\centering
		\includegraphics[width=0.6\textwidth]{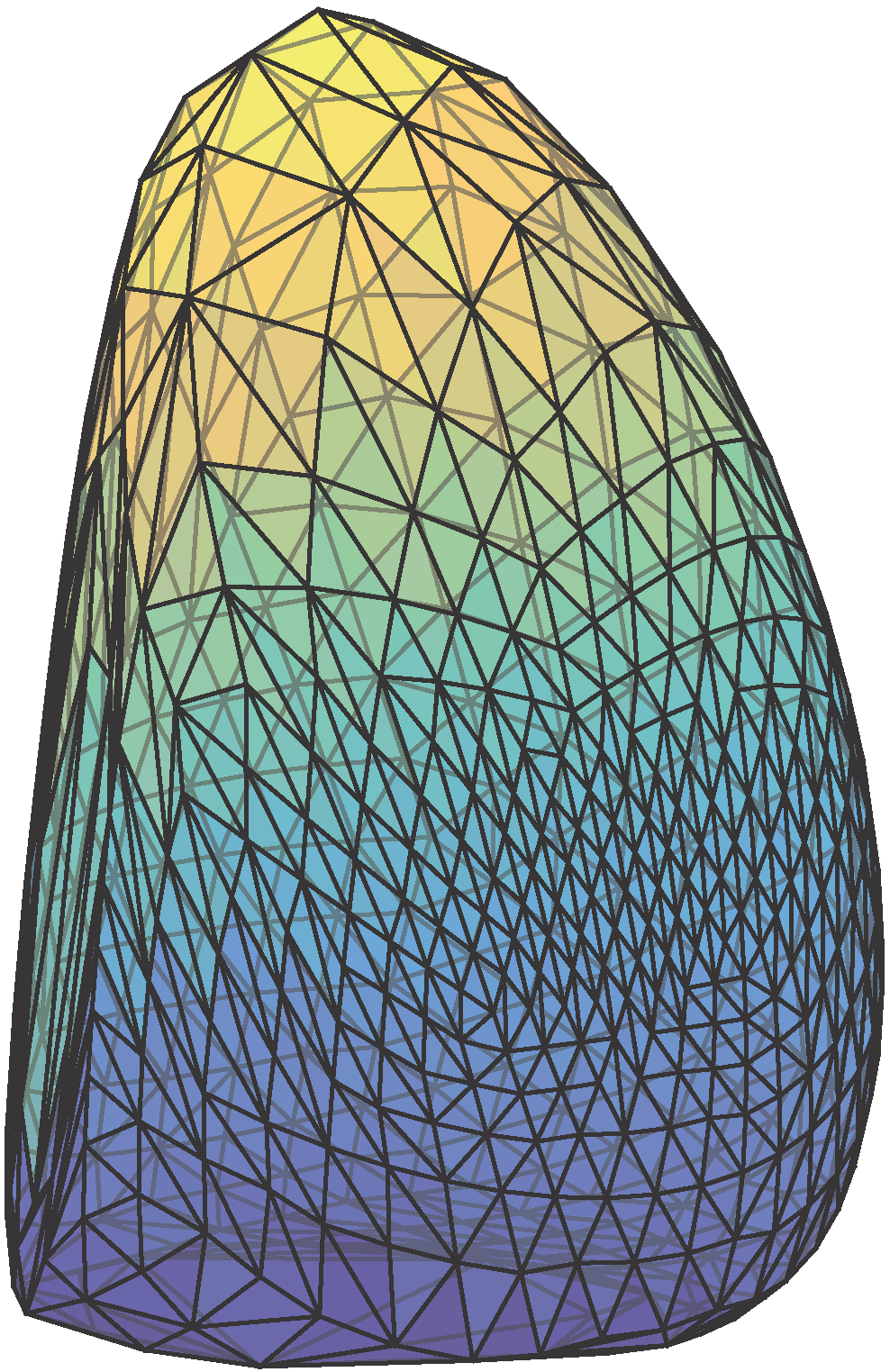}
		\caption{$\fs^{3}$}
	\end{subfigure}%
	\begin{subfigure}{.20\textwidth}
		\centering
		\includegraphics[width=0.6\textwidth]{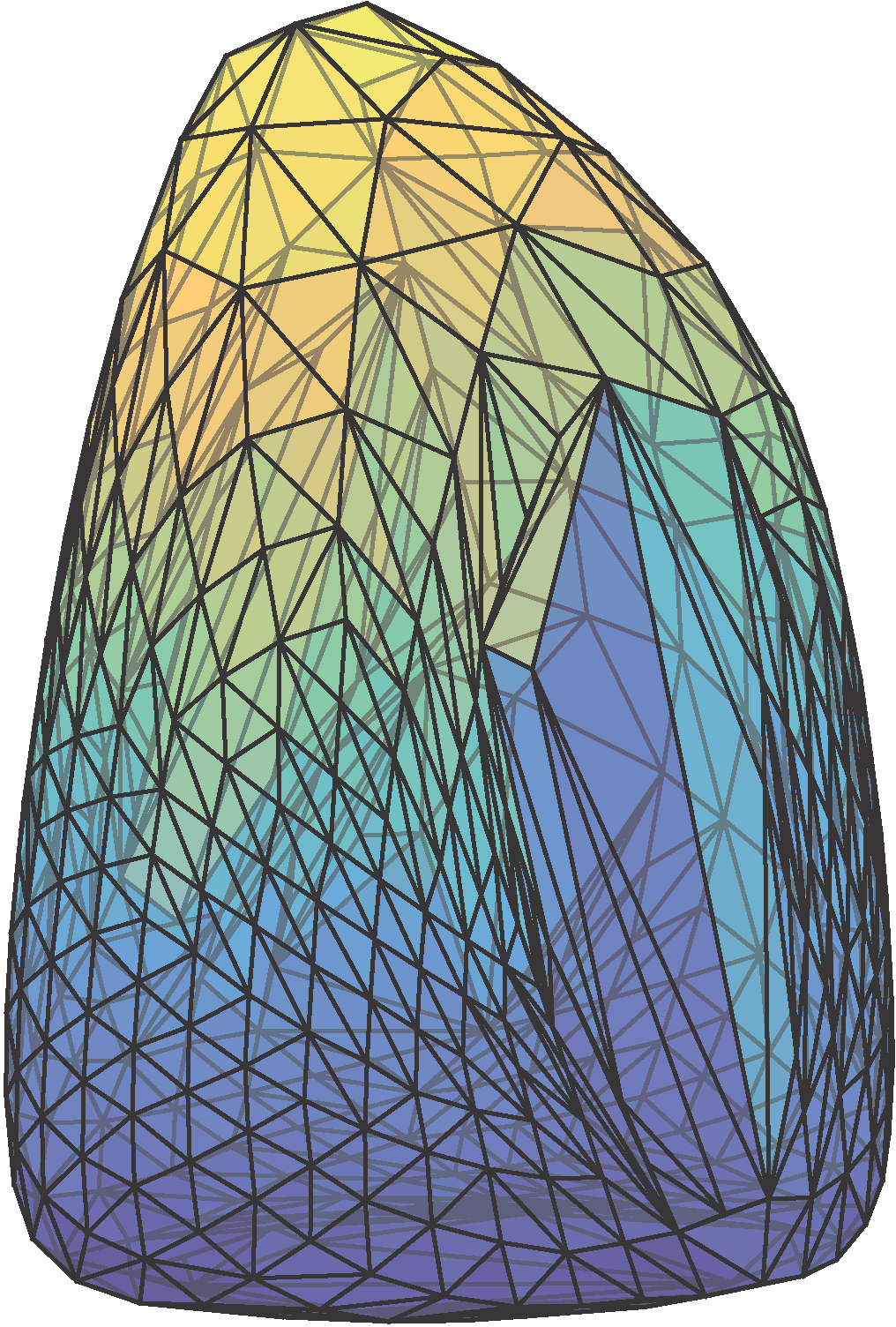}
		\caption{$\fs^{4}$}
	\end{subfigure}%
	\begin{subfigure}{.20\textwidth}
		\centering
		\includegraphics[width=0.6\textwidth]{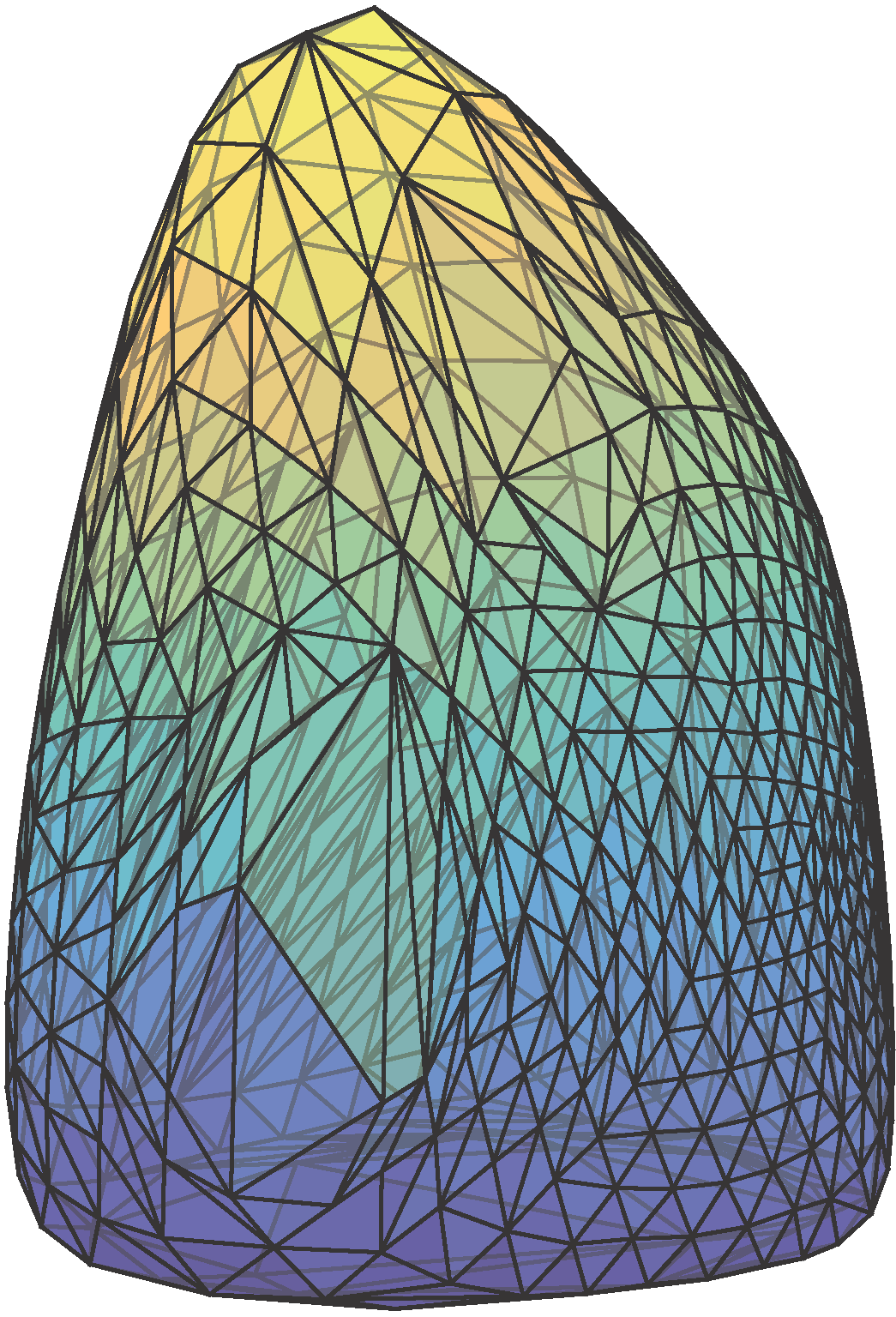}
		\caption{$\fs^{5}$}
	\end{subfigure}%
	\begin{subfigure}{.20\textwidth}
		\centering
		\includegraphics[width=0.6\textwidth]{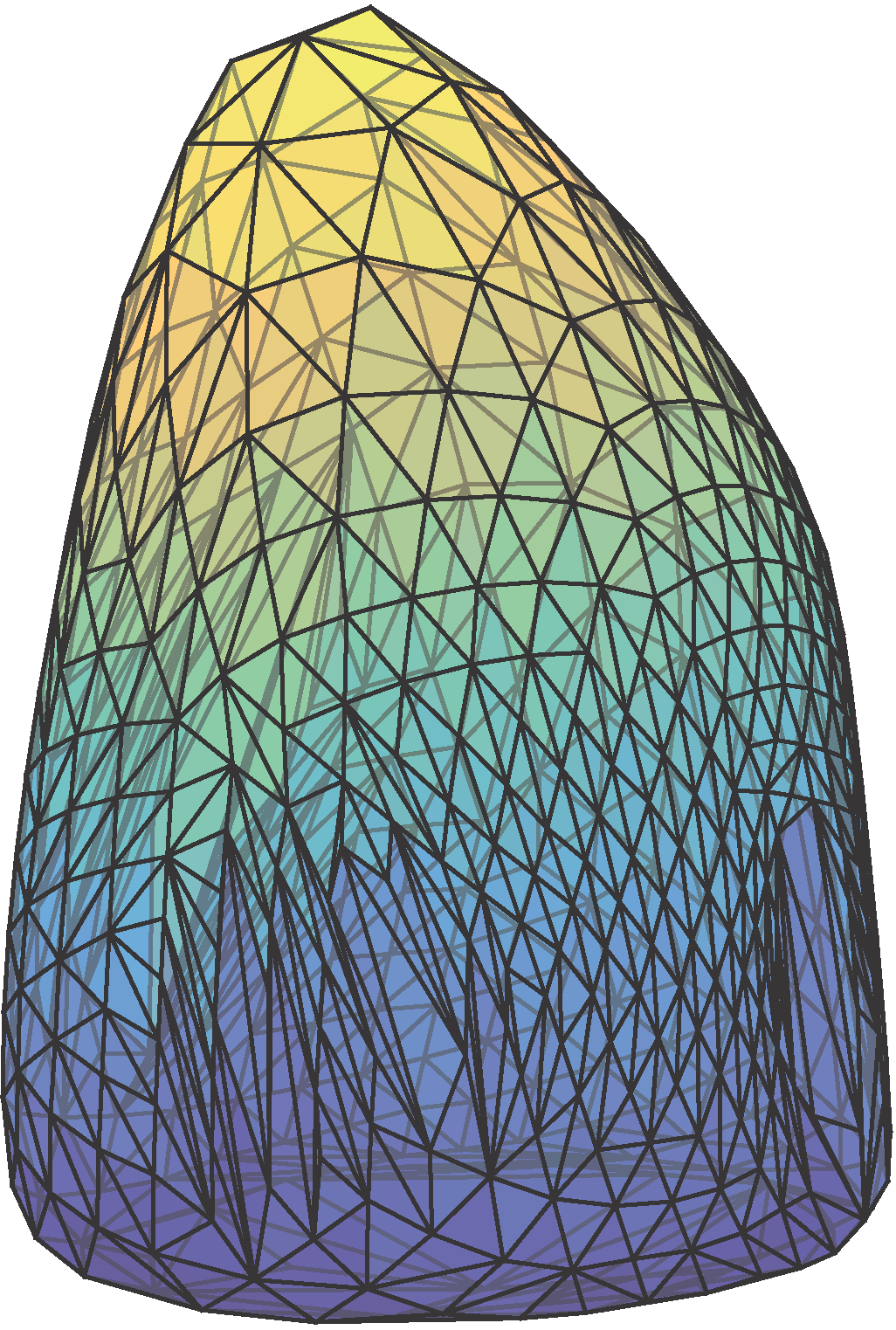}
		\caption{$\fs^{6}$}
	\end{subfigure}%
	\begin{subfigure}{.20\textwidth}
		\centering
		\includegraphics[width=0.6\textwidth]{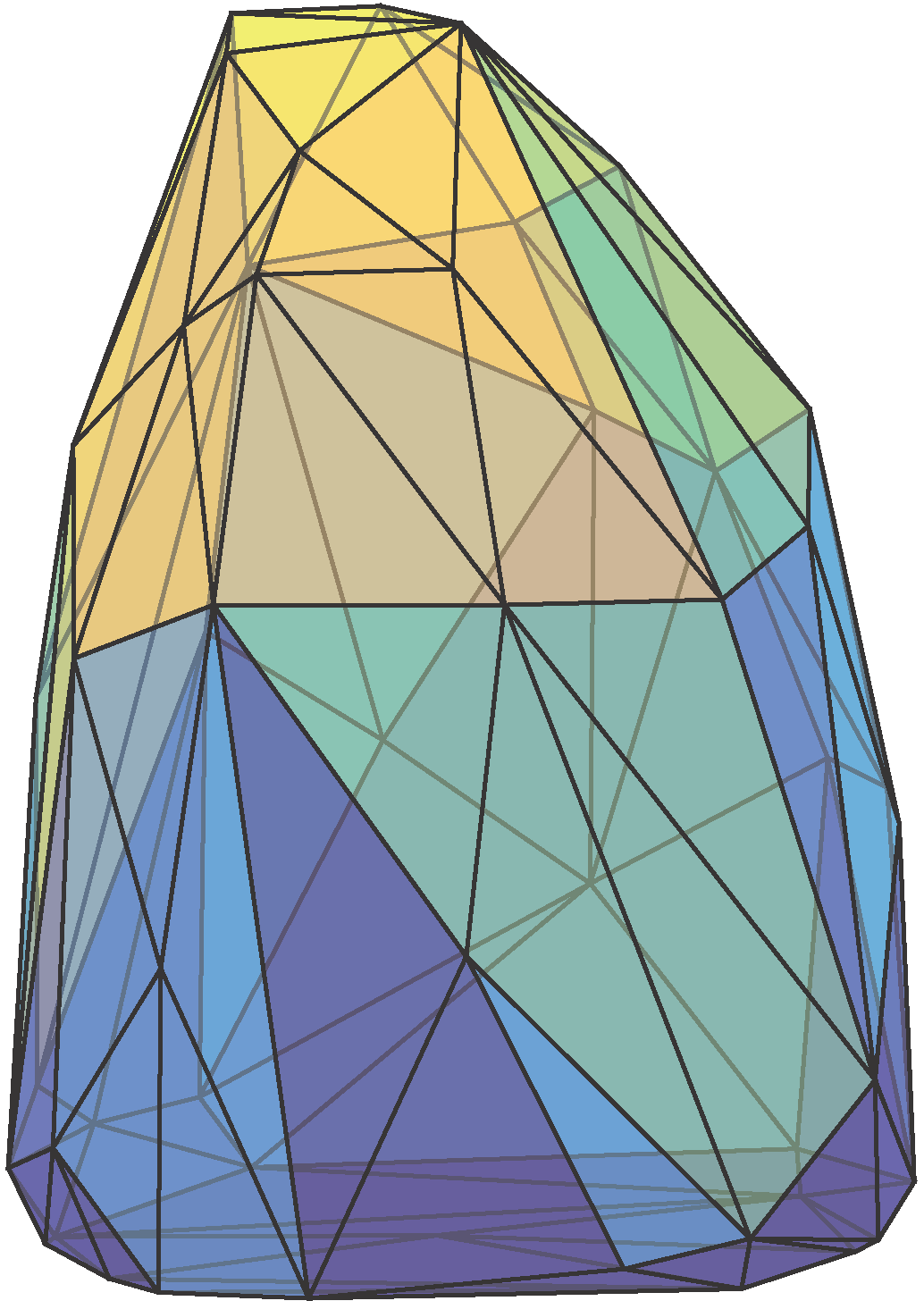}
		\caption{LSE}
	\end{subfigure}%
	
	\bigskip
	
	\begin{subfigure}{.20\textwidth}
		\centering
		\includegraphics[width=0.6\textwidth]{Exp_3D_Lung_q3_n300}
		\caption{$\fs^{3}$}
	\end{subfigure}%
	\begin{subfigure}{.20\textwidth}
		\centering
		\includegraphics[width=0.6\textwidth]{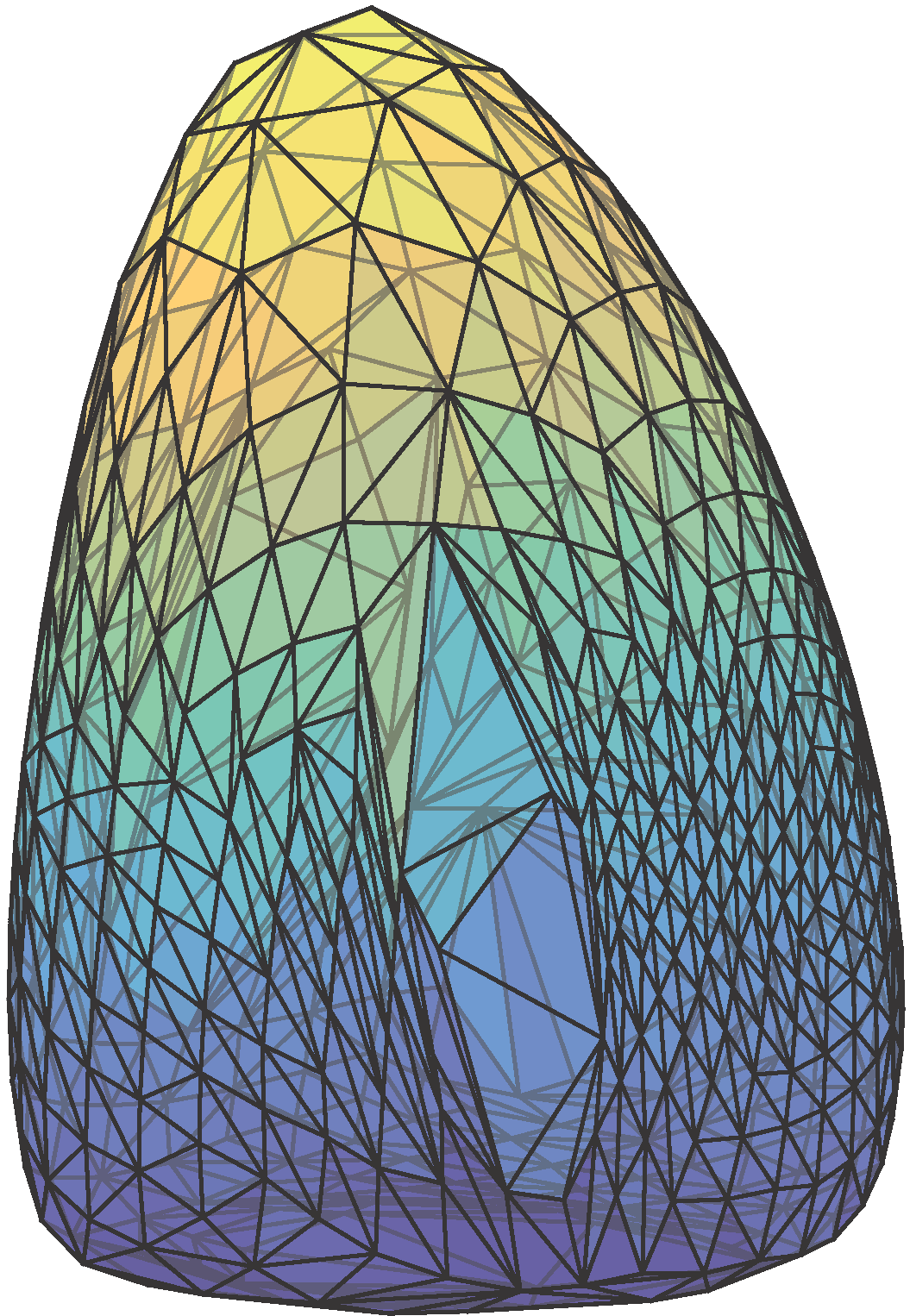}
		\caption{$\fs^{4}$}
	\end{subfigure}%
	\begin{subfigure}{.20\textwidth}
		\centering
		\includegraphics[width=0.6\textwidth]{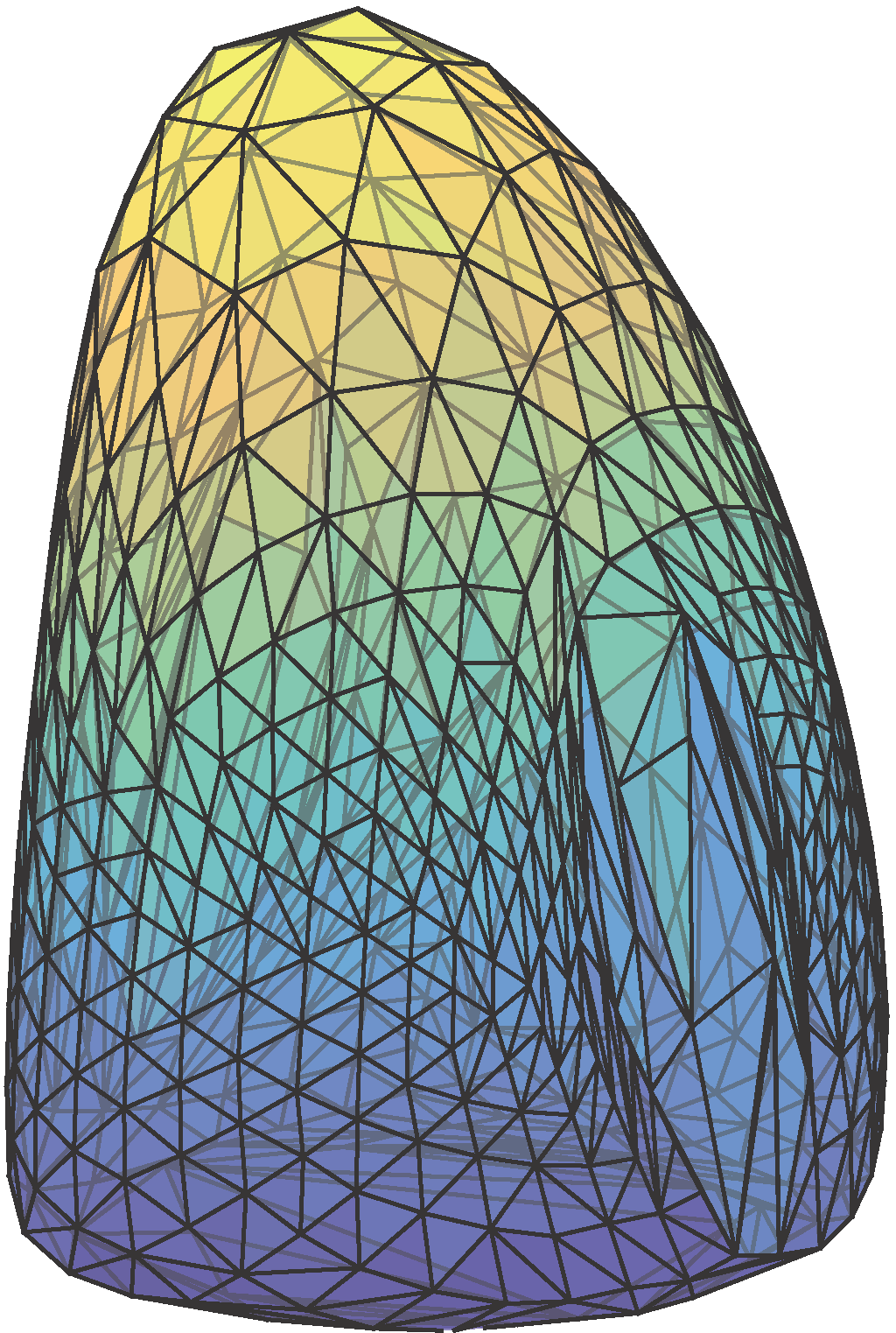}
		\caption{$\fs^{5}$}
	\end{subfigure}%
	\begin{subfigure}{.20\textwidth}
		\centering
		\includegraphics[width=0.6\textwidth]{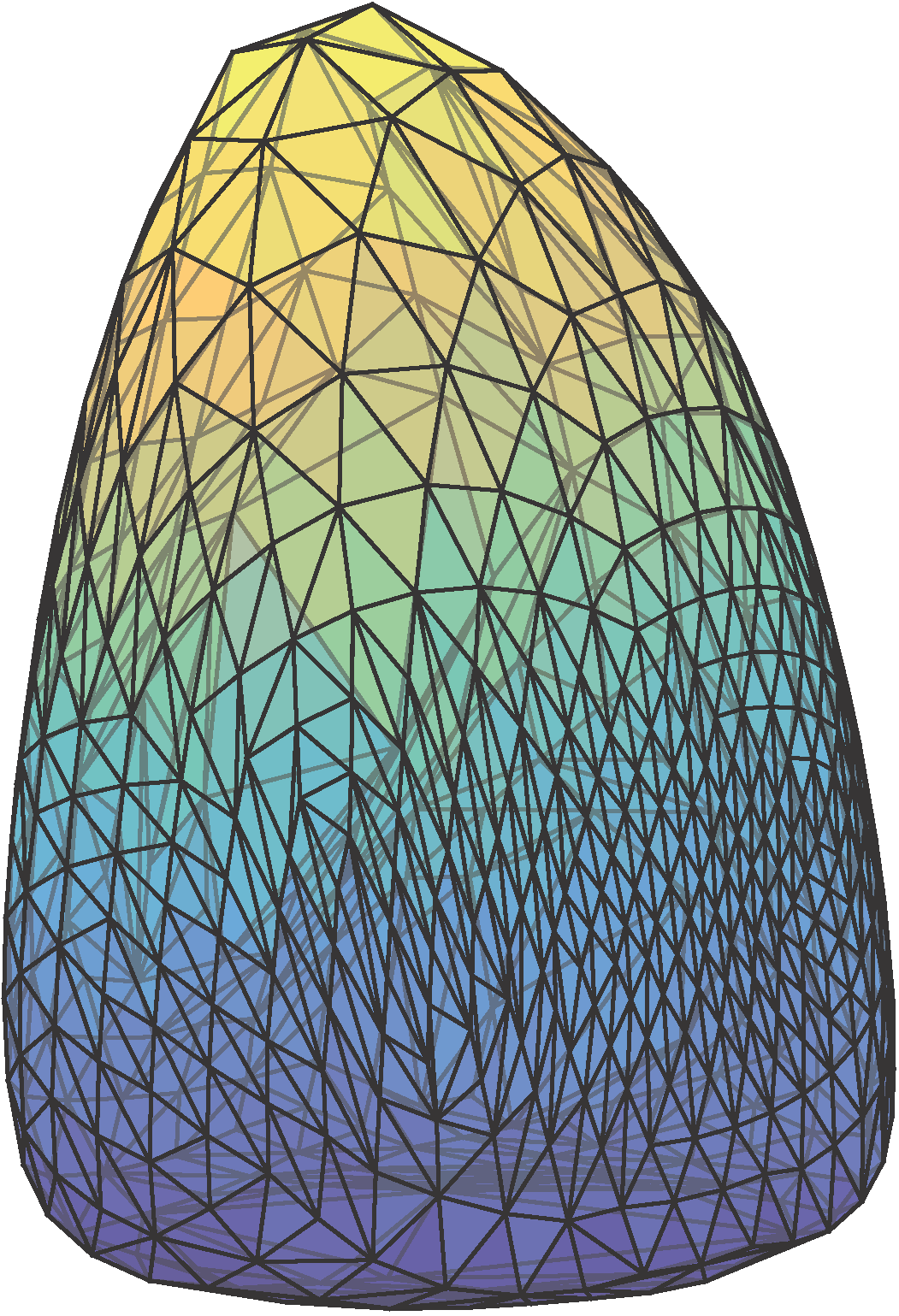}
		\caption{$\fs^{6}$}
	\end{subfigure}%
	\begin{subfigure}{.20\textwidth}
		\centering
		\includegraphics[width=0.6\textwidth]{Exp_3D_Lung_LS_n300}
		\caption{LSE}
	\end{subfigure}%
	\caption{Reconstructions of the left lung from $50$ support function measurements (top row) and $300$ support function measurements (bottom row). Subfigures (a)-(d) and (f)-(i) are projections of spectraplices with dimensions as indicated, and subfigures (e) and (j) are LSEs.}
	\label{fig:Exp_Lung_3d}
\end{figure}


\section{Conclusions and Future Directions} \label{sec:conc}

In this paper we describe a framework for fitting tractable convex sets to noisy support function evaluations.  Our approach provides many advantages in comparison to the previous LSE-based methods, most notably in settings in which the measurements available are noisy or small in number as well as those in which the underlying set to be reconstructed is non-polyhedral.  We discuss here some potential future directions:

\textbf{Informed selection of model complexity.} In practice, a suitable choice of the dimension of the simplex or the spectraplex to employ in \eqref{eq:sc_intro_constrainedlse} may not be available in advance. Lower-dimensional choices for $\C$ provide more concisely-described reconstructions but may not fit the data well, while higher-dimensional choices provide better fidelity to the data at the risk of overfitting.  Consequently, it is of practical relevance to develop methods to select $\C$ in a data-driven manner.  We describe next a stylized experiment to choose $\C$ via cross-validation.

In the first illustration, we are given $100$ support function measurements of the $\ell_1$-ball in $\R^3$ corrupted by Gaussian noise with standard deviation $\sigma = 0.1$.  We obtain $50$ random partitions of this data into two subsets of equal size.  For each partition, we solve \eqref{eq:sc_intro_constrainedlse} with $\C = \simp^q$ (with diffferent choices of $q$) on the first subset and evaluate the mean-squared error on the second subset.  The left subplot of Figure \ref{fig:Exp_F7} shows the average mean-squared error over the $50$ partitions.  We observe that initially the error decreases as $q$ increases as a more expressive model allows us to better fit the data, and subsequently, the error plateaus out.  Consequently, in this experiment an appropriate choice of $\C$ would be $\simp^6$.  In our second illustration, we are given $200$ support function measurements corrupted by Gaussian noise with standard deviation $\sigma = 0.05$ of a set $\K_{S3} = \mathrm{conv}(\sph_1 \cup \sph_2 \cup \sph_3) \subset \R^3$, where $\sph_1,\sph_2,\sph_3$ are defined as follows:
\begin{equation*}
	\text{For } j=1,2,3: ~~~ \sph_j = Q_j \left( \left\{ \left(
	\begin{array}{c}
		\cos \theta \\ 1  \\ \sin \theta
	\end{array} \right) : \theta \in \R
	\right\} \right), \quad Q_j =
	\left(\begin{array}{ccc}
		\cos (2 \pi j / 3) & -\sin (2 \pi j / 3) & 0 \\
		\sin (2 \pi j / 3) & \cos (2 \pi j / 3) & 0 \\
		0 & 0 & 1
	\end{array} \right).
\end{equation*}
In words, the sets $\sph_1,\sph_2,\sph_3$ are disjoint planar discs.  One can check that $\K_{S3}$ is representable as a linear image of $\fs^6$.  The other aspects remain the same as in the first illustration, and as we observe from the right subplot of Figure \ref{fig:Exp_F7}, an appropriate choice for $\C$ would be $\fs^6$.

\begin{figure}
	\centering
	\begin{subfigure}{.40\textwidth}
		\centering
		\includegraphics[width=0.8\textwidth]{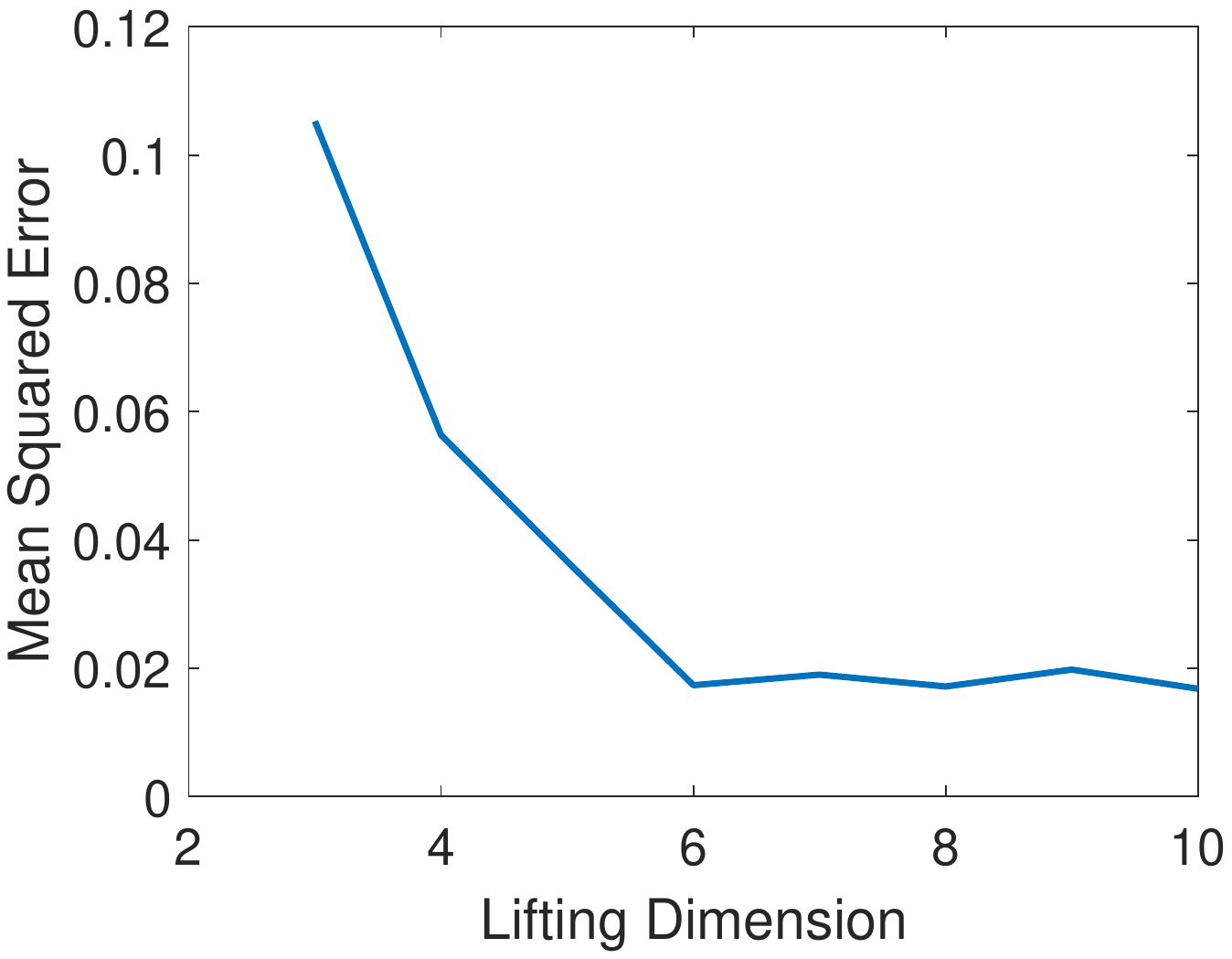}
	\end{subfigure}%
	\begin{subfigure}{.40\textwidth}
		\centering
		\includegraphics[width=0.8\textwidth]{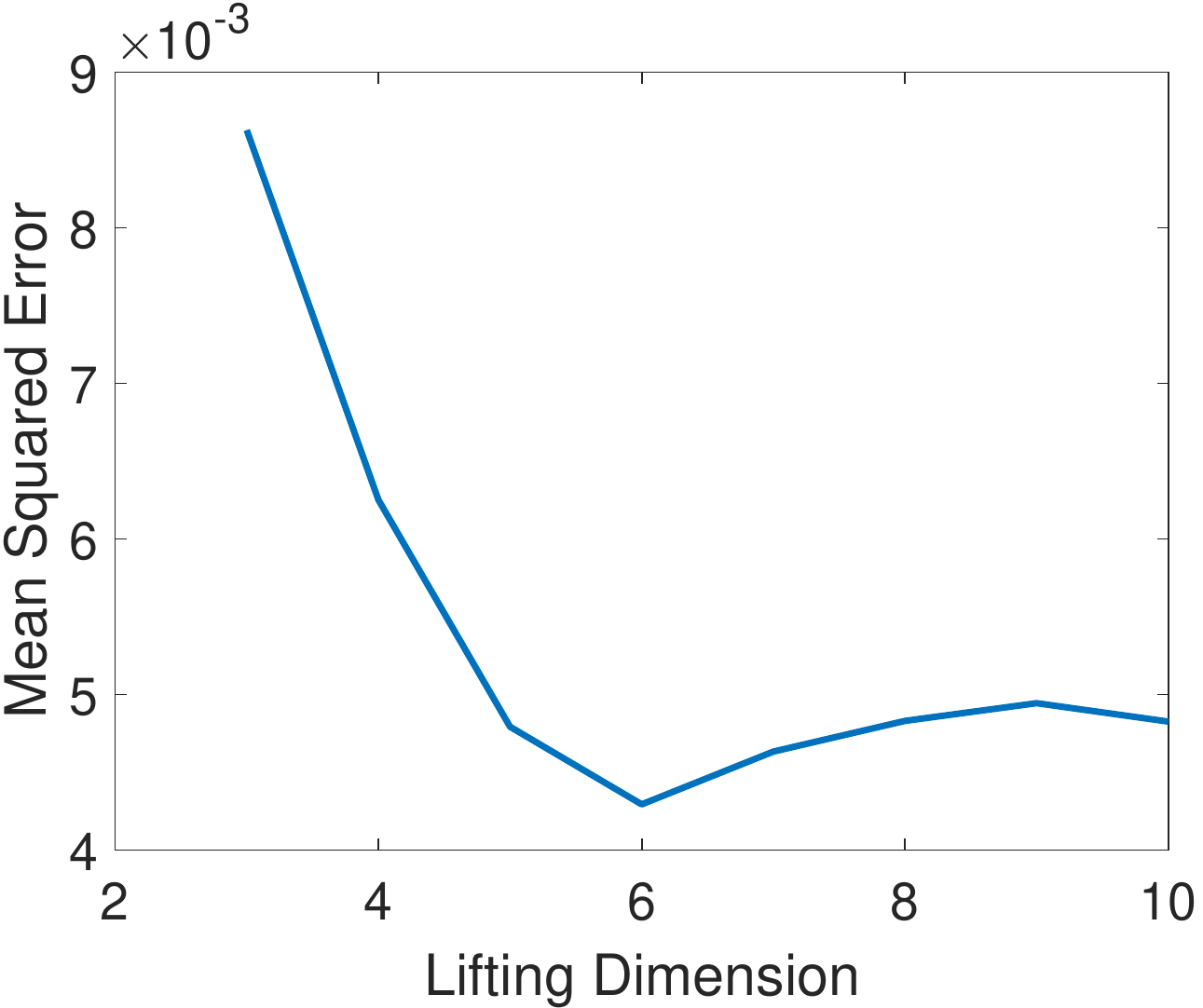}
	\end{subfigure}%
	\caption{Choosing the lifting dimension in a data-driven manner.  The left sub-plot shows the cross validation error of reconstructing the $\ell_1$-ball in $\R^3$ as the projection of $\simp^{q}$ over different choices of $q$, and the right sub-plot shows the same quantity for $\K_{S3} \subset \R^3$ (see accompanying text) as the projection of $\fs^{p}$ over different choices of $p$.}
	\label{fig:Exp_F7}
\end{figure}

\textbf{Approximation power of semidefinite descriptions.}  Theorem \ref{thm:stats_setconvergence} demonstrates that our estimator converges almost surely to the linear image of $\C$ that best approximates the underlying set $K^\star$.  Consequently, a natural question is to understand the quality of the approximation of general convex bodies provided by linear images of the simplex (i.e., polytopes) or the spectraplex.  There is a substantial body of prior work that studies approximations of convex bodies via polytopes (see, for instance \cite{Bronstein08}), but an analogous theory for approximations that are specified as linear images of the spectraplex as well as the more general collection of feasible regions of semidefinite programs is far more limited -- the closest piece of work of which we are aware is a result by Barvinok \cite{Bar:12}.  Progress on these fronts would be useful for understanding the full expressive power of our framework.  More generally, and as noted in \cite{Bar:12}, obtaining even a basic understanding about the approximation power of such descriptions has broad algorithmic implications concerning the use of semidefinite programs to approximate general convex programs.

\textbf{Richer families of tractable convex sets.} A restriction in the development in this paper is that we only consider reconstructions specified as linear images of a \emph{fixed} convex set $\C$; we typically choose $\C$ to be a simplex or a spectraplex, which are given by particular slices of the nonnegative orthant or the cone of positive semidefinite matrices.  As described in the introduction, optimizing over more general affine sections of these cones is likely to be intractable due to the lack of a compact description of the sensitivity of the optimal value of conic optimization problems with respect to perturbations of the affine section.  Consequently, it would be useful to identify broader yet structured families of sets than the ones we have considered in this paper for which such a sensitivity analysis is efficiently characterized.

%


\bibliography{bib_cvxreg}
\appendix

\end{document}